\documentclass[a4paper,10pt]{article}

\usepackage[hidelinks]{hyperref}

\usepackage{amssymb}
\usepackage{amsmath,amsthm,amssymb,latexsym,amsfonts,wrapfig,mathtools,mathbbol,dsfont}
\usepackage[latin1]{inputenc}
\usepackage{math}
\usepackage[english,activeacute]{babel}
\usepackage{graphicx,xcolor}
\usepackage[mathscr]{euscript}
\usepackage{geometry,upgreek,bbm} 
\usepackage[normalem]{ulem}
\usepackage{enumerate}
 \usepackage{mathrsfs}
\usepackage{todonotes}
\providecommand{\sm}[4]{{\Big[\begin{smallmatrix}#1&#2\\#3&#4\end{smallmatrix}\Big]}}

\DeclareSymbolFontAlphabet{\amsmathbb}{AMSb}%

\usepackage{thmtools}
\usepackage{hyperref}
\usepackage{cleveref}

\renewcommand{\norm}[1]{\left\lVert#1\right\rVert}

\def \N{\mathbb N}
\def \Z{\mathbb Z}

\def \R{\mathbb R}
\def \C{\mathbb C}

\def \P{{\mathbb P}}
\def \pa{{\partial}}

\def \O{\mathcal{O}}
\def \T{\mathbb{T}}

\def \E{\mathbb{E}}

\newcommand\independent{\protect\mathpalette{\protect\independenT}{\perp}}
\def\independenT#1#2{\mathrel{\rlap{$#1#2$}\mkern2mu{#1#2}}}

\newcommand{\eps}{\varepsilon}

\newcommand{\Ac}{\mathcal A}

\newcommand{\Cc}{\mathcal C}
\newcommand{\Dc}{\mathcal D}

\newcommand{\Fc}{\mathcal F}
\newcommand{\Gc}{\mathcal{G}}

\newcommand{\Nc}{\mathcal N}

\newcommand{\Es}{\mathscr E}

\renewcommand{\sign}{{\textup{sign}}}

\def\beq{\begin{equation}}   \def\eeq{\end{equation}}
\def\bea{\begin{eqnarray}}  \def\eea{\end{eqnarray}}

\newcommand{\vr}{\varrho}
\newcommand{\bsi}{{\boldsymbol{\upsigma}}}

\newcommand{\sR}{\mathscr{R}}
\newcommand{\sU}{\mathscr{U}}
\renewcommand{\bar}{\overline}

\renewcommand{\vet}[2]{\begin{bmatrix}#1 \\ #2 \end{bmatrix}}

\providecommand{\vect}[2]{{\begin{psmallmatrix}#1\\#2\end{psmallmatrix}}}

\newcommand{\Opbw}[1]{{\rm Op}^{BW}\!\left(#1\right)}

\renewcommand{\be}{\begin{equation}}

\newcommand{\ee}{\end{equation}}

\numberwithin{equation}{section}

\theoremstyle{plain}
\newtheorem{thm}{Theorem}[section]

\newtheorem{lem}[thm]{Lemma}
\newtheorem{prop}[thm]{Proposition}
\newtheorem{cor}[thm]{Corollary}
\theoremstyle{definition}
\newtheorem{defn}[thm]{Definition}

\theoremstyle{remark}
\newtheorem{rk}[thm]{Remark}

\theoremstyle{plain}

\theoremstyle{remark}

\title{ \bf Rogue waves and large deviations \\
for 2D pure gravity deep water waves}

\date{}

\author{M. Berti\footnote{International School for Advanced Studies (SISSA), Via Bonomea 265, 34136, Trieste, Italy. \newline
	\textit{Emails:} \texttt{berti@sissa.it}, \texttt{rgrandei@sissa.it}, \texttt{amaspero@sissa.it}}, \  R. Grande$^*$, \ A. Maspero$^*$, \  G. Staffilani\footnote{	Massachusetts Institute of Technology,
USA \newline 
\textit{Emails:} \texttt{gigliola@math.mit.edu}} }

\begin{document}

	\maketitle
	\begin{abstract}

Rogue waves are extreme ocean events characterized by the sudden formation of anomalously large crests, and remain an important subject of investigation in oceanography and mathematics. 
A central problem is to quantify the probability of their formation under random Gaussian sea initial data.
In this work, we rigorously 
characterize the
tail-probability for the formation of rogue waves of the pure gravity water wave equations in deep water, the most accurate quasilinear PDE modeling waves in open ocean. 
This large deviation result rigorously proves 
various conjectures from the oceanography literature
in the weakly nonlinear regime. 
Moreover, the result holds up to the optimal timescales allowed by deterministic well-posedness theory.
The proof shows that rogue waves most likely arise through ``dispersive focusing", where phase quasi-synchronization produces constructive amplification of the water crest.  
The main difficulty in justifying this mechanism is propagating statistical information over such long timescales, which we overcome by combining normal forms and probabilistic methods.
Unlike prior work, this  
novel approach  does not require approximate solutions to be Gaussian.
Our general method tracks the tail probability of solutions to Hamiltonian PDEs with an integrable normal form and random Gaussian initial data over very long times, even in the absence of (quasi-)invariant measures.

\end{abstract}
\setcounter{tocdepth}{2}	
\tableofcontents

\section{Introduction}
Rogue waves are extreme ocean phenomena characterized by waves of anomalous height.
They are nowadays recognized as serious threats to ships and sea structures: it is estimated that, between 1969 and 1994, 22 super-carriers were lost and 525 fatalities occurred due to collisions with rogue waves in the Pacific and Atlantic Oceans \cite{Kharif0}. 
 The first rogue wave ever recorded was the Draupner wave, which reached a height of 25.6 meters when the significant wave height --defined as four times the standard deviation of the surface elevation-- was at most 12 meters \cite{Draupner}. This event illustrates the textbook definition of a rogue wave: a water wave whose crest-to-trough height exceeds twice the significant wave height. 
 In the past decade, datasets of measured rogue waves have become available, see e.g.~\cite{Christou,Ardhuin}, providing valuable information for statistical analysis.

 \smallskip
 
 Explaining the formation of rogue waves remains a challenging and debated problem in oceanography and physics, see e.g.~\cite{Dysthe,O-mech}.
 Open questions include whether rogue waves can be forecast \cite{Hafner,Breunung,fedele3} and  whether certain sea conditions make their occurrence more likely, see for instance \cite{Toffoli}.
Two prominent theories proposed to explain the generation of rogue waves are {\em nonlinear focusing} and {\em dispersive focusing}.

According to the nonlinear focusing theory, rogue waves form as a late stage of modulational instability, see, for example, \cite{Osborne2000,Dyachenko2005,Zakharov2006} for theoretical analyses and \cite{Chabchoub2011,Chabchoub2012} for experimental evidence in water tanks. Theoretically, this mechanism is described by the one-dimensional focusing semilinear Schr\"odinger equation (NLS), which is a modulational asymptotic model for pure gravity deep water waves \cite{Zakharov1972,Totz2012}. The NLS equation
possesses explicit families of  solutions that exhibit rogue-wave-like features
like the Peregrine soliton \cite{Peregrine1983},
or more complicated breathers constructed via algebraic-integrable techniques, see e.g. \cite{Bertola,Bilman}, and which tend asymptotically to plane waves for both large positive and negative times. 
Such structures have been  observed in wave tank experiments \cite{Chabchoub2011}.

The second theory --dispersive focusing--  predicts that rogue waves form through the superposition of weakly interacting waves whose phases nearly align at a given location, producing constructive interference, which has also been reproduced in laboratory wave tank experiments \cite{Brown2001,Dematteis2019}. 


A common approach to study dispersive focusing is 
to model the water surface profile $\eta(t,x)$ as a superposition of independent Gaussian waves and then compute the probability of forming a large-amplitude crest. Several theoretical and numerical studies (see e.g. \cite{Longuet-Higgins,Janssen2003,Mori2011,O-mech,fedele} and references therein) predict that {\em if} the  free surface profile $\eta(t,x)$ follows a Gaussian distribution $\mathcal{N}_\mathbb{R}(0,\sigma^2)$, then
\begin{equation}\label{prob}
\mathbb{P} \left(2 \sup_{x\in \T}\eta(t,x)>\tH_0\right) = \exp\left(-\frac{\tH_0^2}{8 \sigma^2}\right) \, . 
\end{equation}
Dematteis, Grafke, and Vanden-Eijnden \cite{Dematteis2018} recently developed a novel probabilistic framework  for the study of rogue waves, using large deviation principles, akin to \eqref{prob}, to identify the most likely extreme profiles.
The model was experimentally validated in \cite{Dematteis2019} (and further verified in long wave tanks or through oceanic time series, as referenced in \cite{Hafner,Knobler2022,Gemmrich2022}).

\smallskip 
The goal of this paper is to rigorously justify  the statistic \eqref{prob} 
 for the full {\em  pure gravity water waves equations} in deep water --the most accurate quasilinear nonlocal free boundary  problem modeling the surface motion-- 
 when the initial datum is a  random Gaussian 
 with small  variance $\sigma\ll 1$,  over very long times. We will achieve this through a large deviation principle for PDEs, thereby rigorously  justifying the formal analysis of \cite{Dematteis2018} in the weakly nonlinear setting.
Roughly, our main result  
states the following (see \Cref{thm:main_intro} for a precise version):

\medskip
$\bullet$ {\em 
Let $\eta(t,x)$ be the free water waves surface with a smooth random initial datum $ \eta^\omega(0,x) $  
distributed as   a  Gaussian random variable with zero average and standard deviation $\sigma\ll 1$. Then, 
for any $\sigma\ll \tH_0 \ll 1$ and any time $|t| \leq 
\tH_0^{-3+}$, the probability} 
\begin{equation}\label{main_provv}
     \P \left( 2\sup_{x \in \T} \eta(t,x) > \tH_0  \right) \sim \exp\left( -\frac{\tH_0^2}{8 \sigma^2}\right) \  \quad \mbox{ as } \tH_0 \to 0^+ \, . 
\end{equation}
This statement will follow directly from 
\Cref{thm:main_intro} by setting  $\tH_0 \approx   \eps^{1-\delta}$, where  $0<\delta \ll 1$ 
and by considering the variance of
 $\eta^\omega(0,x)$ to be  approximately $ \sigma^2 \approx \eps^2$, see  the first comment below \Cref{thm:main_intro}.
Note that a wave $\eta(t,x)$ with crest $ \tH_0/ 2  $ is 
considered a rogue wave because,  
the average initial value
has much smaller size $\sigma \ll \tH_0$. 
The condition $\tH_0\ll 1$ corresponds to the weakly nonlinear regime of the water waves dynamics. The time interval  $|t| \leq  
\tH_0^{-3+}$ 
is the optimal one for which the existence of water wave solutions is guaranteed by the deterministic well-posedness theory 
\cite{BFP,DIP,Wu3},  starting from 
sufficiently small and smooth initial data of size 
$ \tH_0 $ in some 
high  Sobolev norm.  
Deterministic energy estimates imply that 
such  Sobolev norm  
remains of size $ \cO (\tH_0) $ over  
the entire time interval of existence. 
This does not contradict the statement \eqref{main_provv} 
that the 
$ L^\infty (\T) $ norm of the water waves profile 
increases with time.
The timescale $|t| \leq   \tH_0^{-3+}$ 
is likely optimal because of the quintic wave 
resonant interactions identified by
 Craig-Worfolk \cite{CW} and Dyachenko-Lvov-Zakharov \cite{DZ},  
which break the Birkhoff  integrability
of the water waves equation, 
and may  lead to  unstable dynamics.  
We emphasize that our goal is not to track the time evolution of the $L^{\infty}$-norm along a single water-wave trajectory which, 
heuristically, it is expected to vary of only a $ \cO(1) $-factor (in $\eps$) over times of validity of the integrable normal form. Rather, following the viewpoint of the physics and oceanography communities, we aim to quantify the probability of extreme events: in a sea where the typical wave height is $\sigma$, what is the likelihood of encountering a wave of height $\tH_0\gg \sigma$?

\smallskip

The proof of \eqref{main_provv} is highly nontrivial. The main difficulty is that, even if the initial free surface $\eta^\omega(0,x)$ is Gaussian (together with the initial velocity potential $\psi^\omega(0,x)$ in \eqref{eq:init_data}), the surface profile $\eta(t,x)$ 
does {\it not} necessarily remain Gaussian at later times. In fact, no such property is expected, since $\eta(t,x)$ depends in a highly nonlinear way on the initial data. Tracking the evolution of a probability measure under a Hamiltonian PDE is a well-known challenge in PDEs and probability. In particular, invariant Gibbs measures have only been rigorously constructed for some semilinear PDEs (see Bourgain \cite{Bourgain} and the recent breakthrough papers by Deng-Nahmod-Yue \cite{DNY,DNY2}, and the references therein), while the water waves equations are quasilinear. The theory of quasi-invariant measures offers an alternative approach based on studying the absolute continuity of the probability measure with respect to the initial Gaussian measure \cite{Tzvetkov, Oh-Tzvetkov}. 
These techniques can sometimes yield upper bounds on the tails of the probability measure \cite{GLT,PTV}, but they are not developed for quasilinear systems and do not yet yield sharp upper/lower bounds such as \eqref{main_provv}. 
Our main result shows that it is possible to propagate the precise tail of the initial distribution --which is Gaussian--, for very long nonlinear times, even in the absence of quasi-invariant/invariant measures.

Considerable progress has also been achieved in the context of wave turbulence theory, where the probability distribution of the solution to semilinear dispersive equations on large tori is tracked via its moments, which are proved to satisfy kinetic equations, see for instance \cite{BGHS,DH,DH2} in the random data and \cite{GH,ST} in the stochastic setting. These results are based on a Picard iteration scheme which is only convergent (without loss of derivatives) in the context of semilinear PDEs. Progress in the wave turbulence theory of the water waves system on large tori has been recently obtained by Deng-Ionescu-Pusateri in \cite{DIP,DIP2} -- while they prove that, up to a set of exponentially small probability, the $L^{\infty}$-norm remains small over long timescales, our work explores the complementary event. 
The approach of \cite{DIP2} could in principle yield upper bounds on the tail distributions, but not with the sharp constants provided by our large deviation principle -- thanks also to the corresponding sharp lower bounds. Such sharp constants are essential to establish the precise conjectures in the physics community (e.g. \eqref{main_provv} and \Cref{thm:df}).
We further comment on this work in point \ref{comp_WT} under \Cref{thm:main_intro} below.

\smallskip

In this article, we address the issue of the evolution of the  tails of the probability measure under the water waves flow  using different methods depending on the timescales. 
\begin{itemize}
\item 
For very long times $ |t | \lesssim  \eps^{-5/2+}$, we construct a sufficiently accurate approximate solution of   
the water waves flow which 
propagates the Gaussianity 
with time-dependent phase corrections, cf. 
\Cref{sec:52}.
\end{itemize}
This construction exploits the famous integrability of the pure gravity water waves Birkhoff normal form proved formally by
Zakharov-Dyachenko \cite{ZakD}, and rigorously in 
Berti-Feola-Pusateri \cite{BFP}. 
This strategy is inspired by the approach Garrido, Grande, Kurianski, and Staffilani \cite{GGKS} developed 
for the cubic 1d-NLS, with various key differences 
commented in point \ref{comp44} below Theorem \ref{thm:main_intro}.

\smallskip

The main novelty lies in the 
dynamical and 
probabilistic 
mechanism 
developed to reach the timescale $|t|\lesssim \varepsilon^{-3+}$,
which is very different from the approach in  \cite{GGKS} for NLS. 
\begin{itemize}
\item 
Up to optimal times  
$ |t| \lesssim \eps^{-3+}$, we do not propagate Gaussianity of the Fourier coefficients of some approximate solution,  but  we are still able to propagate the Gaussianity of the tails of the $L^{\infty}$-norm of the solution, and compute the probability that the nonlinear phases quasi-synchronize at a given time. 
\end{itemize}
In particular, we do not attempt to track the probability distribution of the approximate solution's nonlinear phases: this seems an 
impossible task given the complex, nonlinear relation between the initial datum and the phases' time-evolution.
Instead, we exploit this dependence to compute the probability that the initial phases will lead to quasi-synchronization (at any  time), creating a constructive superposition of waves that produces a rogue wave, cf. \Cref{sec:LDPIII}. 
We achieve this in two steps, as detailed in Section  \ref{sec:ideas}.  
Exact synchronization is 
accomplished using a random version of the celebrated Brouwer fixed point theorem.
However, this represents a zero-probability event. 
Instead quasi-synchronization, on a set of exponentially small probability, is achieved by proving quantitative Lipschitz estimates for the deterministic flow. 
A final key step is to prove a weak version of statistical independence between the random amplitudes of the initial Fourier coefficients and the quasi-synchronization event (factorization property).
All in all, the formation of a rogue wave requires the quasi-synchronization of the phases of the first 
$\tN_\eps$   Fourier modes. Here,
$\tN_\eps \rightarrow \infty$, as slowly as desired, as $\eps\rightarrow 0$.
Naturally this occurs with
an exponentially small probability. 

In this result dispersive focusing 
is identified as the dominant formation mechanism of rogue waves up to the timescale $|t|\lesssim \eps^{-3+}$, establishing the lower bound 
in \eqref{prob}.
\Cref{thm:df} complements this result  by showing  a large deviation principle for the quasi-synchronization of arbitrarily finitely many phases of the full water wave solution, not just the approximate one. 

This method provides a new paradigm  to quantifying  the tail probability of solutions of  
Hamiltonian PDEs with an integrable normal form and random Gaussian initial data, especially  in the context of  quasilinear PDEs for which invariant or quasi-invariant measures are not available (out-of-equilibrium). 

\subsection{Main results}

To precisely state the main results, we first present the water waves equations in detail. 
We consider the Euler equations of hydrodynamics for a 2-dimensional 
incompressible and irrotational fluid  under the action of 
{\it pure gravity}.
The fluid 
occupies the time dependent region
\begin{equation}
\label{domain}
\Dc_{\eta} := \big\{ (x,y)\in \T\times \R \ : \   y<\eta(t,x) \big\} \, , 
\quad \T := \T_x :=\R/ (2\pi\Z) \, ,
\end{equation} 
with infinite  depth   and 
 space periodic boundary conditions.
The unknowns of the problem are the free surface  $ y = \eta (t, x)$
of the domain $\Dc_{\eta} $ and the 
divergence free  velocity field $\begin{pmatrix} 
u(t,x,y) \\ 
v(t,x,y)
\end{pmatrix} $.
The velocity field is irrotational and therefore is 
 the gradient of a harmonic function $\Phi (t,x,y) $, called the generalized velocity potential. 
 We study the water waves equations 
 in the Hamiltonian
 formulation developed by  Zakharov \cite{Zak1}
 and Craig-Sulem \cite{CS}. 
Denoting $ \psi (t,x) := \Phi (t,x, \eta(t,x)) $ 
 the evaluation of the generalized velocity potential at the free interface, 
the velocity potential $ \Phi $ 
is recovered 
by solving the elliptic problem
\begin{equation}
\label{dir}
\Delta \Phi = 0  \ \mbox{ in } \Dc_{\eta} \, , \quad
\Phi = \psi \  \mbox{ at } y = \eta(t,x) \, , \quad
\Phi_y \to  0  \  \mbox{ as } y \to  -\infty \, .
\end{equation}
Inside $ \cD_{\eta}$ the velocity field of the fluid evolves according to
the Euler equations and the 
dynamics  is  
determined by 
two boundary conditions at the free surface. 
We impose 
that the fluid particles at the  time dependent 
profile  remain on it along the evolution
(kinematic boundary condition), and that
the pressure of the fluid 
 is equal to the constant 
atmospheric pressure  at the free surface (dynamic boundary condition). Thus the  
time evolution of the fluid is determined by the following 
system of quasi-linear equations 
\begin{equation}
\label{ww}
\begin{cases}
\pa_t \eta = G(\eta)\psi  \\
\displaystyle{\pa_t \psi = - g\eta  - \frac{\psi_x^2}{2} + 
	\frac{(  \eta_x \psi_x + G(\eta)\psi)^2}{2(1+\eta_x^2)} }  
\end{cases}
\end{equation}
where $ g > 0 $ is the gravity constant which, 
 without loss of generality, we set equal to $ g = 1 $,  and 
$G(\eta)$ is the  Dirichlet-Neumann operator  
\begin{equation}
\label{DN}
G(\eta)\psi :=  \sqrt{1+\eta_x^2} \, (\partial_{\vec n} \Phi )\vert_{y = \eta(t,x)} = (- \Phi_x \eta_x + \Phi_y)\vert_{y = \eta(t,x)} \, , 
\end{equation}
here $ \vec n $ denotes the outer unit normal of $ \cD_\eta $.
The name refers to the fact that  $G(\eta) $ 
transforms the Dirichlet datum  $ \psi$ 
of the  velocity potential $ \Phi $ into its Neumann derivative $G(\eta)  \psi $ at the free surface. 
Since $ G(\eta) \psi  $ has zero 
space average also 
$$ 
\langle \eta \rangle := 
\frac{1}{2\pi} \int_\T \eta(x) \, \di x 
$$ 
is a first integral of \eqref{ww} 
and  
we may  consider $ \eta (t, \cdot)  $  to evolve within a Sobolev space
\begin{equation}\label{hs}
H^s_0(\T,\R)  := \left\lbrace \eta(x) = \frac{1}{\sqrt{\pi}} \sum_{n \in \N} \big ( \alpha_n \cos(nx) + \beta_n \sin(nx) \big ) \ \colon
\quad 
\norm{\eta}_{H^s_0}^2:= \sum_{n \in \N} n^{2s} (\alpha_n^2 + \beta_n^2)  <
+ \infty \right\rbrace
\end{equation}
of zero average functions. 
In addition, the water waves vector field $ X(\eta,\psi ) = (X^{(\eta)}, X^{(\psi)})  $ 
defined by the right hand of \eqref{ww} depends only on 
$ \eta $ and  the mean free component 
$  \psi - \frac{1}{2 \pi}\int_\T \psi \, \di x  $. Thus 
we may consider $ \psi $ to evolve within a
homogeneous Sobolev space $ \dot H^s(\T,\R)$   obtained identifying  
functions which differ  by a constant.

As observed by Zakharov \cite{Zak1} 
the water waves equations \eqref{ww} are the Hamiltonian system 
\begin{equation}\label{ham_ww}
	\pa_t \eta = \nabla_{\psi} H(\eta,\psi)\,, \quad \pa_t \psi = -\nabla_{\eta} H(\eta,\psi)\,,
\end{equation}
where $\nabla_\eta, \nabla_\psi $ denote the $L^2$-gradient, and the Hamiltonian
\begin{equation}\label{ham1}
	H(\eta,\psi) = \frac12 \int_{\T} \Big( \psi \,G(\eta) \psi +  \eta^2 \Big) \wrt x 
\end{equation}
is the sum of the kinetic energy of the fluid 
plus the potential energy due to gravity.  

In addition to being  Hamiltonian,
 the water wave system \eqref{ww} 
is time-reversible with respect to the involution
$$
\rho\vet{\eta(x)}{\psi(x)} := \vet{\eta(-x)}{-\psi(-x)} \, , \quad \text{i.e.} \  X \circ \rho = - \rho \circ X \, ,  
$$
where $ X  
= (X^{(\eta)}, X^{(\psi)}) = (\nabla_\psi H, - \nabla_\eta H )$
is the water waves vector field. 
Equivalently, the solution map  $\Psi^t $  of \eqref{ww} (whenever defined)
satisfies $ \Psi^{-t} = \rho \circ \Psi^t \circ \rho $ (the flow $ \Psi^t (\eta_0, \psi_0) $ is defined as the unique solution of
\eqref{ww} with initial condition 
$ (\eta_0, \psi_0) $). In other words the solutions of \eqref{ww} for $ t < 0 $ are obtained from solutions for positive times 
by conjugation  under $ \rho $.    
Furthermore, 
the water waves  system 
\eqref{ww} 
is translation invariant, namely 
\be\label{X.tra0}
X \circ \tau_\varsigma = \tau_\varsigma \circ X 
\, , \quad \forall \varsigma \in \R \, ,
\qquad \text{where} \qquad 
\tau_\varsigma \colon f(x) \mapsto f(x + \varsigma) 
\ee
is the translation operator. 
This implies, by the Noether theorem, 
the existence of the conserved quantity
$ \int_{\T} \eta_x(x) \psi (x) \, \di x $, called momentum.  

\smallskip

Let us also discuss  regularity  
properties of the water waves system \eqref{ww}, which we shall use. 
Several results about the analyticity of the Dirichlet-Neumann  operator $ G(\eta) $ acting
between Sobolev spaces of functions, with respect to the free 
boundary $ \eta (x) $, have been
proved in \cite{CN,LannesLivre}, 
to which  we refer for an extended
bibliography. These results mainly concern the case of localized profiles and velocity potentials in $ \R^d $. 
In this paper we use the analyticity 
result  in \cite{BMV_dn} valid 
under periodic boundary conditions in
 infinite depth.
Let $B^\ts (r) $
denote the open ball in $ H^\ts  (\T,\R) $ of center $ 0 $ and radius $ r > 0 $.
For  $ \ts > \frac72 $, $ \ts + \frac12 \in \N $, there is $ r > 0 $ small enough, such that  the Dirichlet-Neumann operator mapping    
\be\label{eq:DNan}
\eta \mapsto G(\eta) \, , \quad 
H^\ts  (\T, \R) \cap B^{\frac72} 
(r) \to {\cal L}(H^\ts (\T, \R),H^{\ts-1} (\T,\R)) \, , \quad \text{is analytic} \,  .
\ee
Actually the result in \cite{BMV_dn} 
also holds 
when $ \eta $ is analytic, and $ G(\eta) $ acts between spaces of  analytic functions.
In view of  \eqref{eq:DNan} and the algebra properties of Sobolev spaces, for $ \ts > \frac72 $, 
the water waves vector field   
in \eqref{ww} 
\be\label{wwana}
(\eta, \psi) \mapsto X(\eta, \psi) 
= (X^{(\eta)}, X^{(\psi)}) \, , \quad 
B^{\ts}(r)\times  H^{\ts}(\T,\R) \to 
H^{\ts-1}_0(\T,\R)\times {\dot H}^{\ts-1}(\T,\R) \, , 
\quad \text{is analytic} \, , 
\ee
possibly with a smaller $ r > 0 $.

\smallskip

The local well-posedness theory for these equations is a challenging problem. The major difficulties stem from the quasilinear and nonlocal nature of the vector field in \eqref{ww}, and a ``hidden" hyperbolic structure within the equations.
The first existence results 
with initial data in Sobolev spaces 
were  established by S. Wu  
in the pioneering work \cite{Wu0}.  
In the Eulerian formulation \eqref{ww}, local well posedness  was later achieved by Lannes 
\cite{Lannes} and Alazard-Burq-Zuily \cite{ABZ1}, 
 through the introduction of the 
 ``good unknown" of Alinhac
 $ \upomega $, instead of $ \psi$, and 
a paralinearization of the Dirichlet-Neumann operator.
 This change of variables 
 reveals the hyperbolic nature of the water 
 waves equations. 
The paradifferential approach \cite{ABZ1}
enables to prove the existence of water waves solutions in  
 $$
 (\eta,\psi) \in C^0\left([0,t], H^{\ts}_0(\T, \R)\times \dot H^\ts(\T, \R)\right) \, , \quad {\ts} > \tfrac32 \, , 
 $$
under weaker regularity assumptions on the Cauchy data than in previous works.
We shall present 
the deterministic water waves 
theory in \Cref{sec:normalform_results}. 
These results, 
when specialized to initial data of size $ \varepsilon $,   
 imply  a time of existence  larger than $ c \varepsilon^{-1} $. 
Actually the water waves  exist 
up to times $ \varepsilon^{-3}$,
provided the initial datum is sufficiently regular, as first proved by the paradifferential Birkhoff normal form in \cite{BFP}, see also \cite{Wu3,DIP}. 
These works ultimately rely on the celebrated integrability 
of the pure gravity water waves equations in deep water up to quartic vector fields, 
as formally guessed by 
Zakharov and Dyachenko \cite{ZakD}.
The  approach in  \cite{BFP}   
 to  rigorously   justify the formal computations in  \cite{ZakD} 
is  based on a uniqueness argument of the normal form: 
in view of non-resonance properties satisfied by the linear frequencies, 
 the ``paradifferential" normal and the formal one 
 in  \cite{ZakD}
  have to coincide up to cubic degree.
Other deterministic  long time existence  results  
for periodic water waves are in 
\cite{Wu3,IP,HarIT,BFF,BD,BMM2}
and references therein.

\paragraph{\bf Random initial data.}
In this paper we  study the evolution of the water waves equations \eqref{ww} starting from 
small amplitude {\it random} initial data  
\begin{equation}\label{eq:init_data}
\begin{aligned}
&\begin{cases}
\eta_0^\omega (x)= \frac{\varepsilon}{\sqrt\pi}\, \sum_{n\in\N} c_n \left( \alpha_n^\omega \cos (nx) + \beta_n^\omega \sin (nx)\right) \\
 \psi_0^\omega (x)=\frac{\varepsilon}{\sqrt\pi}\, \sum_{n\in\N} d_n \left( \gamma_n^\omega \cos (nx) + \delta_n^\omega \sin (nx)\right)
\end{cases}\\
&\mbox{ where } (\alpha^\omega_n)_{n\in\N}, 
(\beta_n^\omega)_{n\in\N},
(\gamma_n^\omega)_{n\in\N}, 
(\delta_n^\omega)_{n\in\N} 
\mbox{ are i.i.d.  real Gaussian r.v.   } \sim  
\cN_\R\left(0, 1\right) \, , 
\end{aligned}
\end{equation}
in the probability space $(\mathbb{\Omega}, \cF, \P)$, and
where the deterministic  
coefficients $c_n, d_n$  are  
\begin{equation}
\label{ckdk}
c_j :=   e^{- |j| \tb} \, , \ j \in \Z \setminus \{0\} \, ,  \quad \mbox{for some } \tb > 0 \, ,  
\quad d_n := n^{-1/2} c_n = 
n^{-\frac12} e^{- n \tb} \ \ \  \,  \forall n \in \N   \, .
\end{equation}
The specific choice of $c_j $ in \eqref{ckdk} is introduced for simplicity  to 
guarantee that $\eta_0^{\omega}$ belongs to any $   H_0^s(\T,\R)$ almost surely for any $ s > 0 $ 
with $s$ sufficiently large so that the deterministic normal form theory is available.

For any  $  x\in\T $, $\eta_0^{\omega}(x)$ is a Gaussian random variable {$\sim \mathcal{N}_{\R}(0, \varepsilon^2 \pi^{-1} \sum_{n\in\N} c_n^2)$, 
while $\psi_0^{\omega}(x)$ is a Gaussian random variable
$\sim \mathcal{N}_{\R}(0, \varepsilon^2\pi^{-1} \sum_{n\in\N} d_n^2)$.}
Indeed, $ \E \eta_0^\omega (x) = \E \psi_0^\omega (x)=0 $, which means that the average initial profile of the fluid is flat and the ocean is, on average, initially at rest.
 Furthermore, the independence of the random Gaussian variables implies that 
 the variances of the initial wave 
 surface and velocity potential are 
\be\label{eta0psi0Sobo}
\E (\eta_0^\omega (x))^2  = \frac{\varepsilon^2 }{\pi} \sum_{n\in\N} c_n^2 
\, , 
\qquad 
\E (\psi_0^\omega(x))^2  = \frac{\varepsilon^2 }{\pi}\sum_{n\in\N} d_n^2 
\, , \qquad \forall x\in\T \, ,
\ee
so the  typical size of the 
initial data is  $\O(\varepsilon)$.  It is convenient to define 
 \be\label{upsigma}
 \bsi^2   := \frac{1 }{\pi} \sum_{n\in\N} c_n^2 = 
 \frac{1 }{2\pi} \sum_{j \in \Z \setminus \{0\} } c_j^2 \ , 
 \qquad \mbox{ so that } 
 \eta_0^\omega(x) \sim 
 \mathcal{N}_{\R}(0, \varepsilon^2 \bsi^2) \, . 
 \ee
Furthermore, 
in view of the exponential decay of the $ c_n $ in \eqref{ckdk} and the assumption 
$ d_n =  n^{-\frac12} c_n $, 
the initial data belong almost surely 
to any Sobolev space and
\[
\E \norm{\eta_0^\omega}_{H_0^{s}}^2  =\varepsilon^2 \sum_{n\in\N}  c_n^2 \, n^{2s} =
\varepsilon^2 \sum_{n\in\N} d_n^2 n^{2s+1} = 
\E \norm{\psi_0^\omega}_{\dot H^{s+\frac12}}^2 
\, , \quad \text{for any} \  s \geq 0  \, .
\]
{\bf Linear water waves.} 
 The linearized water waves system 
 \eqref{ww} at the equilibrium $(\eta, \psi) = (0,0)$ is 
\begin{equation}
\label{lin.ww1}
\begin{cases}
\partial_t \eta & = |D| \psi \\
\partial_t \psi & = - \eta  \, ,
\end{cases} 
\end{equation}
where the Dirichlet-Neumann operator $ G(0) $
at the flat surface $\eta = 0$ is
 the Fourier multiplier $  G(0) = |D| $.    
In the linear complex variable 
\begin{equation}\label{zeta0p}
\zeta:=  \frac{1}{\sqrt{2}}|D|^{-\frac{1}{4}}\eta+\frac{\im}{\sqrt{2}}|D|^{\frac{1}{4}}\psi  \, , 
\quad  \quad 
\begin{pmatrix}
\eta \\ 
\psi 
\end{pmatrix}
=
 \begin{pmatrix}
\sqrt{2} \, \Re \, |D|^{\frac14}\zeta \\ 
\sqrt{2} \, \Im \, |D|^{-\frac14}\zeta 
\end{pmatrix} \, , 
\end{equation}
the real system \eqref{lin.ww1} 
is transformed into the equation
\begin{eqnarray}\label{eq:lin00_ww_C}
\pa_t \zeta = - \im \Omega(D) \zeta \, , 
 \qquad \Omega(D):= |D|^\frac12  \, ,
\end{eqnarray}
which, written in Fourier, amounts to 
infinitely many 
decoupled harmonic oscillators 
\be\label{harmonic}
\dot \zeta_j = - \im \, \Omega_j  \zeta_j \, , \qquad \Omega_j :=  |j|^\frac12 \, ,  \qquad \forall  j \in \Z \setminus \{0\} \, . 
\ee 
In view of \eqref{zeta0p},  recalling the definitions of the Sobolev norms \eqref{hs} and \eqref{usobo}, we have  
\be\label{equivalnorms}
\| (\eta, \psi) \|_{H^s_0 \times \dot H^{s+\frac12}}^2  := 
\| \eta \|_{H^s_0}^2 +
\| \psi \|_{\dot H^{s+\frac12}}^2 =
2 \| \zeta \|_{H^{s+\frac14}}^2 \, . 
\ee
The solution of  \eqref{harmonic} with initial datum 
$ \zeta_{0j} $ is the traveling plane wave 
$ \zeta_{0j} e^{- \im ( \Omega_j t - jx) } $. 
The assumption 
$ c_n = n^\frac12 d_n $ in \eqref{ckdk}
implies   that 
the complex initial datum 
\be\label{zeta0}
\zeta_0^\omega= 
\frac{1}{\sqrt{2}}|D|^{-\frac{1}{4}}\eta_0^\omega+\frac{\im}{\sqrt{2}}
|D|^{\frac{1}{4}}\psi_0^\omega =  
\frac{1}{\sqrt{2\pi}}
\sum_{j \in \Z \setminus \{0\}} \zeta_{0j}^\omega  e^{\im j x} \, , 
\ee
has Fourier coefficients  
$ \zeta_{0j}^\omega $ that 
are {\it independent complex Gaussians} random variables
$\cN_\C\big(0, \eps^2 \, c_j^2 \, |j|^{-\frac12}\big)$, for any $ j \in \Z \setminus \{0\} $,  
as we prove in  \Cref{lem:zeta0.d}.
The solution of 
the complex linear system  \eqref{eq:lin00_ww_C} with initial datum $\zeta_0^\omega (x) $ defined  in \eqref{zeta0} is  
\be\label{eq:Gaus}
    \zeta_L(t,x):=\frac{1}{\sqrt{2\pi}}\sum_{j\in \Z\setminus\{0\}}  \zeta_{0j}^\omega \,  e^{-\im\,\left( \Omega_j t- j\,x \right)} \, . 
\ee
Inverting  formula \eqref{zeta0p}, we find that 
the solution of the linear water waves system 
\eqref{lin.ww1} with 
initial datum $(\eta_0, \psi_0)$  is 
\begin{equation}\label{linz200}
        \eta_L(t,x) = \sqrt{2} \, \Re \, |D|^{\frac14} \zeta_L(t,x)  \, , \quad 
\psi_L(t,x) = \sqrt{2} \, \Im\, |D|^{-\frac14} \zeta_L(t,x) \, .
\end{equation}
The 
linear dynamics \eqref{linz200}  preserves the Gaussianity of the initial data, unlike the nonlinear water waves flow. 

\paragraph{Main results.}
In this paper we are interested 
in understanding the set of 
random initial data $ (\eta^\omega_0 (x), 
\psi^\omega_0(x) )  $ 
in \eqref{eq:init_data} for which the deterministic water waves evolution 
develops rogue waves at times $t > 0  $, namely the  wave profile 
$\eta (t,x) $ 
--which
has a typical initial amplitude of size $ \varepsilon $-- 
is larger than a  threshold $\lambda_0 \varepsilon^{1-\delta } $, $ \delta \in (0,1) $. 
The main 
Theorem \ref{thm:main_intro} below 
estimates the probability of 
the event \eqref{eventr} that the water waves solution $(\eta,\psi)$ to \eqref{ww} with initial data \eqref{eq:init_data} exists in 
$C^0([0,t], H^\ts_0(\T, \R)\times \dot H^\ts(\T, \R))$, $\ts > \frac32$,
up to a time $t $ and  develops a rogue wave
\begin{equation}\label{eventr}
\mathfrak{R}^t(\lambda):=
\Big\{ \omega \in \mathbb{\Omega}  \mid \exists \,  \eta \in C^0([0,t], H^\ts_0) \  \mbox{and}\ \sup_{x\in\T} \eta (t,x) \geq \lambda \Big\} \, .
\end{equation}
The physically more interesting case is when 
$ \delta $ is close to $  1 $, as in this case 
the  rogue wave crest profile is bigger. 
Note also 
that the sign of the wave profile $\eta (t,x) $ 
in \eqref{eventr} is important because we  are specifically interested  in the 
possible formation of high crests,
not deep troughs.
This is why we do not consider in \eqref{eventr} the absolute value  of $ \eta (t,x) $. 

The first  main result of this paper is the following large deviation principle for the formation of rogue waves.

\begin{thm}[\bf Rogue wave formation]\label{thm:main_intro}
For any $\lambda_0 >0$, $\delta \in (0,1)$ and $0<\kappa \ll 1-\delta$, the free surface $\eta (t,x) $ of the water waves solution of  \eqref{ww} with random initial datum \eqref{eq:init_data} satisfies the large deviation principle 
\begin{equation}\label{eq:LDP}
\lim_{\varepsilon\rightarrow 0^{+}} \varepsilon^{2\delta}\, \log \P \left(
\mathfrak{R}^t(
\lambda_0 \varepsilon^{1-\delta})
\right) = -\frac{\lambda_0^2}{
2\bsi^2
}    \, , 
\qquad
\mathfrak{R}^t(\lambda) \mbox{  in } \eqref{eventr} , \quad 
\bsi \mbox{ in } \eqref{upsigma}
\, , 
\end{equation}
for any time $t$ satisfying 
\begin{itemize}
    \item[(i)]  $|t| \leq \eps^{-\frac52 (1-\delta) + \kappa}$; 
    \item[(ii)]   $|t| \leq \eps^{-3 (1-\delta) + \kappa}$  provided $\delta \in (\frac35,1)$.
\end{itemize}
\end{thm}

\Cref{thm:main_intro} is a corollary of \Cref{thm:main} stated below.
Let us make some comments.
\begin{enumerate}
    \item {\sc Proof of the large deviation principle  \eqref{main_provv}.} 
    \Cref{thm:main_intro} guarantees 
        the existence  of initial data, 
        with exponentially small  probability, which
        lead to a rogue wave at time $t$. Specifically,  the initial data have average size $\cO(\eps)$, 
    as shown in \eqref{eta0psi0Sobo}, while the resulting 
    rogue wave crest at time $t$ has significantly 
     larger size $\cO(\eps^{1-\delta})$.
 The  statement \eqref{main_provv} follows immediately   by setting  $\tH_0 = 2 \lambda_0 \eps^{1-\delta}$ and 
 noting that the variance of $\eta_0^\omega(x) $ is ${\eps^2}\,\bsi^2 = \frac{1}
 {2\pi} \eps^2 \sum_{j \in \Z\setminus\{0\}} c_j^2$, see \eqref{upsigma}.

    \item{\sc Dispersive focusing.} As part of the proof of \Cref{thm:main_intro},  we construct in \Cref{sec:LDPIII} an explicit family of rogue-wave approximate solutions
    whose phases synchronize
    at time $t$, creating 
    a constructive superposition, see \Cref{cor:phase_sync}. 
    Synchronization occurs across 
    $\tN :=\tN(\varepsilon)$ phases, a number that is finite but grows polynomially with $ \eps^{-1} $, 
   and involves the lowest $2\tN$ Fourier modes of the truncated approximate solution. The contribution of the remaining high Fourier modes towards the rogue wave can be shown to be negligible, see \Cref{cor:phase_sync} for details.
    
    \item {\sc No growth of Sobolev norms.} The rogue waves constructed in 
    this work do not feature any growth of Sobolev $H^s$-norms. Instead, their $H^s$-norms remain almost 
    constant over long timescales and are of size $\O (\varepsilon^{1-\delta})$, as  a consequence 
    of  the energy estimates in \Cref{thm:control_hs}.  
This does not contradict the dispersive focusing mechanism for the $ L^\infty $-norm discussed  in the previous item.
    
    \item{\sc Timescale $|t|\leq \eps^{-3(1-\delta)+\kappa}$ with $ \kappa \ll 1 - \delta $.}
    This timescale restriction 
    guarantees the deterministic well posedness     of the water waves. 
    Since rogue waves in $\mathfrak{R}^t(
\lambda_0 \varepsilon^{1-\delta})$ have size $\O (\varepsilon^{1-\delta})$ in a Sobolev space $H^s(\T) $, the deterministic theory \cite{BFP}  ensures the
existence of water waves solutions up to times 
$|t| \lesssim \eps^{-3(1-\delta)}$ (which we shrink to the slightly smaller $|t| \leq \eps^{-3(1-\delta) + \kappa}$). 
This timescale is likely optimal due to  the quintic wave resonant interactions identified in 
\cite{CW,Dyachenko1995}.  
Such resonances could increase the likelihood of rogue wave formation over time, a phenomenon 
recently demonstrated for 
the 1d semilinear beating NLS in \cite{Grande}. The timescale $\eps^{-3(1-\delta)+}$ in $(ii)$ is clearly much longer than $\eps^{-\frac52(1-\delta)+}$ in $(i)$ for any fixed $\delta\in (0,1)$. The physically more interesting case is when 
$ \delta $ is close to $  1 $, as rogue waves are larger.

    \item {\sc Comparison with \cite{GGKS}.}\label{comp44}  We improve the results in \cite{GGKS} in several ways. First, Theorem \ref{thm:main_intro} concerns the full water waves equations in deep water,
    extending much beyond the semilinear NLS envelope approximation.
    The quasilinear nature of the water waves vector field requires a paradifferential normal form to  construct  approximate solutions.
Another key difference is the 
non-diagonal nature of the Zakharov-Dyachenko normal form
\eqref{theoBirHfull}, unlike the NLS normal form. This significantly complicates the probabilistic argument to prove the Gaussianity of the first approximate solution (see Lemma \ref{lem:gauind}). 
Furthermore, while the NLS result in  
\cite{GGKS} is valid only up to times $ |t| \lesssim \varepsilon^{-2(1-\delta)} $, 
a more  refined energy estimate (Lemma \ref{thm:approx}) is required to extend this result to longer times  $ |t| \lesssim \varepsilon^{- \frac52 (1-\delta) +\kappa } $. 
In any case, the main novelty in this work lies in the new approach detailed in Section \ref{sec:LDPIII}, which employs a non-Gaussian and non-explicit approximation of the water wave dynamics. This framework, based on a random fixed point theorem and phase quasi-synchronization, is  detailed in Section
\ref{sec:ideas} below.
    \item \label{comp_WT} {\sc Comparison with \cite{DIP,DIP2}.}  Deng-Ionescu-Pusateri prove existence of solutions of the water waves 
    equations \eqref{ww} with random initial data, in the weakly nonlinear regime, for \emph{large tori},
    well beyond the deterministic results. A key part of their work is to propagate $\cO(\varepsilon^{1-} )$ bounds for the $L^{\infty}$-norm of the solution with much larger $H^s$-norm, which they prove to be possible excluding a small set of size $ \cO(e^{-\varepsilon^{-a}})$, $0<a\ll 1$, of initial data. 
    \Cref{thm:main_intro} can therefore be seen as a complementary result to \cite{DIP2} in the \emph{unit torus}, as it proves the existence of an  exponentially small set of random initial data for which the solution does develop a large $L^{\infty}$-norm.
\end{enumerate}

\Cref{thm:main_intro} is complemented by  proving 
that the asymptotic probability of rogue wave formation is identical to the asymptotic probability of the set of rogue waves that form due to the {\it quasi-synchronization} of the first  $\tM$ 
phases $ (\theta_j(t))_{0<|j|\leq \tM}$  of the free water profile  
 \be\label{eta.final}
   \eta (t,x)= \frac{1}{\sqrt{\pi}}\, \sum_{j\in\Z\setminus\{0\}} \rho_j(t)\, \cos (\theta_j(t) +j x) \,  , \quad \rho_j(t) \geq 0 \ . 
\ee 
To formalize such statement, consider the event  that the water waves solution $(\eta,\psi)$ to \eqref{ww} 
with random initial data \eqref{eq:init_data} exists in 
$C^0([0,t], H^\ts_0(\T, \R)\times \dot H^\ts(\T, \R))$, $\ts > \frac32$,
up to a time $t $ and has the first $\tM $ phases quasi-syncronized, namely
\begin{equation}\label{eventp}
\mathfrak{P}^t(\tM,\epsilon):=
\Big\{ \omega \in \mathbb{\Omega}  \mid \exists \,  \eta \in C^0([0,t], H^\ts_0) \  \mbox{and}\ |\theta_j(t) \ \mathrm{mod}\ 2\pi|\leq \epsilon \quad \forall 0<|j|\leq \tM \Big\} \, .
\end{equation}
We prove the following result directly valid on the optimal time interval.

\begin{thm}[{\bf Dispersive focusing}]\label{thm:df}
For any $\tM \in \N$,  $\lambda_0 >0$, $\delta \in (\frac35,1)$ and $0<\kappa \ll 1-\delta$, 
the free surface $\eta (t,x) $ of the water waves solution of  \eqref{ww} with random initial datum \eqref{eq:init_data} 
satisfies 
the following large deviation principle for any time $|t| \leq \eps^{-3 (1-\delta) + \kappa}$, 
\begin{equation}\label{0210:1548}
\lim_{\varepsilon\rightarrow 0^{+}} \varepsilon^{2\delta}\, \log \P \left(
\mathfrak{R}^t(\lambda_0 \varepsilon^{1-\delta}) \cap 
\mathfrak{P}^t(\tM,\eps^{\frac{\kappa}{5}})
\right) = -\frac{\lambda_0^2}{
2\bsi^2
}    \, , 
\qquad
\mathfrak{R}^t(\lambda) \mbox{  in } \eqref{eventr} , 
\quad \mathfrak{P}^t(\tM,\epsilon) \mbox{  in } \eqref{eventp}
\, , 
\end{equation}
and $\bsi$  in  \eqref{upsigma}.
\end{thm}
This result suggests  that dispersive focusing is the dominant mechanism in rogue wave formation.

\subsection{Ideas of the proof}\label{sec:ideas}

Let us explain the main ideas of the proof of \Cref{thm:main}, which implies \Cref{thm:main_intro} and the related \Cref{thm:df}. As mentioned before, 
it relies on deterministic long time existence results based on a paradifferential Birkhoff normal form jointly with novel probabilistic arguments to justify  the dispersive focusing mechanism of formation of rogue
waves. 

\smallskip
\noindent{\bf Well-posedness of solutions and normal form.}
We begin by introducing the deterministic water wave theory. 
For  any $\delta, \varepsilon \in (0,1)$, $\tR >0 $,  $s \geq 0$, we define the ball of initial data 
\begin{equation}
\begin{aligned}
\label{cB0}
\cB_0:= \cB_0(\eps, \delta, \tR , s) := & \left\{(\eta_0, \psi_0) \in H^s_0(\T, \R)\times \dot H^{s+\frac12}(\T, \R)  \colon  \quad \norm{(\eta_0, \psi_0)}_{H_0^{s} \times \dot H^{s+\frac12}}
\leq \tR \varepsilon^{1-\delta} \right\}  \,  , \\
& \quad \text{where} \quad 
\norm{(\eta_0, \psi_0)}_{H_0^{s} \times \dot H^{s+\frac12}}^2 := 
 \norm{\eta_0}_{H_0^s}^2 +
 \norm{\psi_0}_{\dot H^{s+\frac12}}^2 \, . 
\end{aligned}
\end{equation}
Following the common probabilistic perspective, when considering random initial  data $(\eta_0^\omega, \psi_0^\omega)$ as in \eqref{eq:init_data},  we shall still  denote by $\cB_0$ the event
\be\label{cB0p}
\cB_0:= \cB_0(\eps, \delta, \tR , s) :=\left\{ \omega\in \mathbb{\Omega}  \ : \  \norm{(\eta_0^\omega, \psi_0^\omega)}_{H_0^{s} \times \dot H^{s+\frac12}} \leq \tR \varepsilon^{1-\delta} \right\} \, . 
\ee
The fact that $ \psi_0 (x) $ is more regular that 
$ \eta_0 (x) $ 
in \eqref{cB0}-\eqref{cB0p} 
is due to the choice $ d_n = c_n n^{-1/2}$ in \eqref{ckdk}, 
which implies the Gaussianity of the Fourier coefficients $ \zeta_{0j}^\omega $ in \eqref{zeta0}, as we explain in  \Cref{lem:zeta0.d}.

The deterministic 
theory of \cite{ABZ1,BFP} guarantees that 
any random initial datum  in $\cB_0$, with $s>s_0 \gg 1 $ sufficiently large, yields a solution to the water waves equations \eqref{ww} that exists up to times $|t| \leq \eps^{-3(1-\delta) + \kappa}$ and remains small,  as stated in \eqref{est.flow.etapsi} of 
\Cref{cor:LIP}. The latter  is established using the Birkhoff normal form which reduces the water wave equations to system \eqref{transformed_eq},
as
proved in \cite{BFP}. The main novelty of Theorem \ref{BNFtheorem},    
compared to the original result in \cite{BFP},  
is that the normal form map is a \emph{phase space transformation} that satisfies the  Lipschitz bounds \eqref{phi_est} and  \eqref{lipB}. To establish these properties
we modify the construction of \cite{BFP}, as   
detailed in \Cref{app:normalform}. 
As a consequence, the map
\begin{equation}\label{eq:z_intro}
(\eta_0^\omega, \psi_0^\omega) \mapsto z(t,x; \eta_0^\omega, \psi_0^\omega) \, , 
\end{equation}
which sends the initial data to the flow of the normal form equation \eqref{transformed_eq} at time $t$, is Lipschitz continuous (see \Cref{cor:LIP}-($iv$)).
This
property is crucial for the probabilistic analysis developed in \Cref{partII}.

It is also important to note that 
the complementary event $\cB_0^c = \mathbb{\Omega} \setminus \cB_0 $,  where  $\cB_0 (\eps, \delta, \tR , s) $  is defined in 
\eqref{cB0p},  
 has negligible probability compared to that in \eqref{eq:LDP}
(cf. Lemma \ref{lem:negligible})
\begin{equation}
\label{terrone_intro}
\log \P(\cB_0^c) \leq  -  \frac{\tR^2}{ 
4\| \vec c \|_{h^s}^2 } \varepsilon^{-2\delta } + 1 \ , \qquad 
\mbox{ where } \  
\| \vec  c \|_{h^s}  \  \mbox{ is defined in } \ \eqref{norm.c} \, .  
\end{equation} 
As previously noted, 
the primary challenge in proving \Cref{thm:main_intro}, is that the surface profile $ \eta (t,x) $   fails to preserve the Gaussianity of the initial data during time evolution. 
We address this fundamental difficulty by employing different approximate solutions depending on the timescales, together with different methods to establish dispersive focusing.

\smallskip
\noindent{\bf Approximate solutions.}
We construct approximate solutions 
exploiting the cubic integrability of the deep water waves  normal form system \eqref{BNF3}. 
The first approximation that we use, see  \eqref{etaapptilde},  is 
\be \label{u_app_intro}
\eta_{\rm app}(t,\cdot) := \sqrt{2} \Re \,  |D|^{\frac14}  \, u_{\rm app}(t,\cdot) \ , \quad 
 u_{\rm app}(t,x):= \frac{1}{\sqrt{2\pi}} \sum_{j \in \Z \setminus\{0\}} e^{- \im t \cL_j(I(\zeta_0^\omega)) + \im \phi_j^\omega} \, |\zeta_{0j}^\omega|  \, e^{\im j x }\, , 
\end{equation}
where  $\zeta_0^\omega (x) $ is the initial datum  in \eqref{zeta0}, 
$\phi_j^\omega$ the phase of the Fourier coefficient $\zeta_{0j}^\omega$
and  $\cL_j(I(\zeta_0^\omega))$  the nonlinear normal form frequencies in  \eqref{HCS1}.
\Cref{thm:approx} demonstrates that $ \eta_{\rm app}(t,x)$ 
stays close in the $L^\infty(\T)$-norm to the water waves profile 
$\eta(t,x)$ for times up to 
 $\eps^{-\frac52(1-\delta) + \kappa}$. Furthermore 
\Cref{cor:etaapp} shows that $ \eta_{\rm app}(t,x)$ preserves the Gaussianity of the initial datum. 
While the moduli of the Fourier coefficients of 
 $ u_{\rm app}(t,x) $ are clearly Rayleigh-distributed (due to the Gaussianity of $\zeta_{0j}^\omega$), it is not immediately clear whether the nonlinear phases $- \im t \cL_j(I(\zeta_0^\omega)) + \im \phi_j^\omega$ are uniformly distributed on the circle and independent from the moduli. 
 This difficulty arises because the nonlinear frequencies 
$\cL_j(I(\zeta_0^\omega)) $ in \eqref{HCS1}
are coupled across all infinitely many actions.

 However \Cref{lem:gauind} resolves this issue by 
using conditional expectation arguments 
which allow us to deduce the Gaussianity of the Fourier coefficients  of $u_{\rm app}(t,x)$ and,
consequently, the Gaussianity of $\eta_{\rm app}(t,x)$. 
Then \Cref{thm:main} $(i)$ on the shorter time interval $|t| \leq \eps^{-\frac52(1-\delta) + \kappa}$, follows from a large deviation principle proved in \cite{GGKS}  for Gaussian random variables.

\smallskip
Proving  \Cref{thm:main} $(ii)$ 
up to the optimal timescales 
$|t|\leq \varepsilon^{- 3(1-\delta) + \kappa }$ 
requires a completely different approach, as the profile \eqref{u_app_intro} 
does not accurately approximate the free water wave surface
$\eta(t,x)$ over this  interval.
The proof of  \Cref{thm:main} $(ii)$ relies on the improved approximation given by 
\eqref{etaapp2}, namely
\be\label{u_intro_app_2}
 \eta_{\rm app2}(t,\cdot) := \sqrt{2} \Re\,  |D|^{\frac14}\, u_{\rm app2}(t,\cdot) \, , 
 \quad 
    u_{\rm app2}(t,x)  := \frac{1}{\sqrt{2\pi}}\!\!\!\sum_{j \in \Z \setminus\{0\}} \!\!\! e^{- \im \int_0^t  \cL_j(I(z(\tau; \eta_0^\omega, \psi_0^\omega))) \di \tau +\im \phi_j^{\omega}} \, |\zeta_{0j}^\omega | \, e^{\im j x } \, , 
\ee
where the phases are now evaluated along the full trajectory of the solution $z(t; \eta_0^\omega, \psi_0^\omega)$ of the normal form equation \eqref{theoBireq}.
\Cref{thm:approx2} proves that 
this new approximation stays close in $L^\infty(\T)$-norm to the full water waves solution $\eta(t,x)$ up to  optimal times  $|t|\leq \varepsilon^{- 3(1-\delta) + \kappa }$, but  
a drawback is the loss of control over its probability distribution. 
Hence we develop a new method that allows to identify its tail distribution.

\smallskip
\noindent{\bf Phase synchronization.} According to the dispersive focusing mechanism, a rogue wave occurs as a result of constructive interference, i.e. the phases of its Fourier coefficients align at a given spatial point, 
where the wave peaks increase its amplitude. 
We make this mechanism
rigorous by computing the probability that,
at a fixed time $|t| \leq \eps^{-3(1-\delta) + \kappa}$, the approximate solution \eqref{u_intro_app_2} has   
its  first $\tN \sim \eps^{-\gamma}$, $\gamma >0$,
 nonlinear phases  nearly synchronized, thereby creating a peak at $x =0$ (clearly this peak could form at any 
 point $ x \in \T $).
 Since the approximate solution is close to the true solution, this phase synchronization happens also for the full solution $\eta(t,x)$, which is the content of our \Cref{thm:df}.

To synchronize the phases of the approximate solution  \eqref{u_intro_app_2}, we 
shall use the random Brouwer fixed point theorem by Bharucha-Reid and Mukherjea, see \cite[Theorem 10]{BReid},
which we state as follows. 

\begin{thm}[{\bf Random Brouwer fixed point}] 
\label{thm:Brouwer}
Let $(\mathbb{\Omega}, \cF, \P)$ be a complete probability space and $K$ be a compact, convex subset of $\R^n$. Consider a map 
$
\cT \colon   \mathbb{\Omega } \times K \to K
$
satisfying 
    \begin{itemize}
        \item[(i)] for any fixed $\omega \in \mathbb{\Omega}$, the map $\cT( \omega, \cdot)\colon K \to K$ is continuous;
        \item[(ii)] for any $\varphi \in K$, the map 
        $\cT(\cdot, \varphi) \colon \mathbb{\Omega} \to K $ is measurable. 
    \end{itemize}
    Then there exists a random variable  $\varphi :(\mathbb{\Omega}, \cF, \P) \rightarrow K$ such that
    $\cT( \omega, \varphi^\omega) = \varphi^\omega$ a.s.  in $\mathbb{\Omega}$.
\end{thm}
Let us explain how we use this result in Section \ref{sec:fixed}.
The nonlinear phases 
of the approximate solution \eqref{u_intro_app_2}, 
$$
\theta_j(t; \omega):= \phi_j^{\omega} - \int_0^t  \cL_j(I(z(\tau; \eta_0^\omega, \psi_0^\omega))) \di \tau \, , \qquad j \in \Z \, , 
$$
depend in a highly nonlinear way on the initial random phases $ (\phi_j^\omega)_{j \in \Z\setminus\{0\}}$ of the random initial data \eqref{eq:init_data}.
We seek the probability that the first $2\tN$ uniformly distributed  random phases $(\phi_j^\omega)_{0<|j| \leq \tN}$ are chosen so that all 
the $\theta_j(t;\omega)$, $0<|j| \leq \tN$, 
 align at zero at time $t$.
\\[1mm]
\indent 
{\sc First step: alignment of first $ 2 \tN $ phases for 
partially randomized initial data.}
We  ask whether
{\it at least one choice} 
of phases exists 
that guarantees such alignment. 
For this analysis, it is  natural to treat  the  nonlinear phases $ (\theta_j(t;\omega))_{0<|j| \leq \tN} $
as depending 
 deterministically on  $2\tN$ auxiliary initial phases $\vec\phi = (\phi_j)_{0<|j|\leq \tN}   \in [0,2\pi)^{2\tN} $, and randomly on the  initial moduli and  remaining phases.
 Precisely we  introduce
\begin{equation}\label{eq:theta.intro}
\theta_j (t;\omega,  \vec{\phi}):= 
\phi_j - \cT_j(\omega, \vec\phi) \ , 
\quad 
\cT_j(\omega, \vec\phi):= 
\int_0^t \cL_j (I(z(\tau; \wt{\eta}_0^\omega(\cdot;\vec{\phi}), \wt\psi_0^\omega(\cdot;\vec{\phi})) \di\tau  \quad \mbox{(mod}\ 2\pi ) \, , 
\ 0<|j|\leq\tN \, ,
\end{equation}
where  $(\wt\eta_0^{\omega}, \wt\psi_0^{\omega}) (x;\vec{\phi})$ is the partially randomized initial datum \eqref{eq:part_rand_init}, whose   
  first  $2\tN$ Fourier modes have deterministic phases   $\vec{\phi} $ (not the randomized ones $(\phi_j^\omega)_{0<|j|\leq\tN}$), whereas all the Fourier moduli and the remaining 
  Fourier phases  on the high modes 
  are randomized as in 
\eqref{eq:init_data}.
Note that if the deterministic phases $(\phi_j)_{0<|j|\leq\tN}$ were randomized as independent uniform distributions, one would obtain the original randomization \eqref{eq:init_data}, i.e. $(\wt\eta_0^{\omega}, \wt\psi_0^{\omega}) (x;\vec{\phi^\omega}) = (\eta_0^{\omega}, \psi_0^{\omega})$ as stated in \eqref{eq:part_to_full}. 

Having the phases  
 $ (\theta_j (t;\omega,  \vec{\phi}))_{0<|j|\leq \tN }$ 
 in \eqref{eq:theta.intro} aligned at zero at a fixed time $t$ amounts to proving the existence of a  fixed point of the random map 
 $$
 \cT:= (\cT_j)_{0<|j| \leq \tN}\colon\cB_0\times  K  \to K \, ,
 \quad  \text{for some} \ K \subset \R^{2\tN}  \ \text{compact and convex} \, . 
 $$
The  crucial \Cref{lem:T_lips} proves that the map $\cT$ satisfies the assumptions of the random Brouwer fixed point theorem \ref{thm:Brouwer}.
Specifically the Lipschitzianity of the map
\eqref{eq:z_intro},  demonstrated by the deterministic  \Cref{cor:LIP}, ensures 
the continuity and measurability of $ \cT $. 
Thus \Cref{thm:Brouwer} guarantees  
that 
\be\label{deteral}
\text{there exists a r.v.} \ 
\vec{\phi}^{\ast,\omega}\in  [0,2\pi)^{2\tN} 
\quad \text{such that} \quad  
\theta_j (t; \omega, \vec{\phi}^{\ast,\omega})=0
\, , \quad \forall 0<|j|\leq\tN \, ,   \ 
\text{a.e. in } \mathbb{\Omega} \, . 
\ee
It is not trivial at all that the fixed point $\vec{\phi}^{\ast,\omega}$ is a random variable,
as  
Brouwer fixed points typically do not 
exhibit continuous dependence on external parameters  (i.e. the remaining Fourier phases and moduli).
Moreover, we note that 
$ \vec{\phi}^{\ast,\omega} $ is {\it independent} of the original $2\tN$ random uniformly distributed 
initial phases $(\phi_j^{\omega})_{0<|j| \leq \tN}$ by construction,  a crucial property that we shall exploit in a moment, see  \Cref{rk:independence}.

\smallskip

\indent 
{\sc Second step: stability and factorization properties. }
We now consider the set of events where  
the uniformly distributed phases  
$\vec{\phi}^{\omega} := 
(\vec{\phi}^{\omega}_j)_{0< |j| \leq \tN}$ are  close to the random fixed point $\vec{\phi}^{\ast,\omega}$, namely   the set
\begin{equation}\label{eq:N_intro}
\Nc ( \alpha)  
:= \Big\{ \omega\in\mathbb{\Omega} \mid \big| 
\phi^{\ast,\omega}_j - \phi^{\omega}_j \big| < \alpha \, , 
\quad \forall 0 < |j | \leq \tN \Big\}  \, , \quad 
\alpha \in (0, \pi) \, . 
\end{equation}
We prove in \Cref{lem:T_lips} three  crucial properties: 
\begin{itemize}
\item $(a)$  the set $\Nc ( \alpha)   $ is  measurable with probability $\left(\frac{ \alpha}{\pi}\right)^{2\tN}$. 
\item $(b)$ {\sc stability property \eqref{5.8}}:  for
any $\omega\in\cN(\alpha)$  the random phases $\vec{\phi}^{\omega}$ 
generate nonlinear phases  
$ (\theta_j(t; \omega, \vec{\phi}^{\omega}))_{0<|j| \leq \tN} $ quasi-synchronized  near zero, provided $\alpha$ is  sufficiently small, cf.
\eqref{alpha.N}. Specifically  $ \alpha $ is small with the inverse of the Lipschitz constant of the flow map \eqref{eq:z_intro}, estimated in \eqref{lip.z}. 

\item $(c)$ 
{\sc factorization property \eqref{splitting}:}
the  probability of the intersection of $\Nc ( \alpha)  $ with an event $ \cA $
in the sub-$\sigma$-algebra generated by the moduli and the remaining phases, is the product of their probabilities.
The 
factorization property  \eqref{splitting} 
ultimately follows  from   the independence of 
 $\vec{\phi}^{\ast,\omega} $ and $ \vec{\phi}^{\omega}$.
\end{itemize}

{\sc Third step: 
phase synchronization for the approximate solution.} In view of 
\eqref{deteral} and  the above stability property $(b)$,
\Cref{cor:phase_sync} proves that, 
at any time $t$, within  $\Nc(\alpha)$ and provided $\alpha$ is sufficiently small,    the approximate solution 
\eqref{u_intro_app_2}
admits  
 the lower bound
\be\label{2909:1522}
\sup_{x \in \T}\eta_{\rm app2}^{\omega}(t,x) 
\geq { \frac{\eps (1-\eps)}{\sqrt{\pi}} }\, \, \sum_{0<|j|\leq \tN}   c_j R_j^{\omega}  + o(\eps^{2-\delta}) \, .
\ee
\Cref{thm:lo_LDP_improved} 
combines these ingredients to establish the following lower bound on the tail-statistics  
 of the crest of $ \eta(t,x)$, 
\begin{align*}
    \P \Big( \sup_{x \in \T}\eta(t,x) \geq \lambda_0 \eps^{1-\delta} \Big) & \geq \P 
    \Big( \sup_{x \in \T}\eta_{\rm app2}^{\omega}(t,x) \geq \lambda_0 \eps^{1-\delta} (1+o_{\eps}(1)) \Big)
    \\
    & \geq 
 \P \Big( \big\{ \sup_{x \in \T}\eta_{\rm app2}^{\omega}(t,x) \geq \lambda_0 \eps^{1-\delta} (1+o_{\eps}(1))\big\} \cap \cN(\alpha) \Big) \\
 &
  \stackrel{\eqref{2909:1522}}{\geq }
    \P \Big(  
    \Big\{ 
    \sum_{0<|j|\leq \tN}   c_j R_j^{\omega}
    \geq \sqrt{\pi}\lambda_0 \eps^{1-\delta}(1+o_{\eps}(1)) \Big\} 
    \cap \cN(\alpha) \Big) \\
    & 
    \stackrel{(c)}{=}
       \P   \Big( 
    \sum_{0<|j|\leq \tN}   c_j R_j^{\omega}
    \geq \sqrt{\pi}\lambda_0 \eps^{1-\delta} (1+o_{\eps}(1))\Big)
    \, \P( \cN(\alpha) )\\
   & 
     \stackrel{(a)}{=}  \P \Big( 
    \sum_{0<|j|\leq \tN}   c_j R_j^{\omega}
    \geq \sqrt{\pi}\lambda_0 \eps^{1-\delta}(1+o_{\eps}(1)) \Big)
    \, \left(\frac{ \alpha}{\pi}\right)^{2\tN} \, .  
\end{align*}
 The  tail-statistics of the last set is readily computed by  \Cref{thm:LDP_Rayleigh}. 
The  choice of $\alpha$  in 
\eqref{choice.alpha} guarantees that  $\left(\frac{ \alpha}{\pi}\right)^{2\tN} =  \exp( \cO (\eps^{-3(1-\delta) - \gamma}))$,
which is larger than 
$
\exp(-\lambda_0^2\eps^{-2\delta}/2\bsi^2)$ provided $0<\gamma<5\delta -3$. This finally proves, for 
any $\delta > \frac35$,  
\Cref{thm:main} $(ii)$.
\\[1mm]
{\bf Organization of the paper. }
In Part \ref{sec:normalform_results}
we report  the deterministic 
well posedness  theory of water waves equations  for small 
periodic initial data. Because
of the quasilinearity of 
the equations this requires delicate energy estimates
based on paradifferential calculus and 
the normal form Theorem \ref{BNFtheorem} 
that we prove in Appendix
\ref{app:normalform}. 
In Part \ref{partII} we prove the large deviation Theorem \ref{thm:main}, which implies
\Cref{thm:main_intro}.
The primary innovation is the new rogue wave formation mechanism, detailed in the extensive \Cref{sec:LDPIII}, which relies on an unconventional combination of a deterministic normal form approximation and a random fixed-point argument to show phase quasi-synchronization. 
We include Section  \ref{sec:52} because its easier argument offers a valuable opportunity to integrate expertise from both PDEs and Probability. 
In \Cref{sec:allr} we prove Theorem \ref{thm:df}.

\paragraph{Notation.} 
We denote   $ \N = \{ 1, 2, \ldots \}$ the natural numbers and 
$ \N_0 := \N \cup \{0\} $. 
We write $a\lesssim b$ to indicate an estimate of the form $a\leq c b$ for some constant $c>0$ which may change from line to line. For functions $f,g:[0,\infty)\rightarrow\R$ of some variable $\varepsilon>0$, we write $f(\varepsilon)\ll g(\varepsilon)$ or $f(\varepsilon)= o (g(\varepsilon))$ to 
mean $\lim_{\varepsilon\rightarrow 0^{+}} f(\varepsilon)/g(\varepsilon) = 0$. We also write $f(\varepsilon)=\O ( g(\varepsilon))$ meaning $f(\varepsilon)\lesssim g(\varepsilon)$ for an implicit constant which is independent on $\varepsilon$.
We denote by $h^s$, $s \in \R $, the Hilbert space 
of sequences $\vec c := (c_j)_{j\in\Z\setminus\{0\}}$, $ c_j \in \C $,  with norm
\be\label{norm.c}
\| \vec c \|_{h^s}^2 := 
{\mathop \sum}_{j \in\Z\setminus \{0\}}  |j|^{2s} \, |c_j|^2  \, . 
\ee
We set $ \| \vec c \|_{\ell^\infty} := 
\sup_{j \in\Z\setminus \{0\}} \, |c_j| $. 
We expand a periodic function $u:\T \rightarrow \C$ in Fourier series   as 
\[
u(x) = \frac{1}{\sqrt{2\pi}} 
\sum_{j \in \Z} u_j e^{\im j x} \, , \quad 
u_j = \frac{1}{\sqrt{2\pi}} \int_0^{2\pi} u(x)\, e^{- \im jx}\,  \di x \, ,
\]
and we often identify 
$ u (x) $ with the sequence of 
its Fourier coefficients 
$ (u_j)_{j \in \Z} $.
We denote by $ H^s := H^s(\T,\C) $, $s \in \R $, the Hilbert space of  $u$ such that $(u_j)_{j\in\Z}\in h^s$.  
The subspace $H^s_0\subset H^s$, resp. 
$ {\dot H}^s := 
\dot{H}^s(\T,\C)$, consists of those functions in $H^s$ with zero average, resp.~equivalent classes modulo constants,  with  norm 
\be\label{usobo}
\| u \|_{H_0^s} = \| u \|_{\dot H^s} = 
\Big( {\mathop \sum}_{j\in\Z \setminus \{0\} } |u_j|^2 |j|^{2s} \Big)^{\frac12}  \, .
\ee
Finally, we define the Banach spaces $\mathcal{F} L^{s,p}$, $s\geq 0$, $p\in [1,\infty) $, of 
periodic functions 
$u:\T \rightarrow \C$ with norm  
\be\label{FL}
\norm{u}_{\mathcal{F} L^{s,p}} = \Big( {\mathop \sum}_{j\in\Z\setminus\{0\}} |u_j|^p  |j|^{sp} \Big)^{\frac{1}{p}} \, . 
\ee
{\bf Random variables.} 
We consider a probability space 
$ (\mathbb{\Omega}, \cF, \P) $ where $\mathbb{\Omega}$ is a measure space
with a  $\sigma $-algebra $ \cF $ and $ \P$ is a probability measure, i.e. $ \P(\mathbb{\Omega}) = 1 $. 
We consider  random variables (r.v.), namely  measurable functions 
$$ 
X :(\mathbb{\Omega},\cF) \to (\R, \cB(\R) ) \, , \quad  \omega \mapsto 
X^\omega \, , 
$$ 
where $\cB(\R) $
is the Borel $\sigma $-algebra of $ \R $. The law or distribution of $ X $ is the push-forward measure of $ \P $ under $ X $ and it is characterized by the 
cumulative distribution function 
$ F_X(x) := \P(X\leq x )$.
We denote by $\sigma(X):=\{ X^{-1}(B) \colon B \in \cB(\R)\}$ the $\sigma$-algebra generated by $X$, 
and by $\sigma(X,Y)$ the smallest $\sigma$-algebra containing $\sigma(X) \cup \sigma(Y)$.
Two $\sigma$-algebras $\cG_1, \cG_2$ are independent iff $\P(A_1 \cap A_2) = \P(A_1) \P(A_2)$ for any $A_j \in \cG_j$, $j=1,2$.
Two random variables $ X, Y $
are {\it independent} iff  $\sigma(X)$ and $\sigma(Y)$ are independent. We write $ X \independent Y $. 
It may happen that the distribution of 
$ X^\omega $ is assigned via a
{\it probability density function} $ f(x) $ (p.d.f.), 
namely a non-negative function $ f : \R \to \R $ such that    
$ \P(X\leq x ) = \int_{-\infty}^x f(y) \, \di y  $.  
In this case 
$ \P(X \in [a,b]) = \int_a^b f(x) \, \di x $. 
The {\it expected value} of the random variable $ X $ is 
$ \E X := \int_{\mathbb{\Omega}} X^\omega \di \P (\omega) = \int_{\R} x f(x) \, \di x $
and its  {\it variance} 
is $ \text{Var}(X) := \E (X- \E X)^2$. 
If $ X, Y $ are independent then $ \E (X Y) = 
\E (X) \E (Y) $.

Given a sub-$\sigma$-algebra $\cG\subset\cF$ and a random variable $X^{\omega}$ in $L^1(\mathbb{\Omega},\cF, \P)$, we denote by $\E[X^{\omega}\mid\cG]$ the conditional expectation of $X^{\omega}$ with respect to $\cG$. When $\cG$ is generated by another r.v. $Y^{\omega}$, i.e. $\cG=\sigma(Y^{\omega})$, we will simply write $\E[X^{\omega}|Y^{\omega}]$ to denote $\E[X^{\omega}\mid\cG]$. See \cite[Sections 10.1-10.2]{Resnick} for the definition and main properties.  

\begin{rk}\label{rk:cond_exp} We record here the two fundamental properties of conditional expectation that we will use in this work. Let  $\cF$ be a $\sigma$-algebra, $\cG$  a sub $\sigma$-algebra of $\cF$, and  $X\in L^1(\mathbb{\Omega},\Fc,\P)$. Then
    \begin{enumerate}
        \item[(i)] {\sc Tower property:} $\E[X] = \E[ \E[X | \cG]]$. 
        \item[(ii)] {\sc ``Take out what is known":} If $Y\in L^{\infty}(\mathbb{\Omega},\Fc,\P)$ is $\cG$-measurable, then $\E[YX | \cG] = Y \E[X|\cG]$.
    \end{enumerate}
\end{rk}

As is customary in probability theory, when considering random initial data
we  use the notation $ (\eta_0^{\omega},\psi_0^{\omega})\in \cB_0$ where $ \cB_0 $ is the ball in \eqref{cB0}, to mean $\omega\in \cB_0$ where $ \cB_0$  is the corresponding event in \eqref{cB0p}. 

We also recall the 
following elementary definitions.

\begin{defn}\label{def:rv}
    Let $\sigma \neq 0$. A random variable $X^\omega$ is a 
\begin{itemize}
    \item[1)]  {\bf Real Gaussian r.v.  $\sim \cN_\R(0, \sigma^2)$} if its p.d.f is  $f_\sigma(x)= \dfrac{1}{\sqrt{2\pi \sigma^2}}e^{-\frac{x^2}{2 \sigma^2}}$. Its expected value is $0 $  and its variance 
    is $ \sigma^2 $. 
    \item[2)] {\bf Complex Gaussian r.v.  $\sim\cN_\C(0, \sigma^2)$} if 
$X^\omega = \dfrac{a^\omega + \im b^\omega}{\sqrt{2}}$ with $a^\omega$, $b^\omega$ independent  real Gaussian variables in $\cN_\R(0, \sigma^2)$ (note that 
$ a^\omega $ and $ b^\omega $ have the same variance).
\item[3)] {\bf Rayleigh r.v.  $\sim\sR(\sigma)$}  if
its p.d.f is 
$
f_\sigma(x) = \dfrac{x}{\sigma^2} e^{-\frac{x^2}{2\sigma^2}}$ for any $x \geq 0$ and $f_\sigma(x) = 0$ for any  $x <0$.
\item[4)]{\bf Uniform r.v. $\sim \sU([0,2\pi])$} if its p.d.f. is $f(x) = \frac{1}{2\pi}$ for any $x \in [0,2\pi]$ and $f(x) = 0$ otherwise.
\item[5)]  {\bf Exponential r.v.}  $ \sim 
\Es (\lambda) $ if  its p.d.f. is $ f(x) =  \lambda e^{-\lambda x} $ for any $ x \geq 0 $ and $ f(x) =  0 $  
for any $ x < 0  $. 
\end{itemize}
\end{defn}
\begin{rk}\label{basic_prob}
The following simple 
properties  will be used repeatedly in the next sections.
\begin{enumerate}
\item[(i)] If $ a^\omega \sim \cN_\R(0,\sigma^2)$ then  $\lambda a^\omega \sim \cN_\R(0, \lambda^2 \sigma^2)$ for any $\lambda \in \R $. Furthermore, the following tail estimate holds:
\be\label{asintail}
 \P(a^\omega > z) \sim \frac{\sigma}{z \sqrt{2\pi}} e^{-\frac{z^2}{2\sigma^2}} \quad \mbox{ as } z \to \infty \, . 
\ee
\item[(ii)] If $ a_1^\omega \sim \cN_\R(0,\sigma_1^2) $ and  $  a_2^\omega \sim \cN_\R(0,\sigma_2^2) $ 
are {\it independent} real Gaussian random variables, then the sum 
$a_1^\omega + 
a_2^\omega \sim \cN_\R(0, \sigma_1^2 + \sigma_2^2)$ is a real Gaussian r.v.
\item[(iii)] If $\zeta^\omega \sim \cN_\C(0, \sigma^2)$ then $\E[\zeta^\omega] = 0$ and $\E[|\zeta^\omega|^2] = \sigma^2$. 
For any  $ \lambda \in \C $, 
$ \lambda \zeta^\omega \sim \cN_\C (0,|\lambda|^2 \sigma^2 ) $ and 
$ \zeta^\omega = |\zeta^\omega| e^{\im \theta^\omega} $ 
where the modulus 
$ |\zeta^\omega|$ is a   Rayleigh  r.v.    $\sim \sR \big( \frac{\sigma}{\sqrt{2}} \big)$, and the phase $\theta^\omega$ is a uniform r.v.  
$\sim \sU([0,2\pi])$.
\item[(iv)] If $R^\omega \sim \sR(\sigma)$ then 
for any $\lambda >0$, $\lambda R^\omega \sim\sR(\lambda\sigma)$. 
\item[(v)] If $X^{\omega} \sim \cN_{\C}  (0,\sigma^2) $
then $ |X^{\omega}|^2 \sim \Es(1/\sigma^2)$.
\item[(vi)]  If 
$ X^\omega $ is $ \sim  \Es (\lambda) $
then $ \alpha X^\omega $ is  
$ \sim  \Es ( \lambda/\alpha) $.  
\end{enumerate}
\end{rk}

\part{Deterministic pure gravity deep water waves}\label{sec:normalform_results}

\setcounter{section}{2}
\setcounter{thm}{0}
\setcounter{equation}{0}

In this part we present the deterministic well posedness theory of pure water waves equations in deep water. 
For initial data in Sobolev spaces the local Cauchy theory was  developed  first by S. Wu \cite{Wu0} using a Lagrangian formulation of the equations  and by 
Lannes \cite{Lannes} and  then Alazard-Burq-Zuily \cite{ABZ1} in an Eulerian framework, i.e. for system \eqref{ww}, expressing the problem in terms of a ``good unknown" 
$ \omega $, instead of $ \psi$, and 
a paralinearization of the Dirichlet-Neuman
operator.
We also mention the more 
geometric approach developed in \cite{SZ1} to study free boundary value problems. 

Longer lifespans follow provided the 
 initial datum is  sufficiently regular and  small,  as first proved in \cite{BFP} via a paradifferential Birkhoff normal form approach, result extended to more general initial data in \cite{Wu3,DIP}. 
These works 
ultimately descend  from the fact that the  pure gravity water waves equations in deep water are {\it integrable} up to quartic vector fields, as  formally guessed in \cite{ZakD} and rigorously proved in \cite{BFP}. This allows to control the evolution of sufficiently smooth and small  initial data of size $\eps $,  for times of order $\eps^{-3}$, as stated in Theorem \ref{thm:mainBFP} below, which  is actually a consequence of the Birkhoff normal form Theorem 
\ref{BNFtheorem}, the energy estimates of Corollary \ref{thm:control_hs} and the local existence result in \cite{ABZ1}.  

We prove Theorem \ref{BNFtheorem} in Appendix \ref{app:normalform}.
 The difference of Theorem  \ref{BNFtheorem} with respect to   \cite{BFP} is that the normal form transformation $ \mathfrak{B}(u) $ that we construct satisfies the Lipschitz properties 
\eqref{phi_est}, \eqref{lipB}. Such properties are necessary in the probabilistic approach developed in \Cref{partII}. 
In Theorem 
\ref{cor:LIP} below we collect
{\it all} the deterministic information 
on the water  waves flow required for the probabilistic proof of  the 
Large deviation \Cref{thm:main}.
In this way a reader may pass to  \Cref{partII},  skipping the (more technical) Appendix \ref{app:normalform}. 

We first report basic definitions and 
properties 
of paradifferential calculus.

\paragraph{Paradifferential operators.}

Paradifferential calculus
can be seen as a nonlinear version of pseudo-differential calculus for 
symbols with finite regularity. 

\begin{defn}{\bf (Symbols)} Given $m \in \R $,  we denote by $\Gamma^m_{L^{\infty}}$ the space of 
functions $ a : \T\times \R\to \C $,  $ (x, \xi) \mapsto a(x, \xi)$, 
which are $C^\infty$ with respect to $\xi$ and such that, for any  $ \beta \in \N_0 $, 
there exists a constant $C_\beta >0$ such that
$$
\big\| \pa_\xi^\beta \, a(\cdot, \xi) \big\|_{L^{\infty}} \leq C_\beta \, \la \xi \ra^{m - \beta}  , \quad \forall \xi \in \R \,  . 
$$
We endow $\Gamma^m_{L^{\infty}}$ with the  family of norms defined, for any $n \in \N_0$, by
\begin{equation}
\label{seminorm}
\abs{a}_{m, {L^{\infty}}, n}:=  \max_{0 \leq \beta \leq n }\, 
\sup_{\xi \in \R} 
\, \big\| \la \xi \ra^{-m+ \beta} \, \pa_\xi^{\beta} a(\cdot,  \xi) \big\|_{L^{\infty}} \, . 
\end{equation} 
\end{defn}
Paradifferential operators are defined as pseudo-differential operators with a suitable regularized symbol. 
Let $\epsilon \in (0,1)$ and consider a 
$ C^\infty $, even cut-off  function $\chi\colon \R \to [0,1]$ such that
\begin{equation}\label{def-chi}
\chi(\xi) =  
\begin{cases}
1 & \mbox{ if } |\xi| \leq 1.1 \\
0 & \mbox{ if } |\xi| \geq 1.9 \, , 
\end{cases}  
\qquad \chi_\epsilon(\xi) := \chi \left( \frac{\xi}{\epsilon}\right) \, . 
\end{equation}
\begin{defn} {\bf (Bony-Weyl paradifferential operator)}
Given a symbol $a \in  \Gamma^m_{L^\infty}$, we define the 
{\em Bony-Weyl paradifferential operator}
$ {\rm Op}^{BW} {(a)} $ that acts on a periodic function $ u (x)  $
as   
$$
\left({\rm Op}^{BW} {(a)}[u] \right)(x) 
= 
\sum_{j \in \Z} \Big( \sum_{k \in \Z}
\widehat a \Big(j-k, \frac{j+k}{2}\Big) 
\, \chi_\epsilon\Big( \frac{j-k}{\langle j+k \rangle}\Big)
\, u_k \Big) e^{\im j  x } \, . 
$$
If the symbol $ a(x) $ is independent of $ \xi $
the operator  ${\rm Op}^{BW} {(a)} $ is called a {\it paraproduct}. 
\end{defn}

It is convenient to use the Weyl quantization 
for several reasons. In particular, it ensures that  
a paradifferential operator is 
self-adjoint if and only if 
the symbol $ a(x, \xi) $ is real. 
 
We remark that the operator 
$ \opbw (a) $ acts on homogeneous spaces of $ 2\pi $-periodic functions.  
Furthermore for a symbol $ a(\xi) $ independent of $x $, the paradifferential operator 
$ {\rm Op}^{BW} {(a)} = a (D) $ reduces to a Fourier multiplier operator. 
We also note that paradifferential operators are defined up to smoothing remainders, i.e. we can take different cut-off functions in \eqref{def-chi}.

The  following  basic property  about the action of a paradifferential operator (see e.g. \cite{BMM}) is often used.
\begin{prop}\label{prop:action}
{\bf (Continuity)}
 Let  $ a  $ be a symbol in $ \Gamma^m_{L^\infty} $ with $ m \in \R $.
Then $\Opbw{a}$ extends to a bounded operator 
$ \dot H^{s}(\T) \to  \dot H^{s-m} (\T) $ for any $ s \in \R $  and, for any $ u \in \dot H^s (\T) $,  
\be\label{cont00}
\norm{\Opbw{a}u}_{\dot H^{s-m}} \lesssim \, \abs{a}_{m, L^\infty, 4} \, \norm{u}_{\dot H^{s}} \, .
\ee
For a paraproduct, the linear map 
$L^\infty (\T)  \mapsto {\cal L}({\dot H}^s) $, $ a \mapsto \Opbw{a(x)} $, taking values in the bounded linear operators of 
$ {\dot H}^s (\T) $ satisfies  $ \norm{\Opbw{a(x)}}_{{\cal L}({\dot H}^s)}  \lesssim \| a \|_{L^\infty }$. 
\end{prop} 
In \Cref{app:para} we shall introduce 
more refined classes of  paradifferential symbols and smoothing operators  which  take into account  their Taylor expansion in $(\eta, \psi)$ at $(0,0)$, cf. \cite{BD}. This is needed in Appendix \ref{app:normalform} for  the proof of  the normal form \Cref{BNFtheorem}.

\paragraph{Birkhoff normal form and long time existence results.}

We  now present the novel Birkhoff normal form Theorem \ref{BNFtheorem} and some of its consequences.  

We  denote 
the horizontal and vertical components of the velocity field at the free interface 
\begin{align} 
\label{def:V}
& V :=  V (\eta, \psi) :=  (\pa_x \Phi) (x, \eta(x)) = \psi_x - \eta_x B \, , 
\\
\label{form-of-B}
& B :=  B(\eta, \psi) := (\pa_y \Phi) (x, \eta(x)) =  \frac{G(\eta) \psi + \eta_x \psi_x}{ 1 + \eta_x^2} \, ,
\end{align}
and we introduce the  ``good unknown'' of Alinhac 
\begin{equation}\label{omega0}
\begin{pmatrix}
\eta\\
\upomega
\end{pmatrix} = 
\cG(\eta,\psi) :=
\begin{pmatrix}
\eta \\
\psi-\Opbw{B}\eta
\end{pmatrix} \, . 
\end{equation} 
Note that $ V, B $ and
$ \upomega$,  are linear functions of $ \psi $.
The variables 
$ (\eta, \upomega) $ are introduced because, 
as noted in  \cite{Lannes,ABZ1}, they transform 
the water waves equations  into 
a quasi-linear hyperbolic system, for which 
Sobolev energy estimates are available,  
thus  guaranteeing the local existence of solutions. 

For future reference we mention  that,  
by \eqref{eq:DNan} 
and the algebra properties of Sobolev spaces,  for any $\ts > \frac72$, for some $ r> 0 $ small enough,  
the maps in \eqref{def:V} and  \eqref{form-of-B}, 
\be\label{an:DNX}
(\eta, \psi) \mapsto V(\eta, \psi) \, , 
(\eta, \psi) \mapsto  B(\eta, \psi) \, , \quad 
B^{\ts}(r)\times \dot H^{\ts}(\T, \R) \to H^{\ts-1}(\T, \R) \, , 
\quad \text{are {\it analytic}} \, ,  
\ee
as well  as  the good-unknown map in \eqref{omega0}, 
\be\label{an:omega}
(\eta, \psi) \mapsto {\cal G} (\eta, \psi) = (\eta, \upomega (\eta,\psi))\, \quad
B^{\ts}(r)\times \dot H^{\ts}(\T, \R) \to H^{\ts}(\T, \R) \times {\dot H}^{\ts}(\T, \R) \, 
 , \quad \text{is {\it analytic}} \, ,  
\ee
using \eqref{an:DNX} and the last property stated in \Cref{prop:action}.

We finally define the complex scalar variable
\begin{equation}\label{u0}
u := \frac{1}{\sqrt{2}}|D|^{-\frac{1}{4}}\eta+\frac{\im}{\sqrt{2}}|D|^{\frac{1}{4}}\upomega \, , \qquad \int_{\T}u(t,x) \, \di x=0 \, , 
\end{equation}
that is well defined  since  $ \int_{\T} \eta(t, \cdot) \, dx=0 $ 
is preserved by the flow of \eqref{ww}. 
Expressing in  \eqref{u0} the good unknown  
$ \upomega $ as a function of $ (\eta,\psi)$
by \eqref{omega0}, by \eqref{an:omega} we deduce that,  
for any $\ts > \frac72$, the map 
\be\label{an:u}
(\eta, \psi) \mapsto u(\eta, \psi) \, , \quad
B^{\ts}(r)\times \dot H^{\ts}(\T, \R) \to H_0^{\ts-\frac14}  \, , 
\quad \text{is {\it analytic}} \, .
\ee
Finally note that the difference of the complex 
variables $ u $ and $ \zeta $, defined respectively in  \eqref{u0} and  \eqref{zeta0p},  is 
\be\label{diffuzeta}
u - \zeta = \frac{\im}{\sqrt{2}}|D|^{\frac{1}{4}}
(\upomega - \psi) =
- \frac{\im}{\sqrt{2}}|D|^{\frac{1}{4}}
\Opbw{B}\eta \, . 
\ee
We now state the 
normal form result for small amplitude  pure gravity 
water waves in deep water. 
\begin{thm}
\label{BNFtheorem}
{\bf (Birkhoff normal form with Lipschitz map)}
Let $u(t)$ be defined as in \eqref{u0},  with $\upomega$ as in \eqref{omega0}, for $(\eta,\psi)$ solution of \eqref{ww}
satisfying 
$$
(\eta, \psi) \in C^0\left([-T, T]; H^{N+\frac14}(\T, \R) \times H^{N+\frac14}(\T, \R)\right) \, , 
\quad 
\int_\T \eta(0,x) \di x = 0 \, . 
$$
There exist $ N \gg K {\gg s_0} \gg 1 $  and $ 0 < \bar{\e} \ll 1  $,   such that, if 
\begin{equation}\label{u01}
\sup_{t\in [-T,T]} \norm{u(t,\cdot)}_{K,N} \ \leq \bar{\e} \, , \qquad 
 \norm{u(t,\cdot)}_{K,N} :=\sum_{k=0}^K \|\partial_t^k u(t)\|_{\dot{H}^{N-k}} \, , 
\end{equation}
then there exists a bounded 
and linearly invertible transformation $\mathfrak{B}=\mathfrak{B}(u)$ of $ \dot H^N $, which depends (nonlinearly) on $u$, 
and a constant $C :=C(N)>0$ 
such that for any  $\ts \in  [s_0, N]$
\begin{align}\label{Germe}
&{\|\mathfrak{B}(u) \|}_{\mathcal{L}(\dot{H}^{\ts}, \dot{H}^{\ts})}
	+ {\|(\mathfrak{B}(u))^{-1} \|}_{\mathcal{L}(\dot{H}^{\ts}, \dot{H}^{\ts})} \leq 1+C_N{\|u\|}_{\dot{H}^{\ts}} \, ,
    \\
\label{phi_est}
&\norm{\mathfrak{B}(u)  - \uno}_{\cL(\dot H^{\ts+1}, \dot H^{\ts} )}    \leq C_\ts \norm{u}_{\dot H^{\ts}} \, , 
\quad 
 \norm{\mathfrak{B}(u)^{-1}  - \uno}_{\cL(\dot H^{\ts+1}, \dot H^{\ts} )}   \leq C_\ts \norm{u}_{\dot H^{\ts}} \, , 
\end{align}
and the variable 
$z:=\mathfrak{B}(u)u$
satisfies the equation
\begin{equation}\label{theoBireq}
\pa_{t}z = -\im \pa_{\bar{z}}H_{ZD}(z,\bar{z}) + {\mathcal X}^{+}_{\geq 4}
\end{equation}
where:
\begin{itemize}
\item[(1)] {\sc (integrable Birkhoff normal form)} the Hamiltonian $H_{ZD}$ has the form 
\begin{align}
\label{theoBirHfull}
H_{ZD} = H^{(2)}_{ZD} + H^{(4)}_{ZD} \, , 
\qquad H^{(2)}_{ZD}(z,\bar{z}) := \frac12 \int_\T \big| |D|^{\frac14} z \big|^{2} \, \di x \, ,
\end{align}
with
\begin{equation}\label{theoBirH}
\begin{aligned}
H^{(4)}_{ZD} (z,\bar{z}) &:=
\frac{1}{4 \pi} \sum_{k \in \Z} |k|^3  \big(  |z_k|^4  - 2 |z_{k}|^2 |z_{-k}|^2   \big)
\\
&+  \frac{1}{\pi} \sum_{\substack{k_1, k_2 \in \Z, \, \sign(k_1) = \sign( k_2 )  \\ |k_2| < |k_1|}} |k_1| |k_2|^2  
  \big( - |z_{-k_1}|^2 |z_{k_2} |^2 + |z_{k_1}|^2  |z_{k_2}|^2 \big)
\end{aligned}
\end{equation}
 where $z_k:=\frac{1}{\sqrt{2\pi}} \int_\T z(x) e^{- \im k x} d x$ denotes the $k$-th Fourier coefficient of the function $z(x)$.
\item[(2)] {\sc (estimate of the remainder)} $ {\mathcal X}^{+}_{\geq 4}  := {\mathcal X}^{+}_{\geq 4} (u,\bar{u},z,\bar{z})$ 
is a quartic nonlinear term satisfying,
for some $C :=C(N)>0$, the ``energy estimate''
\begin{equation}\label{theoBirR}
{\rm Re}\int_{\T}|D|^N {\mathcal X}^{+}_{\geq 4} \cdot \bar{|D|^N z} \, \di x\leq C \|u\|_{\dot{H}^N}^{5} 
\end{equation} 
and 
\begin{equation}
\label{X_4_1}
\| {\mathcal X}_{\geq 4}^+ \|_{\dot H^1} 
\leq C \norm{u}_{\dot H^{s_0}}^4 \, . 
\end{equation}
\item[(3)]
{\sc (Lipschitz property)} 
 for any $\ts \in [s_0, N-1]$, there is $ r :=  r(\ts) >0 $, $ C_{\ts,r}>0$ such that  for   any   $u_i\in \dot H^{\ts}$ with $ \norm{u_i}_{\dot H^{\ts}} \leq r$, $i=1,2$,  one has
\be\label{lipB}
\norm{\mathfrak{B}(u_1)u_1 - \mathfrak{B}(u_2)u_2}_{ \dot H^{\ts}}  \leq C_{\ts,r} \norm{u_1 - u_2}_{\dot H^{\ts}} (1+\norm{u_1}_{\dot H^{\ts+1}}+\norm{u_2}_{\dot H^{\ts+1}} )\, .
\ee
\end{itemize}

\end{thm}

In  view of \eqref{Germe},
the norms of the solution $u(t) $ and the normal form variable $ z(t) $ are  equivalent  
\be \label{eq:normeq}
\norm{u(t)}_{\dot H^{\ts}}  \sim 
\norm{z(t)}_{\dot H^{\ts}}  \, , \quad \forall \ts \in [s_0,N] \, . 
\ee
In order to achieve the new Lipschitz property 
\eqref{lipB} 
we construct the normal form transformation $ \mathfrak B(u) u $ by a variant of the procedure in \cite{BFP}. 
The detailed proof  is postponed 
to  \Cref{app:normalform}. 

\smallskip

We now state a series of consequences of Theorem \ref{BNFtheorem}. 
First note that  system  \eqref{theoBireq} writes explicitly in Fourier as 
\begin{equation}\label{transformed_eq}
\dot z_k  = -\im \cL_k (I(z)) \, z_k + ({\mathcal X}^{+}_{\geq 4})_k \ , \quad \forall k \in \Z\setminus\{0\} \ , 
\end{equation}
where  $ I(z) $ is the ``action vector"  
\begin{equation}\label{actions}
I(z) :=(I_k(z))_{k \in \Z \setminus\{0\} } \ , \quad I_k(z):= (|z_k|^2)_{k \in \Z\setminus\{0\}} \, , 
\end{equation}
and 
$\cL_k(I) $ denotes the  $k$-th ``nonlinear frequency"
\begin{equation}\label{HCS1}
\begin{aligned}
\cL_k (I)
  :=&\ |k|^{\frac12}  - \frac{1}{\pi}\!\!\! \sum_{\substack{|k_4| < |k|, \\ - \sign(k) = \sign( k_4 )}} \!\!\! |k| |k_4|^2  I_{k_4}
  + \frac{1}{\pi} \!\!\!\sum_{\substack{|k_4| < |k|, \\ \sign(k) = \sign( k_4 )}} \!\!\! |k| |k_4|^2  I_{k_4}  \\
  & +  \frac{1}{ 2 \pi}  |k|^3 \big( I_k - 2  I_{-k} \big) 
  -  \frac{1}{\pi}  \!\!\!  \sum_{\substack{|k| < |k_1|, \\ \sign(k_1) = \sign( k )}} \!\!\! |k_1| |k|^2  \big(I_{-k_1}-  I_{k_1}\big) \ . 
\end{aligned}
\end{equation}
We shall use that for each $k \in \Z \setminus \{0 \}$ the map
$$
\cF L^{2,1} \to \R \ , \quad I := (I_k)_{k \in \Z\setminus\{0\}} \mapsto \cL_k(I) \, , 
$$
is Lipschitz with the quantitative estimate: there exists $C >0$ such that 
 \begin{equation}
 \label{L_lips}
{ \abs{\cL_{k}(I) - \cL_{k}(J)} \leq C |k|  
\norm{I - J }_{\cF L^{2,1}} } \, , 
\qquad \forall k \in \Z\setminus\{0\} \, , 
 \end{equation}
 where $ \norm{I}_{\cF L^{2,1}} = \sum_{k \neq 0} 
 k^2 |I_k| = \sum_{k \neq 0} k^2 |z_k|^2 $, cf. \eqref{FL}.

Since each nonlinear frequency $\cL_k(I)$ in \eqref{HCS1} is real,
each action $I_k(z) := |z_k|^2 $ is a first integral of the  cubic Hamiltonian Birkhoff normal form 
system 
\be\label{BNF3}
\dot z_k  = -\im \cL_k (I(z)) \, z_k \, ,
\quad \forall k \in \Z \setminus \{0\} \, , 
\ee
obtained by neglecting  
the quartic terms $ ({\mathcal X}^{+}_{\geq 4})_k $ in \eqref{transformed_eq}.
This directly implies the following energy estimate. 

\begin{cor}\label{thm:control_hs} 
{\bf (Energy estimate)} With the same assumptions and notations of Theorem \ref{BNFtheorem}, there exists   $C=C(N)>0$ such that the solution $z(t)$ of \eqref{theoBireq} satisfies 
\begin{equation}\label{energy_est}
\norm{z(t)}_{\dot H^N}^2 \leq \norm{z(0)}_{\dot H^N}^2 + C\, |t| \sup_{\tau \in [0,t]} \norm{u(\tau)}_{\dot H^N}^5 \ , \qquad \forall |t|\leq T \, , 
\end{equation}
and the vector of actions 
$I(t):= (I_k( z(t)))_{k \in \Z\setminus \{0\}}$  in \eqref{actions} 
satisfies 
\begin{equation}
\label{action_est}
 \norm{I(t) - I(0)}_{{\cF L^{2,1}}} \leq C  |t| \, \sup_{\tau \in [0,t]} \norm{u(\tau)}_{\dot H^{s_0}}^5  \, , \qquad \forall |t|\leq T \ .
\end{equation}
\end{cor}

\begin{proof}
The energy estimate  \eqref{energy_est} follows because
$$
\frac{d}{dt} \norm{z(t)}_{\dot H^N}^2 = 2 
{\rm Re}\int_{\T}|D|^N {\mathcal X}^{+}_{\geq 4} \cdot \bar{|D|^N z} \, \di x \stackrel{\eqref{theoBirR}} \leq C \|u\|_{\dot{H}^N}^{5} \, . 
$$
 Estimate \eqref{action_est} follows because  $ \frac{d}{dt} 
 I_k(t) = 2 {\rm Re}( (\mathcal{X}_{\geq 4}^+)_k \bar z_k ) $ 
 for any $k $ and so   
 \begin{align*}
\| \dot I \|_{{\cF L^{2,1}}} 
 \leq 2 \sum_{k \neq 0}
k^2  |[\mathcal{X}_{\geq 4}^+]_k| \, |z_k(t)| \lesssim \norm{z(t)}_{\dot H^1} \, \norm{\mathcal{X}_{\geq 4}^+}_{\dot H^1} \lesssim_N \norm{u(t)}_{\dot H^{s_0}}^5   
\end{align*}
by \eqref{X_4_1} and the equivalence \eqref{eq:normeq} of the norms 
$ \norm{z(t)}_{\dot H^{s_0}}  \sim 
\norm{u(t)}_{\dot H^{s_0}}  $. 
\end{proof}

The long-time existence of solutions up
to  times $ t \sim \varepsilon^{-3} $, as stated in Theorem \ref{thm:mainBFP}, is a direct 
consequence of  the energy estimate 
\eqref{energy_est} 
and 
the local existence result in \cite{ABZ1}. The latter  
proved that, if 
the initial data $ (\eta,\psi,V,B)_{|t=0} $ belong to the Banach space 
\begin{equation}\label{Xs}
X^s:= H^{s+\frac12}_0(\T, \R)\times \dot H^{s+\frac12}(\T, \R) \times H^s(\T, \R) \times H^s(\T, \R)   
\end{equation}
with $s > \frac32$, endowed with  norm
\be\label{epVBnorm}
\| (\eta, \psi, V, B) \|_{X^{s}}:= 
\|\eta \|_{H_0^{s+\frac12}} + \| \psi\|_{\dot H^{s+ \frac12}} + 
\| V\|_{H^{s}} + \| B \|_{H^{s}} \, , 
\ee
then there exists a time $T>0$ and a unique solution 
 $(\eta,\psi,V,B) \in C^0([-T, T],  X^{s})$ of \eqref{ww}.

\begin{thm}
\label{thm:mainBFP}
{\bf (Long time existence \cite[Theorem~1.2]{BFP})}
There exists $ s_0 > 0 $ such that, 
for all $s \geq s_0 $, there is $ \e_0 >  0 $ such that,  
for any initial data $ (\eta_0, \psi_0) $ satisfying 
\begin{equation}
\label{thm:main1}
\| (\eta_0, \psi_0, V_0, B_0) \|_{X^{s}} \leq \e \leq \e_0 \, , \quad
 \int_{\T} \eta_0 (x) \di x =  0 \, ,  
\end{equation}
where $ V_0 := V(\eta_0, \psi_0) $, $ B_0 := B (\eta_0, \psi_0) $ are defined by
\eqref{def:V}-\eqref{form-of-B}, the following holds: there exist constants $c>0$, $C>0$ and a unique classical solution 
\be\label{etapsi}
(\eta, \psi,V,B)  \in C^0([-T_\e, T_\e], X^{s})
\ee
of the water waves system \eqref{ww} 
with initial condition $(\eta,\psi)(0)= (\eta_0,\psi_0)$  
with   
\begin{equation}
\label{time}
T_\e \geq c  \e^{-3} \, ,
\end{equation}
satisfying 
\begin{equation}
\label{thm:main2}
\sup_{[-T_\e, T_\e]} \big( \| (\eta, \psi) \|_{H^{s} \times H^{s} } + \| (V, B) \|_{H^{s-1} \times H^{s-1} } \big) \leq {C} \e \, 
\end{equation}
and
\be\label{etaave0}
 \int_{\T} \eta (t,x) \di x = 0  \, \ \ \ \forall t \in [-T_\e, T_\e] \  . 
\ee
\end{thm}

\paragraph{Deterministic water waves flow.}
Theorem \ref{cor:LIP} below contains all the information  used in the probabilistic arguments of the forthcoming sections, concerning 
the pure gravity water waves flow  for any initial datum $(\eta_0, \psi_0)$ in the ball $\cB_0(\eps, \delta, \tR, s)$ defined in  \eqref{cB0}, 
provided  $ s $ is large enough and $\eps$ is small enough. To give a self-contained proof 
we  first report the following lemma 
stating the equivalence of the norm 
$\norm{u(t)}_{K,s} $ in  \eqref{u01} and 
the Sobolev norm $ \norm{u(t) }_{\dot H^s} $, for a 
sufficiently 
small solution $ u(t) $ of the water waves system. 
 
\begin{lem} {\bf (Norm equivalence \cite[Lemma~6.3]{BFP})}
   There is $s_0 >0$ such that for any $s \geq s_0$, for all $0< r \leq r_0(s)$ small enough, there is 
    a constant $C_{s,K}>0$ such that the following holds true:
   if $(\eta, \psi)$ is a solution of the water waves equations \eqref{ww} such that the complex function $u$ in \eqref{u0} satisfies 
   $\sup_{[-T, T]} \norm{u}_{K,s} \leq r$, then there is a constant $ C_{s,K} > 0 $ such that 
    \begin{equation}\label{BISMA2}
        \norm{\pa_t^k u(t, \cdot)}_{\dot H^{s-k}} \leq C_{s,K} \norm{u(t, \cdot)}_{\dot H^s} \ , \quad \forall 0 \leq k \leq K \ , 
        \quad \forall |t| \leq T \ .
    \end{equation}
    In particular  the norm $ \norm{u(t, \cdot)}_{K,s}$  is equivalent to the norm  $ \| u(t,\cdot) \|_{\dot H^s}$. 
    \label{lemma6.3}
\end{lem}

The estimate \eqref{BISMA2}  
follows directly from the fact that 
$ u $ solves the PDE    
\eqref{ww.paraco} in the appendix,  whose vector field loses one derivative. 

\begin{thm}\label{cor:LIP}
{\bf (Deterministic  water waves flow)} Let $N \gg K \gg s_0 \gg 1$ be the constants provided by  \Cref{BNFtheorem}.
Fix 
$\delta \in (0,1)$ and  $s \geq  N + \frac54$. 
For any $\tR > 0$, any $0<\kappa \ll 1-\delta$, there exists $\eps_0 := \eps_0(\tR, \kappa)>0$ such that for   any $\eps \in (0, \eps_0)$, the following holds true:
\begin{itemize}
\item[(i)] {\sc (Long time existence)} for any initial data $(\eta_0, \psi_0) \in \cB_0(\eps, \delta, \tR, s)$ in \eqref{cB0}, 
the solution $(\eta, \psi)(t; \eta_0, \psi_0)$ of the water waves equation \eqref{ww}   exists  up to times 
  $|t| \leq T_\eps:= \eps^{-3(1-\delta) + \kappa} $ and fulfills
      \begin{equation}
\label{est.flow.etapsi}
\sup_{t \in [-T_\eps, T_\eps]} \left( \|\eta (t; \eta_0, \psi_0) \|_{H_0^{s-1}} + \| \psi (t; \eta_0, \psi_0)\|_{\dot H^{s- 1}}\right) \leq \und{C}\,  \tR \eps^{1-\delta}  \, , 
\end{equation}
for some $\und{C}>0$ independent from $\delta,  \tR, \kappa, \eps$ and the initial datum $ (\eta_0,\psi_0)$.
    \item[(ii)]  {\sc (Birkhoff normal form map) } The  corresponding function  $u(t; \eta_0, \psi_0)$   defined in 
    \eqref{u0}, with $ \upomega $ given by  \eqref{omega0}, 
 satisfies the smallness condition \eqref{u01} 
 required by  \Cref{BNFtheorem}, 
  and so the map 
   \be
  \begin{aligned}\label{Upsilon}
  \Upsilon^t \colon    \cB_0(\eps, \delta, \tR, s) &  \to \dot H^N \ , \\
     (\eta_0, \psi_0)  &  \mapsto \Upsilon^t(\eta_0, \psi_0):= z(t; \eta_0, \psi_0):= \mathfrak{B}(u(t; \eta_0, \psi_0)) u(t; \eta_0, \psi_0)
  \end{aligned}
  \ee
  is well defined   for any $|t| \leq T_\eps$;

\item[(iii)]  {\sc (Smallness)}   For any 
$(\eta_0, \psi_0) \in \cB_0(\eps, \delta, \tR, s)$, 
the functions $u(t; \eta_0, \psi_0)$ and $z(t; \eta_0, \psi_0)$ defined in the previous item satisfy 
  \begin{equation}\label{z.control}
\sup_{t \in [-T_\eps, T_\eps]}\norm{u(t; \eta_0, \psi_0)}_{\dot H^{N}}  + \sup_{t\in [-T_{\eps},T_{\eps}]} \norm{z(t; \eta_0, \psi_0)}_{\dot{H}^N} \leq \und{C}\, \tR \eps^{1-\delta} \, 
\end{equation}
for some $\und{C}>0$ independent from $\delta, \tR, \kappa, \eps$ and the initial datum $ (\eta_0,\psi_0)$.
\item[(iv)] 
  {\sc (Lipschitz property)} 
There exist $C_1, C_{\rm lip}>0$ such that  for any 
$(\eta_{0i}, \psi_{0i}) \in \cB_0(\eps, \delta, \tR, s)$, $i=1,2$, any 
$|t| \leq T_\eps$, we have
\begin{equation}\label{lip.z}
    \| \Upsilon^t(\eta_{01}, \psi_{01}) - \Upsilon^t(\eta_{02}, \psi_{02}) \|_{L^2} \leq  C_1 e^{C_{\rm lip} |t| }\,\left( \norm{\eta_{01} -  \eta_{02}}_{H^{s_0+2}_0} + 
    \| \psi_{01} -  \psi_{02} \|_{\dot H^{s_0+2}} \right)
    \, . 
\end{equation}
\end{itemize} 
\end{thm}

\begin{proof}
$(i)$ In view of \eqref{an:DNX} and since  
the functions $V, B$ in \eqref{def:V}--\eqref{form-of-B}
vanish at $ \psi = 0 $, 
any initial datum 
$ (\eta_0, \psi_0) $ in $\cB_0(\eps, \delta, \tR, s)$ belongs to the space $X^{s-1}$ (cf.  \eqref{Xs}) and  
\be\label{est.init.data.Xs}
\norm{(\eta_0, \psi_0, V_0, B_0)}_{X^{s-1}} \lesssim  \norm{\eta_0}_{H_0^s} + \norm{\psi_0}_{\dot H^s} 
\stackrel{\eqref{cB0}} \lesssim  \tR \eps^{1-\delta}    \ .
\ee
Provided $\eps$ is sufficiently small, condition    \eqref{thm:main1} holds and Theorem \ref{thm:mainBFP} guarantees the existence of a  
unique solution $(\eta, \psi, V, B) (t,x ) $ in $ X^{s-1}$
    satisfying, by \eqref{thm:main2} and \eqref{est.init.data.Xs}, 
    \begin{equation}
\label{thm:main2app}
 \|\eta (t, \cdot) \|_{H_0^{s-1}} + \| \psi (t, \cdot)\|_{\dot H^{s- 1}} + 
\| V (t, \cdot) \|_{H^{s-2}} + \| B(t, \cdot) \|_{H^{s-2}} \leq \und{C}  \tR \eps^{1-\delta} \,  , 
 \quad 
 \forall |t| \leq c 
 \tR^{-3}  \eps^{-3(1-\delta)} \, ,
\end{equation}
for some $\und{C}>0$ independent from $\delta,  \tR, \kappa, \eps$ and the initial datum $ (\eta_0,\psi_0)$.
For  $\eps_0 :=\eps_0(\tR, \kappa)$ sufficiently small, the time 
$T_\eps = \eps^{-3(1-\delta) + \kappa}\leq c 
 \tR^{-3}  \eps^{-3(1-\delta)}$,   
concluding the proof of $(i)$.

$(ii)$
By  \eqref{an:u}, since the function $ u $ in \eqref{u0}
vanishes at $ (\eta,\psi) =(0,0) $, the assumption 
 $s -1 \geq N + \frac14$ and 
\eqref{thm:main2app}, 
 imply that 
\begin{align}
 \norm{u(t, \cdot)}_{\dot H^{N}} 
\lesssim \| \eta (t, \cdot) \|_{H_0^{N+\frac14}} + \|
 \psi (t, \cdot) \|_{{\dot H}^{N+\frac14}} 
 \leq \underline{C} \tR \eps^{1-\delta} \, , 
 \quad \forall |t| \leq T_\eps  \, , 
 \label{0809:1714}
\end{align}
proving the first bound in \eqref{z.control}
(the constant $ \underline{C} $ may be larger than in \eqref{est.flow.etapsi}).
Then 
 \Cref{lemma6.3}  and \eqref{0809:1714} imply that 
  \begin{equation}\label{u.control}
\sup_{t\in [-T_\eps,T_\eps]} \| u(t; \eta_0, \psi_0) \|_{K,N} =  
\sup_{t\in [-T_\eps,T_\eps]} \sum_{k=0}^K \|\partial_t^k u(t; \eta_0, \psi_0)\|_{\dot{H}^{N-k}} \leq C_1 \tR \eps^{1-\delta}  \ 
\end{equation}
for some $C_1 >0$ and so, for  $\eps$  small enough, 
$ u(t; \eta_0, \psi_0 ) $ 
satisfies 
 the smallness condition \eqref{u01}.
  
$(iii)$ It follows from \eqref{0809:1714} and \eqref{Germe}.

$(iv)$ 
We now prove \eqref{lip.z} for any $ t \in [0,T_\varepsilon] $. By time reversibility the same estimate holds in $[-T_\varepsilon,0] $. 
We rely on the Lipschitz estimates proved in \cite{ABZ1}.
Let $(\eta_i, \psi_i, V_i, B_i) \in C^0([-T_\eps, T_\eps], X^{s-1})$ be the solutions of \eqref{ww} with initial data $(\eta_{0i}, \psi_{0i}) $ in $ \cB_0(\eps, \delta, \tR, s)$
provided by item ($i$). 
Define 
\begin{equation}\label{def:Ns(t)}
\begin{aligned}
\triangle\eta:= \eta_1 - \eta_2 \, ,  \quad
\triangle\psi:= \psi_1 - \psi_2 \, ,  \quad
\triangle V:= V_1 - V_2 ,  \quad
\triangle B:= B_1 - B_2 \, , \\
\fM_i:= \sup_{t\in [0,T_\eps]} \| (\eta_i, \psi_i, V_i, B_i) \|_{X^{ s_0+2}}  \, , 
\quad
\fN(t):= \sup_{\tau \in [0,t]} \norm{(\triangle\eta, \triangle\psi, \triangle V, \triangle B)}_{X^{s_0+1}}\, , 
\end{aligned}
\end{equation}
with  $s_0$ given by \Cref{BNFtheorem} (which can be taken greater than $6$). 
By \cite[Formula 5.3]{ABZ1},  for any $0 \leq t \leq T_\eps$,  
\begin{equation}\label{ABZ.lip}
 \fN(t) \leq \cK(\fM_1, \fM_2) \fN(0) + t \cK(\fM_1, \fM_2) \fN(t) 
 \end{equation}
 for some non-decreasing constant $\cK(\fM_1, \fM_2)\geq 0$.
The constants
$ \mathfrak{M}_i $ in \eqref{def:Ns(t)} satisfy, in view of  
\eqref{epVBnorm}, \eqref{thm:main2app} and  $ s -2 > s_0 $,  
the uniform bound
 \be\label{Mi.est}
 \mathfrak{M}_1 + \mathfrak{M}_2 \leq \underline{C} 
 \tR \eps^{1-\delta} \leq 1 \, , 
 \ee 
 so the constant in \eqref{ABZ.lip} satisfies 
 \be\label{unifb}
 \cK(\fM_1, \fM_2)\leq \cK(1,1) \, .
\ee
  For any $t \in [0, T_\eps]$ where $ T_\eps = \eps^{-3(1-\delta) + \kappa} $ we 
divide 
  $[0, t]$ in  intervals of size $t_1$ where
$t_1$ is smaller than 
the minimum of $t$ and $1/2\cK(1,1)$. 
  Then, on $[0, t_1]$, \eqref{ABZ.lip} and \eqref{unifb} imply  
 $\fN(t) \leq 2\cK(1,1) \fN(0)$ and,  iterating such bound  between $[t_1, 2t_1],\ldots , [t -
t_1, t]$, we deduce, for some $C_{\rm lip} >0$,  
 \begin{equation}
\fN(t) \leq  e^{C_{\rm lip} t }\, \fN(0)  \, , \quad \forall  0 \leq  t \leq T_\eps \, .  
\label{lipconst}
\end{equation}
We now prove the bound \eqref{lip.z} for 
$\Upsilon^t (\eta_{01},\psi_{01}) 
- \Upsilon^t (\eta_{02},\psi_{02}) := z_1-z_2 $ where 
$ z_i:= \mathfrak{B}(u_i)u_i $ and  
$ u_i:= \tfrac{1}{\sqrt{2}} |D|^{-\frac14}\eta_i + \tfrac{\im}{\sqrt{2}}|D|^{\frac14} \upomega_i $, $ i = 1,2 $, satisfy, by  
\eqref{z.control} and $ N \geq s_0 + 1  $,  
\begin{equation}\label{massa1}
\sup_{t \in [-T_\eps, T_\eps]}  \norm{u_i(t, \cdot)}_{\dot H^{s_0+1}} 
\lesssim \tR \varepsilon^{1-\delta}  \, . 
\end{equation}
Thus, for any $ t \in [0,T_\varepsilon ] $, 
\begin{align}\label{intermedios}
    \norm{z_1-z_2}_{L^2}  \leq 
    \norm{z_1-z_2}_{\dot H^{s_0}} 
    & \stackrel{\eqref{lipB}, \eqref{massa1}}\lesssim
    \norm{u_1-u_2}_{\dot H^{s_0}}  \notag \\
 &    \stackrel{\eqref{an:u},\eqref{Mi.est}} \lesssim \| (\triangle \eta, \triangle \psi) \|_{H^{s_0+\frac14} \times {\dot H}^{s_0+ \frac14}} \stackrel{\eqref{def:Ns(t)},\eqref{epVBnorm}}{\lesssim} 
     \mathfrak{N}(t) \notag  \\
     & \stackrel{\eqref{lipconst},\eqref{def:Ns(t)}} 
      {\lesssim}  e^{C_{\rm lip} t }    
     \norm{(\triangle\eta (0), \triangle\psi(0), \triangle V(0), \triangle B(0))}_{X^{s_0+1}} \notag \\
     & \stackrel{\eqref{an:DNX}, \eqref{Mi.est},\eqref{epVBnorm}} 
      {\lesssim}  e^{C_{\rm lip} t }  \big(   
 \| \eta_{01} -  \eta_{02} \|_{H^{s_0+2}} +
\| \psi_{01} -  \psi_{02} \|_{\dot H^{s_0+2}} \big)  \notag
\end{align}
proving \eqref{lip.z}.
\end{proof}

\part{Probabilistic formation of rogue  water waves}\label{partII}

In this Part we prove the Large Deviation
\Cref{thm:main} 
which  directly implies \Cref{thm:main_intro}, as we show in 
Section \ref{sec:main}.
In Sections 
\ref{sec:52} and  \ref{sec:LDPIII}, relying on the normal form
long time well-posedness 
results we construct better approximate water waves solutions, and 
we prove the Large Deviation estimate \eqref{eq:LDP}, 
up to the  much longer times 
$ |t| \leq \eps^{-\frac52 (1-\delta) + } $, 
and  $ |t| \leq \eps^{- 3 (1-\delta) + } $, 
as stated  in items ($i$) and ($ii$). 
The proofs 
are rather different in these two cases.
Finally, in Section \ref{sec:allr} we prove the dispersive focusing mechanism in \Cref{thm:df}.

\section{The main theorem}\label{sec:main}

The following  theorem  provides the sharp estimate 
of  the probability that 
random initial data of typical size $ \varepsilon $ 
will lead to a 
rogue water wave  with a crest  
$ \geq \lambda_0 \varepsilon^{1-\delta} $. 

\begin{thm}{\bf (Large deviation principle)}\label{thm:main}
There exists $\und{s}>0$ 
such that
for any  $\lambda_0 >0$, $\delta \in (0,1)$, and $s \geq \und{s}$ there exists $\tR_0  >0$ and  for any $\tR\geq \tR_0$,   any $0<\kappa \ll 1-\delta$, the following holds true. There exists $\eps_0 > 0 $ such that for   any $\eps \in (0, \eps_0)$, 
for any random initial data $(\eta_0^\omega, \psi_0^\omega) \in  \cB_0(\eps, \delta, \tR , s)$ in \eqref{cB0}, 
the solution $(\eta(t,x), \psi(t,x))$ of the water waves equations \eqref{ww} exists up to times $|t| \leq \eps^{-3(1-\delta)+\kappa}$, and satisfies 
the large deviation principle 
\begin{equation}\label{eq:LDP_main}
\lim_{\varepsilon\rightarrow 0^{+}} \varepsilon^{2\delta}\, \log \P \left( \left\{\sup_{x\in\T}  \eta (t,x) \geq \lambda_0 \varepsilon^{1-\delta}\right\} \cap \cB_0(\eps, \delta, \tR , s) \right) = -\frac{\lambda_0^2}{2\bsi^2}  \, , 
\end{equation}
with $\bsi$ in \eqref{upsigma}, and for any time $t$ satisfying 
\begin{itemize}
    \item[(i)]  $|t| \leq \eps^{-\frac52 (1-\delta) + \kappa}$; 
    \item[(ii)]   $|t| \leq \eps^{-3 (1-\delta) + \kappa}$  provided $\delta \in (\frac35,1)$.
\end{itemize}
\end{thm}

The rest of this section is dedicated to the proof of  \Cref{thm:main_intro} assuming \Cref{thm:main}. 

We first remark that the  assumption 
$ c_{n} = d_{n} \sqrt{n} $ in \eqref{ckdk} implies the following lemma. 

\begin{lem}\label{lem:zeta0.d}
{\bf (Random initial data)}
    Let $(\eta_0^\omega, \psi_0^\omega)$ be a random initial datum as in \eqref{eq:init_data}. Then the Fourier coefficients $(\zeta_{0j}^\omega)_{j \in \Z\setminus\{0\}}$ 
    of the complex initial datum $\zeta_0^\omega (x) $ defined in \eqref{zeta0} are independent 
    complex Gaussian random variables 
\be\label{zeta0.g}
    \zeta_{0j}^\omega  \sim \cN_\C\Big(0, \eps^2 \frac{c_j^2}{ |j|^{\frac12}}\Big) \, , \quad \forall j \neq 0  \,  , 
\ee
where $c_j $ are the constants defined in \eqref{ckdk}. 
    In particular the modulus 
$(|\zeta_{0j}^\omega|)_{j \in \Z\setminus\{0\}}$ are independent Rayleigh r.v.  $\sim\sR\big(\frac{1}{\sqrt2} \eps  |j|^{-\frac14}c_j\big)$.
\end{lem}

\begin{proof}
In view of \eqref{zeta0p} and \eqref{eq:init_data}, 
 the complex initial datum $\zeta_0^\omega (x) $ in \eqref{zeta0} is equal to 
\be\label{eq:rnsn}
\zeta_0^\omega (x) =\frac{1}{\sqrt{2\pi}}  \sum_{j\in \Z\setminus\{0\}}  \underbrace{\eps \frac{r^\omega_j + \im s^\omega_j}{\sqrt{2}}}_{= \zeta_{0j}^\omega } \,  e^{\im j x } 
 \,  \quad \mbox{ where } \quad 
 \begin{cases}
r_{\pm n}^\omega  =  ( c_n n^{-\frac14}\, \alpha_n^\omega \pm  d_n n^{\frac14} \, \delta_n^\omega)/\sqrt{2}\\
s_{\pm n}^\omega  =  ( d_n n^{\frac14} \, \gamma_n^\omega \mp c_n n^{-\frac14}\, \beta_n^\omega)/\sqrt{2} \, , 
\end{cases}         \forall n \in \N \, . 
\ee
By assumption \eqref{eq:init_data}, $(\alpha^\omega_n)_{n\in\N}$, $(\beta_n^\omega)_{n\in\N}$, $(\gamma_n^\omega)_{n\in\N}$, and $(\delta_n^\omega)_{n\in\N}$ are i.i.d.  real Gaussian random variables  $\sim\cN_\R \left(0, 1 \right)$ and, by the assumption 
$ c_{n} = d_{n} \sqrt{n} $ in \eqref{ckdk},
  each $r^\omega_{\pm n}$ and $s^\omega_{\pm n}$ is a real
 Gaussian distribution  $\sim \cN_\R \Big(0,\dfrac{c_n^2}{  n^{\frac12}}\Big)$.
 Thus each Fourier coefficient 
 $ \zeta_{0j}^\omega  $ satisfies \eqref{zeta0.g}. 
Regarding their independence, note that 
$$
r_j^\omega  \independent  s_{j'}^\omega \ \ \  \forall j, j' \in \Z\setminus\{0\} , \quad
r_j^\omega  \independent r_{j'}^\omega \mbox{ and } s_j^\omega  \independent  s_{j'}^\omega  \ \ \ \forall |j| \neq |j'| \, . 
$$
In addition, the Gaussian variables 
$r_n^\omega$ and $r_{-n}^\omega$  are independent if and only if
\[
0= 2 \E[r_n^\omega \,  r_{-n}^\omega] =  n^{-\frac12} \, c_n^2 \E[ \alpha_n^2] - n^{\frac12} \,d_n^2 \E[ \delta_n^2] 
\quad\Leftrightarrow\quad c_n = n^{\frac12} \, d_n \, .
\] 
The  same condition implies the independence between $s_n^\omega$ and $s_{-n}^\omega$. This proves that all the 
$(\zeta_{0j}^\omega)_{j \in \Z\setminus\{0\}}$  are  independent. 
\end{proof}

\begin{proof}[Proof of \Cref{thm:main_intro} assuming \Cref{thm:main}]
We now  fix $\lambda_0>0$, $\delta\in (0,1)$, $0<\kappa\ll 1-\delta$ and a time $t$ satisfying \emph{(i)} or \emph{(ii)} in \Cref{thm:main_intro}.  Define the event
 \be\label{def:At}
 \mathcal{A}^t(\lambda) :=\Big\{ \omega\in\mathbb{\Omega} \ \big |\ \sup_{x\in\T} \eta^{\omega}(t,x) \geq \lambda \Big\} \, .
 \ee
For any initial datum in the ball 
$ \cB_0 (\eps, \delta, \tR , s) $
with $s>\und{s}$ given by \Cref{thm:main}, 
the solution $(\eta,\psi)$ of the water waves equations \eqref{ww} is well defined and therefore
 \begin{equation}\label{A_equals_R}
 \mathcal{A}^t(
\lambda_0 \varepsilon^{1-\delta})\cap \cB_0(\eps, \delta, \tR , s)=\mathfrak{R}^t(
\lambda_0 \varepsilon^{1-\delta})\cap \cB_0(\eps, \delta, \tR , s) 
 \end{equation}
 where  the set $\mathfrak{R}^t (\lambda) $ is defined in \eqref{eventr}.
 \\[1mm]
\noindent{\sc Lower bound:} 
In view of  \eqref{A_equals_R} we have 
 \begin{equation}\label{eq:lb_main}
     \P( \mathfrak{R}^t(
\lambda_0 \varepsilon^{1-\delta}))\geq \P(\mathfrak{R}^t(
\lambda_0 \varepsilon^{1-\delta})\cap \cB_0(\eps, \delta, \tR , s))= \P(\mathcal{A}^t(
\lambda_0 \varepsilon^{1-\delta})\cap \cB_0(\eps, \delta, \tR , s)) \, .
 \end{equation}
The estimates \eqref{eq:lb_main} and \eqref{eq:LDP_main} then imply 
\begin{equation}\label{eq:lb_main2}
    \liminf_{\varepsilon\to 0^{+}} \varepsilon^{2\delta}\, \log \P( \mathfrak{R}^t(
\lambda_0 \varepsilon^{1-\delta})) \geq -\frac{\lambda_0^2}{2\bsi^2} \, .
\end{equation}
\noindent{\sc Upper bound:} 
We need the following  lemma which shows that
the complementary event 
$ \cB_0(\eps, \delta, \tR , s)^c = 
\mathbb{\Omega}\setminus \cB_0(\eps, \delta, \tR , s) $ where 
$\cB_0 $ is defined in \eqref{cB0p}, is {\it negligible}.

\begin{lem}\label{lem:negligible} Let $\delta\in (0,1)$, and $\tR,s,\eps>0$. Then the probability  $ \P( \cB_0(\eps, \delta, \tR , s)^c)$ in \eqref{cB0p} is bounded by:
\begin{equation}
\label{terrone}
\log \P(\cB_0(\eps, \delta, \tR , s)^c) \leq  -  \frac{\tR^2}{ 4 
\| \vec c \|_{h^s}^2 } \varepsilon^{-2\delta } + 1 \ , \qquad 
\mbox{ with } \  
\| \vec  c \|_{h^s}  \  \mbox{defined in } \ \eqref{norm.c} \, . 
\end{equation} 
\end{lem}
\begin{proof} 
By  \eqref{equivalnorms} and \eqref{usobo},
\begin{equation}\label{B0_to_exp}
\norm{(\eta_0^{\omega},\psi_0^{\omega})}_{H_0^s\times {\dot H}^{s+\frac12}}^2  = 2 \norm{\zeta_0^{\omega}}_{H^{s+\frac14}}^2 = 2\,\varepsilon^2\, \sum_{j\in\Z\setminus\{0\}} X_j^\omega  \qquad \mbox{where}\ X_j^{\omega}:= c_j^2 |j|^{2s+\frac12}\, \varepsilon^{-2}\, |\zeta_{0j}^{\omega}|^2.
\end{equation}
Since $\zeta_{0j}^{\omega}\sim \cN_{\C}\big (0,\varepsilon^2 \frac{c_j^2}{|j|^{\frac12}}\big )$ (cf. \eqref{zeta0.g}), we have that $X_j^{\omega}\sim \Es \big (\frac{1}{c_j^2|j|^{2s}}\big )$ by \Cref{basic_prob} (v). We have to bound the tails of a sum of exponential random variables. Note that 
\begin{equation}\label{mgen_func}
    \E [e^{\rho X_j^{\omega}}] = 1+\frac{c_j^2|j|^{2s} \rho}{1-c_j^2|j|^{2s} \rho},\qquad \mbox{for any}\ \rho<\frac{1}{c_j^2|j|^{2s}}.
\end{equation}
By the Chernoff inequality, the independence of $(X_j^{\omega})_{j\in\Z\setminus\{0\}}$ and \eqref{mgen_func}, for any $0<\rho<\frac{1}{\norm{\vec{c}}_{h^s}^2}$ (cf. \eqref{norm.c})
\begin{align}
\P\Big( \sum_{j\in\Z\setminus\{0\}} X_j^{\omega} \geq \lambda \Big) & \leq e^{-\rho\lambda}\, \E [e^{ \rho\sum_{j\in\Z\setminus\{0\}} X_j^{\omega} }]  \leq e^{-\rho\lambda} \prod_{j\in\Z\setminus\{0\}} \Big( 1+\frac{c_j^2|j|^{2s} \rho}{1-c_j^2|j|^{2s} \rho}\Big)  \notag \\
& =\exp\Big(-\rho\lambda +{\mathop \sum}_{j\in\Z\setminus\{0\}} \log 
\Big( 1+\frac{c_j^2|j|^{2s} \rho}{1-c_j^2|j|^{2s} \rho}\Big)\Big) \, . \label{inter310}
\end{align}
Setting $\rho:= (2\norm{\vec{c}}_{h^s}^2)^{-1}$, we have $1-c_j^2|j|^{2s} \rho\geq 1/2$ for all $j\in\Z\setminus\{0\}$ and 
\eqref{inter310} and $\log (1+x)\leq x$
imply 
\begin{equation}\label{ciochevoglio}
    \P\Big( \sum_{j\in\Z\setminus\{0\}} X_j^{\omega} \geq \lambda \Big) \leq \exp\Big( -\frac{\lambda}{2\norm{\vec{c}}_{h^s}^2}+\sum_{j\in\Z\setminus\{0\}} \frac{c_j^2|j|^{2s}}{\norm{\vec{c}}_{h^s}^2}\Big) = 
    \exp\Big(  1- \frac{\lambda}{2\norm{\vec{c}}_{h^s}^2}\Big) \, .
\end{equation}
Recalling \eqref{cB0}
and \eqref{B0_to_exp} we have 
\begin{align}
\P(\cB_0(\eps, \delta, \tR , s)^c) & = \P \left( \norm{(\eta_0^{\omega},\psi_0^{\omega})}_{H_0^s\times \dot{H}^{s+\frac12}}^2 >\tR^2\, \eps^{2-2\delta}\right) \notag \\
&\leq \P\Big(  \sum_{j\in\Z\setminus\{0\}} X_j^{\omega} \geq \frac{\tR^2\, \eps^{-2\delta}}{2}\Big)\leq 
\exp\Big( 1- \frac{\tR^2 \eps^{-2\delta}}{4\norm{\vec{c}}_{h^s}^2}\Big) \notag 
\end{align}
by the estimate \eqref{ciochevoglio} 
with  $\lambda = \tR^2 \varepsilon^{-2\delta}/2$.  
\end{proof}

\begin{rk}
The optimal bound  is 
$ \log \P(\cB_0^c) \leq  -  \frac{\tR^2}{ 
2\| \vec c \|_{h^s}^2 } \varepsilon^{-2\delta } + 1 $,
as  proved in 
\cite[Theorem~5.1]{Janson}. 
\end{rk}

We can now conclude the proof of the upper bound. 
Taking  $\tR > 2 \lambda_0 \| \vec c \|_{h^s}  \bsi^{-1} $ 
and  $\eps$ sufficiently small, 
the probability  $ \P( \cB_0(\eps, \delta, \tR , s)^c)$ is much smaller than \eqref{eq:LDP_main}, in particular 
\begin{equation}\label{eq:B0c_control}
    \P( \cB_0(\eps, \delta, \tR , s)^c) \leq \P( \mathcal{A}^t(
\lambda_0 \varepsilon^{1-\delta})\cap \cB_0(\eps, \delta, \tR , s)) \, .
\end{equation}
Using \eqref{A_equals_R} and \eqref{eq:B0c_control}, we find
 \begin{align}
     \P( \mathfrak{R}^t(
\lambda_0 \varepsilon^{1-\delta})) & = \P( \mathfrak{R}^t(
\lambda_0 \varepsilon^{1-\delta})\cap \cB_0(\eps, \delta, \tR , s)) + \P( \mathfrak{R}^t(
\lambda_0 \varepsilon^{1-\delta})\cap \cB_0(\eps, \delta, \tR , s)^c)\notag \\
& \leq  \P( \mathcal{A}^t(
\lambda_0 \varepsilon^{1-\delta})\cap \cB_0(\eps, \delta, \tR , s)) + \P( \cB_0(\eps, \delta, \tR , s)^c)\notag\\
& \stackrel{\eqref{eq:B0c_control}} \leq  2 \P( \mathcal{A}^t(
\lambda_0 \varepsilon^{1-\delta})\cap \cB_0(\eps, \delta, \tR , s)) \, .
\label{eq:ub_main}
 \end{align}
The estimates \eqref{eq:ub_main} and 
\eqref{eq:LDP_main} then imply 
\begin{equation}\label{eq:ub_main2}
    \limsup_{\varepsilon\to 0^{+}} \varepsilon^{2\delta}\, \log \P( \mathfrak{R}^t(
\lambda_0 \varepsilon^{1-\delta})) \leq -\frac{\lambda_0^2}{2\bsi^2}.
\end{equation}
In conclusion, 
the bounds \eqref{eq:lb_main2} and \eqref{eq:ub_main2} yield the desired result \eqref{eq:LDP_main}.
\end{proof}

\section{Rogue waves on nonlinear  timescales}
\label{sec:52} 

In this section we prove Theorem \ref{thm:main}
for any time $ |t| \leq \eps^{-\frac52 (1-\delta) + \kappa} $ in item ($i$). 
The key step is to define an approximate profile  $\eta_{\rm app}(t,x)$  which stays close to the  water waves 
surface $\eta(t,x)$ over this time 
interval, and whose distribution is 
Gaussian,  
see properties ($i$)-($ii$) below. 

\subsection{The approximate solution} 

Let us first define  the following linear operator.  
Given $T>0$, $\ts \geq 0$ and a function $ g\in C^0 ([0,T],{ \dot H^1}(\T,\C))$, $ \tau \mapsto g(\tau) $,  let 
\begin{equation}\label{eq:def_upxi}
\begin{split}
\Upxi (t,g) : \dot H^\ts  & \longrightarrow  C^0([0,T], \dot H^\ts )\\
 v & \longmapsto \Upxi (t,g)[v](x):  = \frac{1}{\sqrt{2\pi}}\,\sum_{j \in\Z\setminus \{0 \}} e^{-\im \int_0^t  \cL_j(I(g(\tau)))\, \di\tau} \,v_j  \,e^{\im jx}\, ,
\end{split}
\end{equation}
where  $\cL_j (I)$ are the nonlinear frequencies in \eqref{HCS1},  and 
\begin{equation}
\label{action}
I(g(\tau))=  \big(|g_k(\tau)|^2 \big)_{k\in\Z\setminus \{0\}} 
\end{equation}
are the actions of $g(\tau ) \in \dot H^1 (\T) $,
cf.  \eqref{actions}, and $ g_k (\tau) $ denotes its $k$-th Fourier coefficient.  
For each $t$ fixed each 
 $ \cL_j (I) $ is finite by \eqref{L_lips} since $ g(\tau) \in \dot H^1 (\T) $. Furthermore  
 $\Upxi (t,g)$ is a Fourier multiplier and, since 
 all the $\cL_j (I)$ are real, 
it is a  unitary transformation from $\dot H^\ts$ to itself. 

The approximation of the free water waves surface profile  $\eta(t,x)$ with initial datum 
$ (\eta_0^\omega, \psi_0^\omega )$ that we shall use is
\begin{equation}
\label{etaapptilde}
\eta_{\rm app}(t,\cdot) := \sqrt{2} \Re \,  |D|^{\frac14}  \, u_{\rm app}(t,\cdot) \ , 
\qquad u_{\rm app}(t,x):=   \Upxi(t, \zeta_0^\omega) \zeta_0^\omega = \frac{1}{\sqrt{2\pi}}\,\sum_{j \in \Z \setminus\{0\}} e^{- \im t \cL_j(I(\zeta_0^\omega))} \, \zeta_{0j}^\omega  \, e^{\im j x }\, , 
\end{equation}
where  $\zeta_0^\omega (x) $ is the initial datum  in \eqref{zeta0}. 
Note that $ \eta (t,x)$ 
and $ \eta_{\rm app}(t,x) $
have the same initial datum $ (\eta_0^\omega, \psi_0^\omega)$. Remark that 
$ u_{\rm app} (t,x) =\Upxi(t,\zeta_0^\omega) \zeta_0^\omega$  is the solution 
of the integrable cubic Birkhoff 
normal form 
\eqref{BNF3} with  random initial datum $\zeta_{0}^\omega (x) $ in 
\eqref{zeta0}. It is not clear that  it is still a complex Gaussian random variable, since its phases  involve also the moduli $(|\zeta_{0j}|)_{j \in \Z\setminus \{0\}}$.
We shall prove the following key properties of the approximate solution $u_{\rm app}(t,x)$:  
\begin{itemize} 
    \item ($i$) {\bf Approximation: } the approximate profile 
    $ \eta_{\rm app}(t,\cdot )$ stays close to the  water waves surface $ \eta (t,\cdot ) $   
    in the sup-norm $ \| \  \|_{L^\infty(\T)} $, 
    up to times $|t|\leq \eps^{- \frac52 (1-\delta)+\kappa} $, see Lemma \ref{thm:approx}. 
\item ($ii$) {\bf Preservation of Gaussianity: } for any 
  $(t,x)$ the approximate water wave 
  profile $ \eta_{\rm app}(t,x)$ is a  Gaussian random variable, 
  see Lemma \ref{cor:etaapp}.
\end{itemize}
We start with the first property ($i$). 

\begin{lem}\label{thm:approx} {\bf (Approximation)}
Fix $\delta \in (0,1)$ and  $s \geq  N + \frac54$ with $ N $ given by \Cref{BNFtheorem}. 
For any $\tR > 0$, any $0<\kappa \ll 1-\delta$, there exist $\eps_0 := \eps_0(\tR, \kappa), \tC>0$ such that for   any $\eps \in (0, \eps_0)$, the following holds true.
For any initial datum 
 $(\eta_0^\omega, \psi_0^\omega) \in  \cB_0(\eps,\delta, \tR,s)$ defined in \eqref{cB0},
 the solution $(\eta(t,x), \psi_0^\omega, \psi(t,x) )$ of the water waves equations \eqref{ww} exists up to times $|t| \leq T_\eps:=\eps^{-3(1-\delta)+\kappa}$ and
 \begin{equation}\label{upxi_error20}
\norm{\eta(t, \cdot)- \eta_{\rm app}(t, \cdot)}_{L^\infty(\T)}
\leq \tC \, \tR^4 \, \varepsilon^{1-\delta+\kappa}  \ , \qquad \forall |t|\leq  \varepsilon^{-\frac52(1-\delta)
+ \kappa} \, , 
\end{equation} 
where
$ \eta_{\rm app}(t,x)$ is the approximate solution     in \eqref{etaapptilde} (which  has the same initial datum $ (\eta_0^\omega, \psi_0^\omega)$). 
\end{lem}
\begin{proof}
By \Cref{cor:LIP}, for any $(\eta_0^\omega, \psi_0^\omega) \in \cB_0(\eps, \delta, \tR, s)$ the functions 
$$ 
u(t) := u(t,x) :=u(t,x; \eta_0^\omega, \psi_0^\omega) \, , \quad 
z(t):=z(t,x) := z(t,x; \eta_0^\omega, \psi_0^\omega) = \mathfrak{B}(u(t)) u(t) 
$$ are well 
 defined  for all times $|t| \leq T_\eps$ and satisfy  \eqref{z.control}.
 We have to bound the sup-norm of 
\be\label{eq:diffe}
 \eta(t,\cdot ) - \eta_{\rm app}(t, \cdot ) 
\stackrel{\eqref{etaapptilde}, \eqref{u0}}{=} \sqrt{2} \, \Re |D|^{\frac14} (u - u_{\rm app}) \, .
\ee
The key  to estimate  $u(t, \cdot)- u_{\rm app}(t, \cdot)$ is the normal form \Cref{BNFtheorem} which reduces the problem to  comparing the distance between the normal form variable $z(t, \cdot)$ and the solution  
 $\Upxi(t, z_0)z_0$ of the integrable Birkhoff  normal form \eqref{BNF3} with initial datum $ z_0 $.  
In view of  the invertibility of $\mathfrak{B}(u)$ (see \Cref{BNFtheorem}) and setting 
$u_0(x):= u(t,x)\vert_{t=0}$ and
 $z_0(x):= z(t,x)\vert_{t=0}$, 
by the triangle inequality we have 
\begin{align} \notag
\| \eta(t) - \eta_{\rm app}(t)\|_{L^\infty(\T)} \stackrel{\eqref{eq:diffe}} 
\lesssim  & 
\norm{u(t)- u_{\rm app}(t)}_{\dot H^1} \\
 \lesssim & \norm{\mathfrak{B}^{-1}(u(t)) z(t)- z(t)}_{\dot H^1} 
\label{u-uapp1}
\\
& + \norm{z(t) - \Upxi(t, z_0)z_0 }_{\dot H^1}
\label{u-uapp2}
\\
& + \norm{\Upxi(t, z_0)z_0 - \Upxi(t, u_0) u_0 }_{\dot H^1}
\label{u-uapp3}
\\
& + \norm{\Upxi(t, u_0) u_0 - u_{\rm app}(t)}_{\dot H^1} 
\, . 
\label{u-uapp5}
\end{align}
We estimate each term separately. 
We anticipate that the worst contribution, the one giving the restriction on the times 
$|t| \leq \eps^{-\frac52 (1-\delta) + \kappa}$ in \eqref{upxi_error20}, comes from \eqref{u-uapp2}.
\\[1mm]
{\sc Estimate of \eqref{u-uapp1}}. 
By   \eqref{phi_est} and \eqref{z.control} we have, with $s_0 \gg 1$ the constant of \Cref{BNFtheorem} 
\begin{equation}\label{0209:1924}
\eqref{u-uapp1} \lesssim \norm{u(t)}_{\dot H^{s_0}} \,\norm{z(t)}_{\dot H^{s_0+1}}  \lesssim \tR^2 \eps^{2(1-\delta)}\ , \quad \forall |t| \leq T_\eps  \, .
\end{equation}
{\sc Estimate of \eqref{u-uapp3}}. 
First  we bound
\begin{equation}\label{0209:1925}
\eqref{u-uapp3} \lesssim  \norm{\Upxi(t, z_0)z_0 - \Upxi(t, u_0) z_0 }_{\dot H^1} + \norm{\Upxi(t, u_0)z_0 - \Upxi(t, u_0) u_0 }_{\dot H^1} \, . 
\end{equation}
To estimate the second term in \eqref{0209:1925}, we use that $\Upxi(t,u_0)$ is a unitary Fourier multiplier, 
hence for every $t \in \R $,
\begin{equation}
\label{0209:1212}
\norm{\Upxi(t, u_0)z_0 - \Upxi(t, u_0) u_0 }_{\dot H^1} 
=  \norm{z_0 - u_0}_{\dot H^1}
=  {\norm{ \mathfrak{B}(u_0) u_0 - u_0}_{\dot H^1}
} \stackrel{\eqref{phi_est}}{\lesssim} \norm{u_0}_{\dot H^{s_0+1}}^2
\stackrel{\eqref{z.control}}{\lesssim}
\tR^2 \eps^{2(1-\delta)}
 \ . 
\end{equation}
To  bound the first term in  \eqref{0209:1925}, we compare different phases: using the explicit expression \eqref{eq:def_upxi}, we have
\begin{align}\notag
\norm{\Upxi(t, z_0)z_0 - \Upxi(t, u_0) z_0 }_{\dot H^1} & \lesssim 
|t| \, \Big(\sum_{j \in \Z\setminus \{0\}}  j^2 \abs{\cL_j(I(z_0)) - \cL_j(I(u_0))}^2 \, 
|z_{0j}|^2 \Big)^{\frac12} \label{passinte} \\
& {\stackrel{\eqref{L_lips}}{\lesssim }
|t| \,  \norm{I(z_0)- I(u_0)}_{\cF L^{2,1}} \, \norm{z_0}_{\dot H^2}} \\
\label{0209:1213}
&{ \lesssim |t| \norm{z_0 - u_0}_{L^2} (\norm{u_0}_{\dot H^2} + \norm{z_0}_{\dot H^2})^2}
 \stackrel{\eqref{eq:normeq}}{\lesssim}
 |t|  \norm{u_0}_{\dot H^{s_0+1}}^4 
 \stackrel{\eqref{z.control}}{\lesssim}
 |t| \tR^4 \eps^{4(1-\delta)}
 \, . 
\end{align}
Thus by \eqref{0209:1925}, \eqref{0209:1212} and \eqref{0209:1213} we have that
\begin{equation}
\eqref{u-uapp3} \lesssim \tR^2 \eps^{2(1-\delta)} + |t| \tR^4 \eps^{4(1-\delta)} \lesssim \tR^4 \eps^{\frac32 (1-\delta) + \kappa} \, , \quad \forall |t| \leq  \varepsilon^{-\frac52(1-\delta)+\kappa} \, .
\end{equation}
{\sc Estimate of \eqref{u-uapp5}}. 
Using the definition 
$ u_{\rm app}(t) =   \Upxi(t, \zeta_0^\omega) \zeta_0^\omega $ in 
 \eqref{etaapptilde}, the term  \eqref{u-uapp5} is estimated similarly to 
\eqref{u-uapp3}, replacing the estimate on $\norm{z_0 - u_0}_{\dot H^1}$ with 
\be\label{u0-zeta0}
\norm{u_0 - \zeta_0^\omega}_{\dot H^1} \stackrel{ \eqref{diffuzeta}}{\lesssim }  
\norm{\opbw{(B)} \eta}_{\dot H^{\frac54}} \stackrel{\Cref{prop:action}}{\lesssim } 
\norm{(\eta_0^\omega,\psi_0^\omega)}_{H^{s_0}_0 \times \dot H^{s_0+\frac12}}^2 \lesssim \tR^2 \eps^{2(1-\delta)} \, , 
\ee
hence obtaining the same kind of estimate  
\begin{equation}\label{0209:1926}
\eqref{u-uapp5} \lesssim \tR^4 \eps^{\frac32 (1-\delta) + \kappa} \ , \qquad \forall |t| \leq  \varepsilon^{-\frac52(1-\delta)+\kappa} .
\end{equation}
{\sc Estimate of \eqref{u-uapp2}}.
Writing the equation \eqref{transformed_eq} as
$$
\dot z_k = - \im \cL_k(I(z_0)) z_{k}  - \im \left( \cL_k(I(z)) - \cL_k(I(z_0)) \right) z_k + (\cX_{\geq 4}^+)_k \ , \quad \forall k \in \Z \setminus \{0 \} \, , 
$$
and using the Duhamel formula, we obtain that 
\begin{align*}
z(t)  = & \  \Upxi(t, z_0)z_0 -  \im\, \int_0^t \Upxi (t-\tau,z_0) \Big[ \Big( \cL (I(z (\tau)))-\cL(I(z_0)) \Big) \,z(\tau) \Big] \, \di\tau \\
& + \int_0^t \Upxi (t-\tau,z_0) ({\mathcal X}^{+}_{\geq 4})(u, \bar u, z, \bar z)(\tau)\, \di\tau \, , 
\end{align*}
where we denoted the Fourier multiplier $\cL(I(z))w:= \sum_{j \neq 0} \cL_j(I(z)) w_j e^{\im j x}$.
Then, since  $\Upxi(t, \cdot)$ is a unitary Fourier multiplier and using \eqref{L_lips} (as in \eqref{passinte})
\begin{align}\notag
\eqref{u-uapp2} \lesssim &\  |t|\, \norm{I(z)-I(z_0)}_{L^{\infty}([0,t],\Fc L^{2,1})}\norm{z}_{L^{\infty}([0,t],\dot H^2)} + |t|\,\norm{ ({\mathcal X}^{+}_{\geq 4})(u, \bar u, z, \bar z)}_{L^{\infty}([0,t],\dot H^1)}\\
\notag
\stackrel{\eqref{action_est},\eqref{X_4_1}}{\lesssim } &\  |t|^2 \sup_{\tau\in [0,t]} \norm{u(\tau)}_{\dot H^{s_0}}^5\, \norm{z}_{L^{\infty}([0,t],\dot H^2)} + |t|\,\sup_{\tau\in [0,t]} \norm{u(\tau)}_{\dot H^{s_0}}^4\\
\label{0209:1927}
\stackrel{ \eqref{z.control}}{\lesssim } &\  |t|^2\, \tR^6 \,\eps^{6(1-\delta)} + |t|\,\tR^4 \, \eps^{4(1-\delta)}\lesssim \tR^6\, \eps^{1-\delta +\kappa} , \qquad \forall |t| \leq  \varepsilon^{-\frac52(1-\delta)+\kappa} .
\end{align}
Estimate \eqref{upxi_error20} follows from \cref{0209:1924,0209:1925,0209:1925,0209:1927}.
\end{proof}

\subsection{Preservation of Gaussianity}

We now prove the second important property, namely that  
$u_{\rm app}(t,x)$ is  a complex Gaussian random variable.
This is nontrivial since each nonlinear frequency 
$\cL_k(I(\zeta_0^\omega))$ in \eqref{HCS1} depends on all the {\it infinitely} many 
actions $I_j^\omega := |\zeta_{0j}^\omega |^2$
and therefore two random variables  $\cL_k(I(\zeta_0^\omega))$, $\cL_{k'}(I(\zeta_0^\omega))$ are unlikely to be independent. 
\Cref{lem:gauind} below generalizes \cite[Lemma~4.2]{GGKS} and 
\cite[Lemma~4.2]{Oh-Tzvetkov}, which deal with the much simpler case where   $\mathcal{L}_k(I)=I_k$ for any $ k $.
This diagonal structure of the normal frequencies allows for a much simpler proof of independence.    
We therefore proceed differently with a new probabilistic argument.
 We start with a  useful lemma about conditional expectation.
 
\begin{lem}\label{thm:ricarsnick} 
Let $n,m\in\N$ and  
let $X$ be a random vector in a probability space $(\mathbb{\Omega},\Fc, \P)$ taking values in $\R^n$, and let $f:\R^n\times\R^m \rightarrow \R$ be a bounded Borel-measurable function. Define the function
\[
g(y) := \E [ f(X,y)] \, , \qquad y\in\R^m \, .
\]
If $\Gc\subseteq \Fc$ is a $\sigma$-algebra such that $\sigma (X)$ is independent of $\cG$ and  $Y$ is a  $\Gc$-measurable random vector taking values in $\R^m$, then
\[ 
\E [ f(X,Y) \mid \Gc] = g(Y) \, .
\]
\end{lem}

\begin{proof}
This  is proved in \cite[(10.17)]{Resnick}.  
\end{proof}

Using  \Cref{thm:ricarsnick} we deduce the following result. 
 
 \begin{lem}\label{lem:gauind}
{\bf (Preservation of gaussianity)}
Let $ \zeta^\omega \equiv (\zeta_j^\omega)_{j \in \Z\setminus\{0\}}$ be a sequence of independent  complex Gaussian random variables 
\begin{equation}\label{eq:pres_gauss_cond}
\zeta_j^\omega \sim \cN_\C(0, \sigma_j^2) 
\quad  \text{with} \quad  \sum_{j \in\Z\setminus\{0\}}  j^2 \sigma_j^2  < \infty  \, . 
\end{equation}
Then, for any time $t\in\R$ fixed, 
$\left(\zeta_j^\omega \, e^{- \im t \cL_j(I(\zeta^\omega))}\right)_{j\in \Z\setminus\{0\}} $ where each $\cL_j(I(\zeta^\omega))$ is defined by \eqref{actions}-\eqref{HCS1},  is a sequence of independent complex Gaussians with $\zeta_j^\omega \, e^{- \im t \cL_j (I(\zeta^\omega))} \sim \cN_\C(0, \sigma_j^2)$.
\end{lem}

\begin{proof}
For each $ k \in\Z\setminus\{0\}$, we write
$\zeta_k^\omega = R_k^\omega e^{\im \phi_k^\omega}$ 
where $R_k^\omega$ is Rayleigh distributed and $\phi_k^\omega \sim \sU([0,2\pi])$ are independent, as well as independent from $(R_j^\omega)_{j\in\Z\setminus\{0,k\}}$ and $(\phi_j^\omega)_{j\in\Z\setminus\{0,k\}}$. For simplicity of notation we drop the dependence on $ \omega $. 
For fixed $t$, the random variable
$X_k :=-t \cL_k(I(\zeta))$, is a.s.~finite by condition   \eqref{eq:pres_gauss_cond}. We note that $X_k $ only depends on the sequence $(R_j)_{j\in\Z\setminus\{0\}}$, but is otherwise independent of $(\phi_j)_{j\in\Z\setminus\{0\}}$. 
Our goal is to show that
\[
\varphi_k (X_k):= \phi_k + X_k \quad \mbox{(mod}\ 2\pi)
\]
has a uniform distribution in $[0,2\pi]$, and that it is independent of $(R_j )_{j\in\Z\setminus\{0\}}$ and $(\varphi_j (X_j))_{j\in\Z\setminus\{0,k\}}$.

\medskip

\noindent {\sc Uniform distribution:} Using the tower property (cf. \Cref{rk:cond_exp} (i)), we compute the characteristic function
\begin{equation}\label{eq:phase_distribution}
\E [e^{\im \lambda \varphi_k (X_k)}] = \E  [ \E [ e^{\im \lambda \varphi_k (X_k)} \mid X_k]] = \E [ h_k (X_k)]
\end{equation}
where $h_k(r) := \E [e^{\im\lambda \varphi_k (r)}]$ and in the second equality we applied \Cref{thm:ricarsnick} with $\cG= \sigma(X_k)$, exploiting that  $X_k$ and $\phi_k $ are independent.
Using the translation invariance of the uniform distribution, it is clear that $\varphi_k (r)$ has a uniform distribution in $[0,2\pi]$ and thus $h_k (r) = h_k(0)$, which yields the desired result.

\medskip

\noindent {\sc Independence:} 
To  prove the independence of an arbitrary finite number of different r.v.~$R_{k_1},\ldots,R_{k_n}$ and $\varphi_{j_1} (X_{j_1}),\ldots,\varphi_{j_m} (X_{j_m}) $ such that $k_r\neq k_s$ if $r\neq s$, $j_r\neq j_s$ if $r\neq s$ and $j_r\neq k$ for all $r$, it is  sufficient to show the following factorization property of the characteristic function 
\begin{equation}\label{eq:char_func_factor}
    \E \left[ \exp\left( \im \lambda_0 \varphi_k (X_k) + \im \sum_{s=1}^{n} \lambda_s R_{k_s} + \im \sum_{s=1}^{m} \lambda_{s+n} \varphi_{j_s} (X_{j_s})\right)\right]=\E \left[ e^{\im \lambda_0 \phi_{k}} \right] \,  \prod_{s=1}^m \E \left[ e^{\im \lambda_{s+n} \,\phi_{j_s}} \right]  \, \prod_{s=1}^{n} \E \left[ e^{\im \lambda_{s} R_{k_s}}\right].
\end{equation}
To prove \eqref{eq:char_func_factor},
consider the sigma-algebra generated by this finite collection of random variables, namely
\[
\mathcal{G}=\sigma(R_{k_1},\ldots,R_{k_n}, X_{j_1}, \ldots , X_{j_m}) \, .
\]
Using the independence of $(\phi_j)_{j\in\Z\setminus\{0\}}$ from $\mathcal{G}$, as well as \Cref{rk:cond_exp} (i) and (ii), \Cref{thm:ricarsnick} yields
\begin{multline}\label{eq:char_func_cond}
\E \left[ \exp\left( \im \lambda_0 \varphi_k (X_k) + \im \sum_{s=1}^{n} \lambda_s R_{k_s} + \im \sum_{s=1}^{m} \lambda_{s+n} \varphi_{j_s} (X_{j_s})\right)\right] \\
= \E \left[ \E \left[ \exp\left( \im \lambda_0 \varphi_k (X_k) + \im \sum_{s=1}^{n} \lambda_s R_{k_s} + \im \sum_{s=1}^{m} \lambda_{s+n} \varphi_{j_s} (X_{j_s})\right) \Big |\, \mathcal{G} \right] \right]\\
= \E \left[ \E \left[ \exp\left( \im \lambda_0 \varphi_k (X_k) + \im \sum_{s=1}^{m} \lambda_{s+n} \varphi_{j_s} (X_{j_s})\right) \Big |\, \mathcal{G} \right] \exp\left(  \im \sum_{s=1}^{n} \lambda_s R_{k_s} \right) \right]\\
= \E \left[ h(X_k, X_{j_1}, \ldots, X_{j_m}) \, \exp\left(  \im \sum_{s=1}^{n} \lambda_s R_{k_s} \right) \right] 
\end{multline}
where
\[
h(r_0,r_1, \ldots, r_m) := \E \left[ \exp\left( \im \lambda_0 \varphi_k (r_k) + \im \sum_{s=1}^{m} \lambda_{s+n} \varphi_{j_s} (r_{j_s})\right)\right].
\]
Using the invariance of the (circular) uniform distribution under translations and the independence of $(\phi_j)_{j\in\Z\setminus\{0\}}$ among each other, we find
\begin{equation}\label{eq:h_const}
h(r_0,r_1, \ldots, r_m)=h(0,\ldots, 0)=\E \left[ e^{\im \lambda_0 \phi_{k}} \right] \,  \prod_{s=1}^m \E \left[ e^{\im \lambda_{s+n} \,\phi_{j_s}} \right]  .
\end{equation}
Using \eqref{eq:char_func_cond}, \eqref{eq:h_const}, and the independence of $(R_j)_{j\in\Z\setminus\{0\}}$, we obtain
\begin{multline*}
\E \left[ \exp\left( \im \lambda_0 \varphi_k (X_k) + \im \sum_{s=1}^{n} \lambda_s R_{k_s} + \im \sum_{s=1}^{m} \lambda_{s+n} \varphi_{j_s} (X_{j_s})\right)\right]\\
= \E \left[ h(X_k, X_{j_1}, \ldots, X_{j_m}) \, \exp\left(  \im \sum_{s=1}^{n} \lambda_s R_{k_s} \right) \right] \\
=\E \left[ e^{\im \lambda_0 \phi_{k}} \right] \,  \prod_{s=1}^m \E \left[ e^{\im \lambda_{s+n} \,\phi_{j_s}} \right]  \, \prod_{s=1}^{n} \E \left[ e^{\im \lambda_{s} R_{k_s}}\right]
\end{multline*}
thereby establishing \eqref{eq:char_func_factor}.
\end{proof}

Lemmata \ref{lem:gauind} and \ref{lem:zeta0.d} directly imply the following result.

\begin{lem}\label{cor:etaapp}
For any random initial datum $(\eta_0^\omega,\psi_0^\omega)$ as  in \eqref{eq:init_data}, for any $ t \in \R $, $ x \in \T $,  
\begin{enumerate}
\item the 
approximate solution 
$u_{\rm app}(t,x)$ defined in \eqref{etaapptilde} is
a complex Gaussian random variable;
\item the real profile 
$\eta_{\rm app}(t,x)$ in \eqref{etaapptilde} 
    is Gaussian $ \sim \cN_\R\big(0, \eps^2 \bsi^2)$ 
and it has  the form
\begin{equation}\label{eta.app.f}
\eta_{\rm app}(t,x) = \frac{\eps}{\sqrt{\pi}} \sum_{j \in \Z \setminus \{0\} } c_j R_j^\omega \, \cos(\theta_j^\omega (t) + jx) \, , 
\end{equation}
where $(R_j^\omega)_{j \in \Z\setminus \{0\}}$ are time-independent i.i.d. Rayleigh r.v. $\sim \sR(\frac{1}{\sqrt{2}})$ and $(\theta_j^\omega(t))_{j \in \Z\setminus \{0\}}$ are i.i.d. uniform r.v.  $\sim \sU([0,2 \pi])$. In addition  $R_j^\omega  \independent \theta_{j'}^\omega(t)  $ for any $ j,j' \in \Z \setminus \{0\}$.
\end{enumerate}
\end{lem}

\begin{proof}
By  \eqref{zeta0.g}, the complex Gaussian r.v. $(\zeta_{0j}^\omega )_{j \in \Z\setminus\{0\}}$ in \eqref{zeta0} fulfill the assumptions of Lemma  \ref{lem:gauind}, hence
$ \zeta_{0j}^\omega e^{- \im t \cL_j(I(\zeta_0^\omega))} \sim \cN_\C\Big(0, \eps^2\tfrac{c_j^2}{|j|^\frac12}  \Big) $. 
Thus item 1  follows by the bullets below  \Cref{def:rv}.

Formula \eqref{eta.app.f} follows from \eqref{etaapptilde} by writing 
$\zeta_{0j}^\omega e^{ - \im t 
  \cL_j (I(\zeta_0^\omega)) }  
 = |\zeta_{0j}^\omega | e^{\im \theta_j(t)}$ 
 with  
$ \theta_j^\omega (t) = \phi_j^\omega - 
 t \cL_j (I(\zeta_0^\omega) ) $, i.e.
\[
\eta_{\mathrm{app}} (t,x) = \Re \, \sqrt{2} 
|D|^{\frac14} u_{\mathrm{app}}(t,x) = 
\frac{\varepsilon }{\sqrt{\pi}}\,  \sum_{j\in\Z \setminus \{0\}}    \,  
c_j 
\underbrace{\frac{ |j|^{\frac14}}{\eps c_j}
|\zeta_{0j}^\omega|}_{=: R_j^\omega}\, \cos (\theta_j^\omega(t) + jx),
\]
where 
 $R_j^\omega$ are independent  Rayleigh r.v. $\sim \sR( \frac{1}{ \sqrt 2})$ by the last bullet 
 below \Cref{def:rv}, and $(\theta_j^\omega(t))_{j \in \Z\setminus \{0\}}$ are i.i.d. uniform r.v.  $\sim \sU([0,2 \pi])$ by \Cref{lem:gauind}.
\end{proof}

In order to derive Large Deviation Principles involving Rayleigh-distributed random variables, 
the following result in \cite[Section 2]{GGKS}  is useful. 

\begin{lem}[\cite{GGKS}]\label{thm:LDP_Rayleigh}
Suppose that $(R_j^\omega)_{j\in\Z\setminus\{0\}}$ are i.i.d. Rayleigh r.v. $\sim \sR(1/\sqrt{2})$, and let 
  $(c_j)_{j \in \Z\setminus\{0\}} $ be the sequence in \eqref{ckdk}. 
Then for any fixed $\lambda_0, \delta >0$,  we have that 
\[
\lim_{\epsilon\rightarrow 0^{+}} \epsilon^{2\delta} 
\log \P \Big( \sum_{j\in\Z\setminus\{0\}}
c_j R_j^\omega \geq \lambda_0 \epsilon^{-\delta}\Big) 
= - \frac{\lambda_0^2}{\sum_{j\in\Z\setminus\{0\}} c_j^2} \ .
\]
\end{lem}

Using Lemmata \ref{cor:etaapp} and \ref{thm:LDP_Rayleigh}, we are ready to prove our main theorem for timescales $|t|\lesssim \eps^{-\frac52 (1-\delta)+\kappa}$.

\paragraph{Proof of Theorem \ref{thm:main}-($i$).}
\label{proof(i)} Let $s \geq \und{s} := N+\frac54$ (with $ N $ given by \Cref{BNFtheorem}) and
\be\label{R0}
\tR  > \max \left( \lambda_0  \underline{C}^{-1} ,\  2 \lambda_0 \| \vec c \|_{h^s}  \bsi^{-1}
\right):= \tR_0
\ee
where $\underline{C}$ is as in \eqref{est.flow.etapsi}. 
The first condition guarantees that the  
 $\left\{\sup_{x\in\T}  \eta (t,x) \geq \lambda_0 \varepsilon^{1-\delta}\right\} \cap \cB_0(\eps, \delta, \tR, s) $ is empty, the second one is used in \eqref{1909:1326}.
\\[1mm]
By \Cref{thm:approx}, 
provided $\eps$ is small enough, 
 any initial datum $(\eta_0^\omega, \psi_0^\omega) \in \cB_0(\eps, \delta, \tR, s)\equiv \cB_0$  gives rise to a  water waves solution defined 
 for any  $|t| \leq T_\eps=\eps^{-3(1-\delta) + \kappa} $, whose wave profile 
is uniformly approximated by 
$ \eta _{\rm app}(t,x)$ defined in  \eqref{etaapptilde}, as stated in \eqref{upxi_error20}. 
By  the triangle inequality 
\begin{multline}\label{sandwich}
\sup_{x\in \T} \eta_{\rm app}(t,x) - \norm{\eta(t)-\eta_{\mathrm{app}}(t)}_{L^\infty(\T)} \leq \sup_{x \in \T} \eta(t,x) \leq \sup_{x \in \T} \eta_{\rm app}(t,x) + 
\norm{\eta(t)-\eta_{\mathrm{app}}(t)}_{L^\infty(\T)} 
\end{multline} 
so  the wave surface $ \eta(t,x) $ becomes large
if the approximate profile 
$\eta_{\rm app}(t,x)$ does. 

For any $ \lambda >0 $ we define the events
\be\label{eq:events}
\begin{aligned}
\Ac_{\rm app} (\lambda) & :=  \Big \{\sup_{x\in\T}\eta_{\rm app}(t,x)  \geq \lambda \Big\} \cap \cB_0  \, , \qquad 
\Cc  (\lambda)  := \left\{ \norm{\eta(t)-\eta_{\mathrm{app}}(t)}_{L^\infty(\T)} \geq \lambda \right\} \, ,  \\
\Dc_{\pm} (\lambda) & := \Big\{ \sup_{x\in\T} \eta_{\rm app}(t,x)  \pm \norm{\eta(t)-\eta_{\mathrm{app}}(t)}_{L^\infty(\T)}  \geq \lambda \Big\} \cap \cB_0 \, . 
\end{aligned}
\ee
In the sequel we denote $ \cC(\lambda)^c := \mathbb{\Omega} \setminus \cC (\lambda) $ the complementary event.  We fix 
\be\label{undC}
\und{\tC} := 2 \tC \lambda_0^{-1}  \quad \text{where} \ \tC >0 \mbox{ is defined in estimate \eqref{upxi_error20}} \, . 
\ee
\underline{Upper bound:}
Using the second inequality in \eqref{sandwich}, and recalling \eqref{eq:events}, we get 
\begin{align}
\P \Big (\Big\{ \sup_{x\in\T} \eta (t,x) \geq \lambda_0 \varepsilon^{1-\delta} \Big \} \cap \cB_0 \Big)  \leq &\ 
\P \left( \Dc_{+} \left({\lambda_0 \varepsilon^{1-\delta}}\right)\right)\nonumber\\
= &\ \P \left( \Dc_{+} \left({\lambda_0 \varepsilon^{1-\delta}}\right)
 \cap 
 \Cc \left(  {\und{\tC}\,\tR^6 \,\lambda_0\varepsilon^{1-\delta+\kappa}}\right)^c\, \right) \nonumber\\
 & + \P \left(\Dc_{+}
 \left({\lambda_0 \varepsilon^{1-\delta}}\right) \cap 
 \Cc \left( \und{\tC}\,\tR^6 {\lambda_0\varepsilon^{1-\delta+\kappa}}\right)\right) \nonumber\\
 \leq &\ 
 \P \left( \Ac_{\rm app} \left({\lambda_0\varepsilon^{1-\delta}}\, 
 (1- \und{\tC}\,\tR^6 \, \varepsilon^{\kappa})\right)\right) +
 \P \left(\Cc \left(\und{\tC}\,\tR^6{\lambda_0\varepsilon^{1-\delta+\kappa}}\right)\right) \notag \\ 
& \stackrel{\eqref{terrone}} \leq  \P \left( \Ac_{\rm app} \left({\lambda_0\varepsilon^{1-\delta}}\, 
 (1- \und{\tC}\,\tR^6 \, \varepsilon^{\kappa})\right)\right)  + 
\exp\Big( -\frac{\tR^2\,\varepsilon^{-2\delta}}{4\norm{\vec c}_{h^s}^2}+1 \Big)  \label{eq:B}
\end{align}
because by \Cref{thm:approx} and \eqref{undC}, for any $|t|\leq  \varepsilon^{-\frac52(1-\delta)+\kappa} $ we have the inclusion  
$$
\Cc \left(\und{\tC}\,\tR^6 {\lambda_0\varepsilon^{1-\delta+\kappa}}\right) \subseteq \cB_0(\eps, \delta, \tR, s)^c \, .
$$
We now prove that the first term in \eqref{eq:B} is dominant.  Indeed, 
 by \eqref{eta.app.f}, and recalling \eqref{eq:events},  we get 
 \begin{align}
 \P \left( \Ac_{\rm app} \left(\lambda_0\varepsilon^{1-\delta}\, (1- \und{\tC}\, \tR^6  \,\varepsilon^{\kappa})\right) \right) 
& \leq \P \Big( \sum_{j\in\Z\setminus\{0\}} c_j R_j^\omega \geq \sqrt{\pi}\lambda_0\varepsilon^{-\delta}(1-\und{\tC} \tR^6 \,\varepsilon^{\kappa})\Big) \notag \\
& \leq \exp\Big( {-  \frac{ \varepsilon^{-2 \delta}  \lambda_0^2}{2\bsi^2} + o(\varepsilon^{-2\delta}) } \Big) \label{piccolastima}
\end{align}
by \Cref{thm:LDP_Rayleigh} with $\epsilon \leadsto \eps (1-\und{\tC} \tR^6 \,\varepsilon^{\kappa})^{-1/\delta}$, since each $  R_j^\omega\sim \sR(1/\sqrt{2})$, with $\bsi$ in \eqref{upsigma}.

In conclusion \eqref{eq:B} and  \eqref{piccolastima} imply that, 
since $\tR > 2 \lambda_0 \| \vec c \|_{h^s}  \bsi^{-1} $, for $ \varepsilon \to 0^+ $,   
\be\label{1909:1326}
\P \Big(\Big\{ \sup_{x\in\T} \eta (t,x) \geq \lambda_0 \varepsilon^{1-\delta} \Big\} \cap \cB_0 \Big)  
\leq  \exp\Big( - \varepsilon^{-2 \delta}  \frac{\lambda_0^2}{2\bsi^2}
+ o(\eps^{-2\delta})\Big) \, 
\ee
and we deduce the upper bound estimate  
in \eqref{eq:LDP} 
$$
\limsup_{\eps \to 0^+}  \eps^{2\delta} \log \P \Big(
\Big\{ \sup_{x\in\T} \eta (t,x) \geq \lambda_0 \varepsilon^{1-\delta}\Big\} \cap \cB_0\Big) \leq  -\frac{\lambda_0^2}{2\bsi^2} \, , \quad \bsi\ \mbox{in \eqref{upsigma}.}
$$
\underline{Lower bound:} Using the first inequality in \eqref{sandwich}, and recalling \eqref{eq:events}, we estimate 
\begin{align}
\P \Big(\Big\{ \sup_{x\in\T} \eta (t,x) \geq \lambda_0 \varepsilon^{1-\delta} \Big\} \cap \cB_0 \Big)  
\geq &\ \P \left( \Dc_{-} \left({\lambda_0 \varepsilon^{1-\delta}}\right)\right)  \label{lb52} \\
\geq  &\ \P \left( \Dc_{-} \left({\lambda_0 \varepsilon^{1-\delta}}\right)
 \cap 
 \Cc \left({\und{\tC}\,\tR^6 \lambda_0\varepsilon^{1-\delta+\kappa}}\right)^c\, \right) \nonumber \\
\geq &\ \P \Big( \Ac_{\rm app} \left({\lambda_0\varepsilon^{1-\delta}}\, (1+\und{\tC}\tR^6 \,\varepsilon^{\kappa})\right) \cap \Cc \left(\und{\tC}\,\tR^6 {\lambda_0\varepsilon^{1-\delta+\kappa}}\right)^c\Big) \notag  \\
 \geq &\ \P \left( \Ac_{\rm app} \left({\lambda_0\varepsilon^{1-\delta}}\, (1+\und{\tC}\tR^6 \,\varepsilon^{\kappa})\right)\right) - \P \left(\Cc \left(\und{\tC}\,\tR^6 {\lambda_0\varepsilon^{1-\delta+\kappa}}\right)\right) \notag \\ 
 \stackrel{\eqref{eq:events},\eqref{terrone}}  \geq &\  \P \Big( \sup_{x\in\T}  \eta_{\rm app} (t,x) \geq \lambda_0\varepsilon^{1-\delta} (1+\und{\tC} \,\tR^6 \varepsilon^{\kappa} ) \Big)   - 
\exp\Big( -\frac{\tR^2\,\varepsilon^{-2\delta}}{4\norm{\vec c}_{h^s}^2}+1 \Big) \, . \notag 
\end{align}
By \Cref{cor:etaapp} we get 
$\varepsilon^{-1} \eta_{\rm app}(t,0) \sim \cN_{\R}(0,\bsi^2)$  with  $\bsi$ in \eqref{upsigma}, 
and thus, for $\eps$ sufficiently small,   
\begin{align}\notag
 \P & \Big(  \sup_{x\in\T}  \eta_{\rm app} (t,x) \geq \lambda_0\varepsilon^{1-\delta} (1+\und{\tC} \,\tR^6 \varepsilon^{\kappa} ) \Big) 
  \geq 
\P \left( \eta_{\rm app}(t,0) 
\geq \lambda_0 \varepsilon^{1-\delta} (1+\und{\tC} \,\tR^6 \varepsilon^{\kappa} ) \right) \\
&\qquad  = 
\P \left( \varepsilon^{-1} 
\eta_{\rm app}(t,0)  \geq \lambda_0 \varepsilon^{-\delta} (1+\und{\tC} \,\tR^6 \varepsilon^{\kappa} ) \right) 
 \stackrel{\eqref{asintail}} \gtrsim 
\frac{\varepsilon^\delta  \bsi }{ \lambda_0\, \sqrt{2 \pi}}
\,\exp\Big( -\frac{\lambda_0^2\varepsilon^{-2\delta}}{2\bsi^2} (1+ \und{\tC} \,\tR^6 \varepsilon^{\kappa})^{2}
\Big).
\label{ULTIDAV}
\end{align}
The estimates \eqref{lb52}, \eqref{ULTIDAV}  imply that, provided $\tR > 2 \lambda_0 \| \vec c \|_{h^s}  \bsi^{-1} $ 
$$
\P \Big(\Big\{ \sup_{x\in\T} \eta (t,x) \geq \lambda_0 \varepsilon^{1-\delta} \Big\} \cap \cB_0 \Big)  
\geq   \exp\Big( - \varepsilon^{-2 \delta}  \frac{\lambda_0^2}{2\bsi^2}
+ o(\eps^{-2\delta})
\Big) \, \quad \bsi\ \mbox{in \eqref{upsigma}} \, , 
$$
and we deduce the lower bound  estimate in \eqref{eq:LDP} 
$$
\liminf_{\eps \to 0^+}  \eps^{2\delta} \log \P \Big(
\Big\{ \sup_{x\in\T} \eta (t,x) \geq \lambda_0 \varepsilon^{1-\delta}\Big\} \cap \cB_0\Big) \geq  -\frac{\lambda_0^2}{2\bsi^2} \, . 
$$

\section{Rogue waves on  optimal timescales}\label{sec:LDPIII}
In this section we prove Theorems \ref{thm:main} 
for any time $ |t| \leq \eps^{- 3 (1-\delta) + \kappa} $ in item ($ii$). Over this extended (optimal) timescale,
the  profile in \eqref{etaapptilde} is not  an approximation of the surface water wave anymore, and we
need a more refined one. 
The  idea is to  use
as approximate solution 
$
\eta_{\rm app2}(t,\cdot)
$  in \eqref{etaapp2}, 
whose
phases are evaluated along the solution $z(t; \eta_0^\omega, \psi_0^\omega)$ of \eqref{theoBireq} and not along the solution of the cubic Birkhoff normal form \eqref{BNF3}, which preserves the actions, 
cf.  \eqref{etaapptilde}. 
The introduction of {\em nonlinearly time-evolving phases} improves the time-validity of the approximation, cf. 
Lemma \ref{thm:approx2}. This is sufficient to prove the upper bound of the Large Deviation Principle of   \Cref{thm:up_LDP3}.  Conversely, a new issue arises: we can no longer guarantee that the new approximate solution
 $\eta_{\rm app2}(t,x)$ 
 is still a Gaussian random variable. A new strategy is needed to prove 
 the lower bound of the Large Deviation Principle of 
 \Cref{thm:lo_LDP_improved}.
 This is addressed by the construction in \Cref{lem:T_lips}. 
Finally, in Section \ref{sec:allr} we prove \Cref{thm:df}.

\subsection{The new approximate solution
and the upper bound} 
We define
\begin{equation}\label{etaapp2}
\begin{aligned}
    \eta_{\rm app2}(t,\cdot) & := \sqrt{2} \Re\,  |D|^{\frac14}\, u_{\rm app2}(t,\cdot) \ , \\
    u_{\rm app2}(t,x) & :=   \Upxi\big(t, z(\cdot; \eta_0^\omega, \psi_0^\omega)\big) \zeta_0^\omega  = {\frac{1}{\sqrt{2\pi}}}\sum_{j \in \Z \setminus\{0\}} e^{- \im \int_0^t  \cL_j(I(z(\tau; \eta_0^\omega, \psi_0^\omega))) \di \tau } \, \zeta_{0j}^\omega  \, e^{\im j x } \, , 
\end{aligned}
\end{equation}
where  $\zeta_0^\omega (x) $ is the initial datum in \eqref{zeta0}, 
with Fourier coefficients
$\zeta_{0j}^\omega = 
|\zeta_{0j}^\omega| e^{\im \phi_j^\omega}$, 
$\Upxi(t,\cdot)$ is the linear unitary operator  in \eqref{eq:def_upxi}, 
and $z(t; \eta_0^\omega, \psi_0^\omega)$ is  defined in \eqref{Upsilon} (recall it solves  the equation \eqref{theoBireq}).
The approximate wave profile \eqref{etaapp2} can be written as  
 \begin{equation}\label{eq:eta2notg}
 \eta_{\rm app2} (t,x) =  {\frac{\varepsilon}{\sqrt{\pi}}}\,\sum_{j\in\Z\setminus\{0\}}   c_j R_j^{\omega}\,\cos\left(\vartheta_j (t;\omega) + j x \right) 
\end{equation}
where $ c_j $ are defined in \eqref{ckdk}, 
\begin{itemize}
\item the amplitudes 
$R_j^\omega := 
{\dfrac{ |j|^{\frac14}|\zeta_{0j}^\omega|}{\eps c_j}}$ are i.i.d Rayleigh r.v. $\sim \sR(1/\sqrt{2})$ (by \eqref{zeta0.g}); 
\item nonlinear  phases are 
\begin{equation}\label{eq:theta}
\vartheta_j (t;\omega):=-\int_0^t \cL_j (I(z(\tau; \eta_0^{\omega}, \psi_0^{\omega}))) \di\tau + \phi_j^{\omega} \quad \mbox{(mod}\ 2\pi ) \, . 
\end{equation}
\end{itemize}
Note that  since 
$z(\tau; \eta_0^{\omega}, \psi_0^{\omega})$ depends on the initial datum $(\eta_0^{\omega}, \psi_0^{\omega})$  in a very complex and nonlinear way, the phases $ \vartheta_j (t;\omega)$ are 
 generally not uniformly distributed in the desired timescale nor independent among themselves or from the $R_j^\omega$'s. This stands in contrast 
with the previous approximation $\eta_{\rm app}(t,x)$ in \eqref{etaapptilde}.

However, a key property is that 
 $\eta_{\rm app2}(t,x)$ approximates  the  water waves surface 
 $\eta(t,x)$ in the $\norm{\cdot}_{L^\infty}$
 on the whole time  interval $|t| \leq \eps^{-3(1-\delta)+\kappa}$, as proved in the next lemma.

\begin{lem}\label{thm:approx2} {\bf (Approximation 2)} 
Fix $\delta \in (0,1)$ and  $s \geq  N + \frac54$ with $ N $ given by \Cref{BNFtheorem}. 
For any $\tR > 0$, any $0<\kappa \ll 1-\delta$, there exist $\eps_0 := \eps_0(\tR, \kappa), \tC_2>0$ such that for   any $\eps \in (0, \eps_0)$, the following holds true.
For any initial datum 
 $(\eta_0^\omega, \psi_0^\omega) \in  \cB_0(\eps,\delta, \tR,s)$ in \eqref{cB0}, 
  the solution $(\eta(t,x),  \psi(t,x)) $ of the water waves equations \eqref{ww} exists up to times $|t| \leq T_\eps:= \eps^{-3(1-\delta)+\kappa}$,  as well as the approximate solution  $\eta_{\rm app2}(t, \cdot)$ 
  in \eqref{etaapp2}, and satisfy 
\begin{equation}\label{upxi_error2}
\norm{\eta(t, \cdot)- \eta_{\rm app2}(t, \cdot)}_{L^\infty(\T)}\lesssim \norm{u(t,\cdot ) - u_{\rm app2}(t, \cdot )}_{H^1}
\leq \tC_2 \, \tR^4 \, \varepsilon^{1-\delta+\kappa}   \ , \qquad  \forall |t|\leq T_\eps \, .
\end{equation}
\end{lem}

\begin{proof}
By \Cref{cor:LIP}, for any $(\eta_0^\omega, \psi_0^\omega) \in \cB_0(\eps, \delta, \tR, s)$ the functions 
$$ 
u(t) := u(t,x) :=u(t,x; \eta_0^\omega, \psi_0^\omega) \, , \quad 
z(t):=z(t,x) := z(t,x; \eta_0^\omega, \psi_0^\omega) = \mathfrak{B}(u(t)) u(t) \, , 
$$ are well 
 defined  for all times $|t| \leq T_\eps$ and satisfy  \eqref{z.control}.
 Then the approximate solutions 
 $u_{\rm app2}(t,x)$ and $\eta_{\rm app2}(t,x)$ in \eqref{etaapp2} are well defined for all 
 $|t| \leq T_\eps $.
We now estimate
\be\label{eq:ladiff}
\eta(t,\cdot ) - \eta_{\rm app2}(t, \cdot ) 
\stackrel{\eqref{etaapp2},\eqref{u0}} 
= \sqrt{2} \, \Re |D|^{\frac14} \big(
u(t,\cdot) - u_{\rm app2}(t,\cdot) \big) \, . 
\ee 
As in \Cref{thm:approx}, the key idea is to use the normal form \Cref{BNFtheorem} to reduce the problem to comparing the normal form variable $z(t,\cdot)$ with the solution $\Upxi(t, z(\cdot))z_0$ of the integrable Birkhoff normal form \eqref{BNF3}, where $z_0(x) := z(0,x)$, but now allowing for nonlinearly evolving phases.
Set 
$u_0(x):= u(t,x)\vert_{t=0}$.  
By \eqref{eq:ladiff}, Sobolev embedding,
and since  $\mathfrak{B}(u)$ is invertible (see \Cref{BNFtheorem}),  we estimate 
\begin{align}\notag
\norm{\eta(t,\cdot ) - \eta_{\rm app2}(t, \cdot )}_{L^\infty(\T)}\lesssim &
\norm{u(t, \cdot)- u_{\rm app2}(t, \cdot)}_{\dot H^1} \\
 \lesssim & \norm{\mathfrak{B}^{-1}(u(t, \cdot)) z(t, \cdot)- z(t, \cdot)}_{\dot H^1} 
\label{u-uapp12}
\\
& + \norm{z(t, \cdot) - \Upxi(t, z)z_0 }_{\dot H^1}
\label{u-uapp22}
\\
& + \norm{\Upxi(t, z) z_0 - u_{\rm app2}(t, \cdot)}_{\dot H^1}\, . 
\label{u-uapp42}
\end{align}
We estimate each contribution separately. Line \eqref{u-uapp12} is
equal to \eqref{u-uapp1} which is estimated by
\eqref{0209:1924}.\\
{\sc Estimate of  \eqref{u-uapp42}.} By \eqref{etaapp2} and since  $\Xi(t,\cdot)$ is a  unitary Fourier multiplier for every $t \in \R$, we have 
\begin{align}
\label{u-uapp33}
& \eqref{u-uapp42} \leq
\norm{z_0 - \zeta_0^\omega}_{\dot H^1} \leq
\norm{z_0 - u_0}_{\dot H^1} +
\norm{u_0 - \zeta_0^\omega}_{\dot H^1} 
\stackrel{\eqref{0209:1212},\eqref{u0-zeta0}}{\lesssim}
\tR^2 \eps^{2(1-\delta)}  \ . 
\end{align}
{\sc Estimate of  \eqref{u-uapp22}.} 
Write  
$
z(t) := \Upxi(t,z) v(t)$; by variation of constants, using  
\eqref{transformed_eq}, \eqref{eq:def_upxi},  the variable 
$v(t)$ satisfies 
$$
\pa_t v  = \Upxi(t, z)^{-1} \cX_{\geq 4}^+(u, \bar u, z, \bar z)\, ,   \quad v(0) = z_0 \, . 
$$
By the unitarity of $\Upxi(t,z)$ and estimate \eqref{X_4_1} we bound 
$\norm{\pa_t v(t)}_{\dot{H}^1} 
\lesssim \norm{u(t)}_{\dot H^{s_0}}^4 $, 
implying
\be\label{u-uapp43}
\eqref{u-uapp22}  =
\norm{v(t) - z_0}_{\dot{H}^1}  \lesssim |t| 
\sup_{|\tau| \leq t}\norm{u(\tau)}_{\dot H^{s_0}}^4 \stackrel{\eqref{z.control}}{\lesssim } \tR^4 \eps^{1-\delta +\kappa} \ , \quad \forall |t| \leq T_\eps \ .
\ee
Estimate \eqref{upxi_error2} follows from 
\eqref{0209:1924},
\eqref{u-uapp33},
\eqref{u-uapp43}.
\end{proof}

With this approximation lemma, we can readily prove  the  upper bound of \Cref{thm:main} $(ii)$. 

\begin{prop}[{\bf LDP: Upper-bound on optimal timescale}]\label{thm:up_LDP3}
There exists $\und{s}>0$ 
such that
for any  $\lambda_0 >0$, $\delta \in (0,1)$, and $s \geq \und{s}$ there exists $\tR_0  >0$ and  for any $\tR\geq \tR_0$,   any $0<\kappa \ll 1-\delta$, the following holds true. There exists $\eps_0 > 0 $ such that for   any $\eps \in (0, \eps_0)$, 
for any random initial data $(\eta_0^\omega, \psi_0^\omega) \in  \cB_0(\eps, \delta, \tR , s)$ in \eqref{cB0}, 
the solution $(\eta(t,x), \psi(t,x))$ of the water waves equations \eqref{ww} exists up to times $|t| \leq \eps^{-3(1-\delta)+\kappa}$, and satisfies 
the upper bound 
\begin{equation}\label{uppboptimal}
\limsup_{\eps \to 0^+}  \eps^{2\delta} \log \P \left(
\left\{ \sup_{x\in\T} \eta (t,x) \geq \lambda_0 \varepsilon^{1-\delta}\right\} \cap \cB_0(\eps,\delta, \tR, s) \right) \leq  -\frac{\lambda_0^2}{{2 \bsi^2}} \, ,
\end{equation}
with $\bsi$ in \eqref{upsigma}.
\end{prop}
\begin{proof}
Argue as in the proof of  Theorem \ref{thm:main} $(i)$ (upper bound), replacing $\eta_{\rm app}(t,x)$ with 
$\eta_{\rm app2}(t,x)$, the constant $\tC$ with  
$\tC_2$, 
using Lemma \ref{thm:approx2} in place of Lemma \ref{thm:approx} and  using that
$$
\norm{\eta_{\rm app2} (t,\cdot)}_{L^\infty(\T) } \stackrel{\eqref{eq:eta2notg}}{\leq} 
\frac{\eps}{\sqrt{\pi}} \sum_{j \neq 0}
c_j \,  R_j^\omega \quad 
\text{where} \quad 
R_j^\omega   \sim \sR \left(\frac{1}{\sqrt2}\right)\, , 
$$ 
from which it follows that 
 \[
\P \left( \sup_{x\in\T}  \eta_{\rm app2} (t,x) \geq \lambda_0 \varepsilon^{1-\delta}(1-\tC_2\tR^4\,\varepsilon^{\kappa}) \right)
\leq 
\P \Big( \sum_{j\in\Z\setminus\{0\}} c_j R_j^\omega \geq \sqrt{\pi}\, \lambda_0\varepsilon^{-\delta}(1-{\tC_2} \tR^4 \,\varepsilon^{\kappa})\Big) \, .
\]
Then \Cref{thm:LDP_Rayleigh}
implies \eqref{uppboptimal}. 
\end{proof}

Proving a sharp lower bound over the desired timescale $|t|\leq \varepsilon^{-3(1-\delta)+\kappa}$ requires much more work.
The key point to have a rogue wave form  at time $t$ is that plenty of the  phases 
$\vartheta_j(t;\omega)$ in \eqref{eq:theta}
align at zero, so that, in the expression \eqref{eq:eta2notg} evaluated at $x =0$, 
$\cos\left(\vartheta_j (t;\omega) \right) \simeq 1$ and one can exploit again \Cref{thm:LDP_Rayleigh}.
We can guarantee that, given an arbitrary time 
 $|t| \leq \eps^{-3(1-\delta) + \kappa}$, a finite but large number of  phases \eqref{eq:eta2notg} align at zero
via a random fixed point theorem, that is the content of the next \Cref{sec:fixed}. The probability that random initial phases coincide with the fixed point is zero, so we construct an event of quasi-synchronized phases by perturbing around the fixed point. An additional difficulty is to estimate the measure of this event.
It is in this point that we need to keep track of the Lipschitz dependence of  $z(t,\cdot)\equiv z(t; \eta_0^{\omega}, \psi_0^{\omega})$   on the initial datum.

\subsection{Random fixed point}\label{sec:fixed}

For any integer $\tN\geq 1$ 
we consider the low/high frequency  projectors
\be\label{truncN}
\begin{aligned}
& \Pi_{ \leq \tN} : \dot L^2 (\T, \C)  \longrightarrow C^\infty (\T, \C) \ , 
\quad 
 f  (x) \longmapsto \Pi_{\leq \tN}f(x) :=  {\frac{1}{\sqrt{2\pi}}}\sum_{0 < |j| \leq \tN} f_j \, e^{\im j  x} \ ,  \\
& \Pi_{>\tN} : \dot L^2 (\T, \C)  \longrightarrow L^2 (\T, \C) \, , 
\quad 
 f  (x) \longmapsto \Pi_{>\tN}f(x) := {\frac{1}{\sqrt{2\pi}}}\sum_{|j|>\tN} f_j \, e^{\im j  x} \, ,
\end{aligned}
\ee
so that  $ \uno =  \Pi_{\leq \tN} + \Pi_{>\tN}$. 
The truncation parameter $ \tN $
will be linked to $ \varepsilon $ in \eqref{N.choice}.

\paragraph{\bf Partially randomized initial data.} 
For any $\vec{\phi}\in\R^{2\tN}$ we define the 
{\it partially randomized} initial data
\be 
\begin{aligned}\label{eq:part_rand_init}
& \wt\eta_0^{\omega}  (x;\vec{\phi})  := 
{ \frac{\eps}{\sqrt{\pi}} } \sum_{0<|j|\leq\tN}c_j\, R_j^{\omega}  \,\cos\left(\phi_j+  j x\right)+
\underbrace{{ \frac{\eps}{\sqrt{\pi}} } \sum_{|j|>\tN}c_j\, R_j^{\omega}  \,\cos\left(\phi_j^{\omega}+  j x\right)}_{= \Pi_{>\tN} \wt\eta_0^{\omega}  } 
\\
&  \wt\psi_0^{\omega}  (x;\vec{\phi})  :=
{ \frac{\eps}{\sqrt{\pi}} }\sum_{0<|j|\leq\tN}d_j\, R_j^{\omega} \,\sin\left(\phi_j+  j x\right)+\underbrace{{ \frac{\eps}{\sqrt{\pi}} } \sum_{|j|>\tN}d_j\, R_j^{\omega} \,\sin\left(\phi_j^{\omega}+  j x\right)}_{= \Pi_{>\tN} \wt \psi_0^{\omega} }
\end{aligned}
\ee
where $ c_j $ are the coefficients defined in \eqref{ckdk}, $ d_j = c_j |j|^{-\frac12}$,  and 
\begin{itemize}
\item 
$(R_j^{\omega})_{j\in\Z\setminus\{0\}}$ are i.i.d.  Rayleigh random variables $ \sim \sR(1/\sqrt{2})$;
\item $(\phi_j^{\omega})_{|j|>\tN}$ are i.i.d. random variables with a uniform distribution 
$ \sim \sU ([0,2\pi])$;
\item the random variables  $(R_j^{\omega})_{j\in\Z\setminus\{0\}}$ 
and $(\phi_j^{\omega})_{|j|>\tN}$ are independent.
\end{itemize} 
Note that in \eqref{eq:part_rand_init} 
the high Fourier modes are fully random, while low Fourier modes have random moduli but {\it deterministic} phases $\vec{\phi}$.

If $\vec{\phi}^{\omega}=(\phi_j^{\omega})_{0<|j|\leq \tN}$ are i.i.d.  random variables $ \sim \sU([0,2\pi]) $,  independent of $(R_j^{\omega})_{j\in\Z\setminus\{0\}}$ and $(\phi_j^{\omega})_{|j|> \tN}$, then 
\begin{equation}\label{eq:part_to_full}
(\wt\eta_0^{\omega}, \wt\psi_0^{\omega}) (x;\vec{\phi}^{\omega})=(\eta_0^{\omega},\psi_0^{\omega})(x)
\end{equation}
where the right-hand side denotes the fully randomized initial datum in \eqref{eq:init_data}. 
In particular, for any $\vec \phi \in \R^{2\tN}$,   
 \begin{align}\label{1909:1520}
 \big\| (\wt\eta_0^{\omega}(\cdot,\vec{\phi}),\wt{\psi}_0^{\omega}(\cdot,\vec{\phi})) \big\|_{H^s_0 \times \dot H^{s+\frac12}}^2 
 & =2\,\eps^2\,\sum_{j\in\Z\setminus\{0\}} c_j^2 (R_j^{\omega})^2 | j|^{2s}\\
 & \hspace{-1cm} = 2\,\eps^2\,
 \!\! \sum_{0<|j|\leq \tN}  \!\! c_j^2 (R_j^{\omega})^2 | j|^{2s}+\big\| (\Pi_{>\tN}\wt\eta_0^{\omega},\Pi_{>\tN}\wt{\psi}_0^{\omega}) \big\|_{H^s_0 \times \dot H^{s+\frac12}}^2 =  \norm{(\eta_0^{\omega},{\psi}_0^{\omega})}_{H^s_0 \times \dot H^{s+\frac12}}^2 \, . \notag
 \end{align}
As a result, the event $\cB_0(\eps,\delta, \tR, s) $ in \eqref{cB0p}  only depends on the variables 
$$
(\vec{R}^{\omega},\Pi_{>\tN}  \wt\eta_0^{\omega},\Pi_{>\tN} \wt\psi_0^{\omega})
\quad \text{where} 
\quad 
\vec{R}^{\omega} :=(R_j^{\omega})_{0<|j|\leq \tN}
\quad \text{in} \  \eqref{eq:part_rand_init} \, . 
$$  

\paragraph{\bf Random nonlinear phase operator.} At any  time  $|t| \leq \eps^{-3(1-\delta) + \kappa}$, we define the operator
\be\label{def:T}
\begin{split}
\cT : \cB_0(\eps,\delta, \tR, s) \times \R^{2\tN} & \longrightarrow \R^{2\tN} \, , \\
(\omega , \vec{\phi} ) & \longmapsto (\cT_j 
(\omega , \vec{\phi} ))_{0<|j|\leq \tN} \, , 
\end{split}
\ee
where, for any $0<|j|\leq \tN$, 
\be\label{def:T2}
\cT_j (\omega , \vec{\phi} ):= \int_0^t \cL_j (I(z(\tau ; \omega,\vec{\phi})))\, \di\tau \, , \qquad z(t;\omega,\vec{\phi}) := \Upsilon^t \big( \wt\eta_0^{\omega}(\cdot,\vec{\phi}),\wt{\psi}_0^{\omega}(\cdot,\vec{\phi})\big) \, ,
\ee
and  $\Upsilon^t $  is defined in \eqref{Upsilon}. The domain $\cB_0(\eps,\delta, \tR, s) $ in \eqref{def:T} is the {\it restriction}  $(\cB_0(\eps,\delta, \tR, s),\wt{\Fc}\cap \cB_0(\eps,\delta, \tR, s),\wt{\P})$ of the probability space $(\mathbb{\Omega},\Fc,\P)$ given by the $\sigma$-algebra generated by the random moduli and the high-Fourier modes, i.e.
\be\label{259:1639}
\wt{\Fc}:= \sigma(\vec{R}^{\omega},\Pi_{>\tN}  \wt\eta_0^{\omega},\Pi_{>\tN} \wt\psi_0^{\omega})\subseteq \Fc \, , 
\qquad 
\wt{\P}(F):=\P(F) \, ,\quad \forall F=\wt{F}\cap\cB_0(\eps,\delta, \tR, s)\in \wt{\Fc}\cap\cB_0(\eps,\delta, \tR, s) .
\ee
Without loss of generality, we can assume that $(\Omega,\wt{\Fc},\mathbb{P})$ and $(\cB_0(\eps,\delta, \tR, s),\wt{\Fc}\cap \cB_0(\eps,\delta, \tR, s),\wt{\P})$ are complete and independent of the initial uniformly distributed phases $\vec{\phi}^{\omega}$.

The mapping $\cT$ is well defined since, for any $\omega \in \cB_0(\eps,\delta, \tR, s)$, 
 \[
 \big\| (\wt\eta_0^{\omega}(\cdot,\vec{\phi}),\wt{\psi}_0^{\omega}(\cdot,\vec{\phi})) \big\|_{H^s_0 \times \dot H^{s+\frac12}} 
\stackrel{\eqref{1909:1520}}{=}  \norm{(\eta_0^{\omega},{\psi}_0^{\omega})}_{H^s_0 \times \dot H^{s+\frac12}} 
 \leq \tR \eps^{1-\delta}
 \]
and so \Cref{cor:LIP} guarantees that the functions 
$$ 
\wt u(t)   
:= \wt u(t,x; \wt \eta_0^\omega (\cdot,\vec{\phi}),
\wt \psi_0^\omega (\cdot,\vec{\phi}) ) \, , \quad 
z(t;\omega,\vec{\phi})= \Upsilon^t \big( \wt\eta_0^{\omega}(\cdot,\vec{\phi}),\wt{\psi}_0^{\omega}(\cdot,\vec{\phi})\big) 
= 
\mathfrak{B}( \wt u(t)) \wt u(t)  \, , 
$$ 
are well 
 defined  for all times $|t| \leq T_\eps$ and satisfy  \eqref{z.control}.

The next crucial proposition  guarantees the existence of a random fixed point for $\cT $, and studies properties of the event of random phases being close to the fixed point. Recall that, at the random fixed point,  the nonlinear phases \eqref{eq:theta} vanish, and a rogue wave forms.

The following proposition contains all the results to prove the  lower bound \eqref{eq:lo_LDP_improved}. 
\begin{prop}[{\bf Random fixed point}]\label{lem:T_lips}
Let $\tN \in \N$, $\delta\in (0,1)$ and $s \geq  N + \frac54$.
For any $\tR >0$,  any 
$0 < \kappa \ll 1-\delta$, there exists $\eps_0 >0$ such that for any $\eps \in (0, \eps_0)$, any 
$|t|\leq  T_\eps:=\eps^{-3(1-\delta)+\kappa}$, the  random nonlinear phase operator  $\cT $ defined in \eqref{def:T}--\eqref{def:T2}
satisfies the following properties:  
\begin{itemize}
\item[(i)] {\sc Lipschitz property.} For any $\omega \in \cB_0(\eps,\delta, \tR, s)$, the map $\cT(\omega, \cdot)\colon \R^{2\tN}\to \R^{2\tN}$ is Lipschitz, and, for any $ 0 < |j| \leq \tN $,     
\begin{equation}\label{eq:bilipschitz}
\big| \cT_j (\omega, \vec\phi) - \cT_j(\omega, \und{\vec\phi}) \big|  \leq C_2 \tN\, \tR^2 \,  
\eps^{2(1-\delta)} \, e^{C_{\rm lip} |t|} \,
\| \vec \phi - \und{\vec \phi} \|_{\ell^\infty}
 \, , \quad \forall \,  
 \vec \phi,  \und{\vec \phi} 
\in \R^{2\tN} \, . 
\end{equation}
\item[(ii)] {\sc Measurability.} For any $\vec{\phi} \in \R^{2\tN}$, the map  $\cT(\cdot, \vec{\phi})\colon 
( \cB_0(\eps,\delta, \tR, s), \wt{\Fc} \cap \cB_0(\eps,\delta, \tR, s))
\to (\R^{2\tN}, \cB(\R^{2\tN}))$ is measurable.
\item[(iii)] 
{\sc Fixed point.} 
There exists a random fixed point $\vec{\phi}^{\ast,\omega} = ( \phi^{\ast,\omega}_j)_{0 <|j| \leq \tN}
$ of $\cT$,  
\be\label{eq:T_fixedp}\cT (\omega, \vec{\phi}^{\ast,\omega}) = \vec{\phi}^{\ast,\omega} \, ,
\ee
which is a random variable in the restricted probability space $(\mathbb{\Omega},\wt{\Fc},\wt{\P})$ defined in \eqref{259:1639}.
\end{itemize} 
Moreover, the fixed point $ \vec{\phi}^{\ast,\omega}  $
satisfies the following properties:   
\begin{itemize}
\item[(a)] If $\vec{\phi}^{\omega}=(\phi_j^{\omega})_{0<|j|\leq \tN}$ are  i.i.d. uniform  random variables  $\sim \sU([0,2\pi])$ independent of $(R_j^{\omega})_{j \in\Z\setminus\{0\}}$ and $(\phi_j^{\omega})_{|j|> \tN}$,
then, for any $\alpha \in (0, \pi)$, the set 
\begin{equation}\label{eq:nhood_fixedp}
\Nc ( \alpha) := \Big\{ \omega\in\mathbb{\Omega} \mid |\phi_j^{\omega} - \phi_j^{\ast,\omega}|< \alpha   \quad \forall\ 0<|j|
\leq \tN \Big\}  
\end{equation}
is $\cF$-measurable and  
$ 
\P( \Nc (\alpha)) 
= \left(\frac{ \alpha}{\pi}\right)^{2\tN} $.
    \item[(b)] 
    {\sc Stability.}
    For any $\omega \in \Nc(\alpha)\cap \cB_0(\eps,\delta, \tR, s)$ we have 
\be\label{5.8}
\abs{\phi_j^\omega - \cT_j (\omega , \vec{\phi}^\omega)}  < 2\, \alpha\, \tN \,  e^{ C_{\rm lip} \varepsilon^{-3(1-\delta)}}, \qquad \forall\ 0<|j|
\leq \tN \, , 
\ee
where $C_{\rm lip}>0$ is the constant in 
\Cref{cor:LIP}-(iv).
    \item[(c)]{\sc Factorization property.} For any measurable set $\cA \in \wt \cF$ (see \eqref{259:1639}), 
    \be\label{splitting}
\P(\cA \cap \Nc (\alpha)) 
 = {\P(\cA)\, \P (\Nc (\alpha))} 
= \left(\frac{ \alpha}{\pi}\right)^{2\tN} \P(\cA ) \, .
    \ee
\end{itemize}
\end{prop}

\begin{rk}\label{rk:independence} 
The domain in \eqref{def:T} is restricted to  \eqref{259:1639} in order to guarantee 
that the first $2\tN$ uniformly distributed random phases $\vec{\phi}^{\omega}$ are independent from 
the random fixed point $\vec{\phi}^{\ast,\omega}$ (constructed in \Cref{lem:T_lips}), which is
generated
 only by the random variable $(\vec{R}^{\omega},\Pi_{>\tN}  \wt\eta_0^{\omega},\Pi_{>\tN} \wt\psi_0^{\omega})$. 
This is the key to ensure the crucial factorization property in  \Cref{lem:T_lips} ($c$).
\end{rk} 

Remarkably, the factorization property \eqref{splitting} holds even though 
$\Nc(\alpha)$ and $\wt\cF$ are not independent. 
This property will play a key role in establishing the sharp lower bound 
for the probability of rogue wave formation in 
\Cref{thm:lo_LDP_improved}. 
In particular, one has 
$\P(\Nc(\alpha)) = \left( \frac{\alpha}{\pi} \right)^{2\tN}$.
Since $\left( \frac{\alpha}{\pi} \right)^{2\tN} \to 0$ as $\tN \to +\infty$, 
it is important to choose $\tN = \tN(\varepsilon) \sim \varepsilon^{-\gamma}$ 
as in \eqref{N.choice} below, which ensures that the measure of the set 
$\Nc(\vec{\phi}^{\ast,\omega}; \alpha)$ does not vanish too rapidly as 
$\varepsilon \to 0^{+}$.

\begin{cor}{\bf (Phase synchronization for the approximate solution)} \label{cor:phase_sync}
Let $\tN \in \N$, $\delta\in (0,1)$ and $s \geq  N + \frac54$.
For  any $\tR >0$,  any 
$0 < \kappa \ll 1-\delta$, there exists $\eps_0 >0$ such that for any $\eps \in (0, \eps_0)$, any 
$|t|\leq  T_\eps:=\eps^{-3(1-\delta)+\kappa}$,
the following holds. 
Let $\alpha \in (0, \pi)$  such that 
\be\label{alpha.N}
2\, \alpha\, \tN  \leq  \eps \, e^{ -C_{\rm lip} \varepsilon^{-3(1-\delta)}}. 
\ee
Then, for any $\omega \in \Nc (\alpha)\cap \cB_0(\eps, \delta, \tR, s)$
(the set   
$ \Nc ( \alpha)  $ is defined in \eqref{eq:nhood_fixedp}),
 the first $2\tN$-phases $\vartheta_j(t; \omega) $
 of   the approximate wave profile
 $\eta_{\rm app2}(t,x)$ 
in \eqref{eq:eta2notg} 
\emph{synchronize} at $ x = 0 $, i.e.
$$
|\vartheta_j(t; \omega)| < \eps  \, ,
\quad 
\forall\,  0< |j| \leq \tN \, , \quad 
\vartheta_j(t;\omega) \mbox{ in }  \eqref{eq:theta} \, .
$$
Consequently 
\be\label{cor:eta.syn}
\sup_{x \in \T}\eta_{\rm app2}^{\omega}(t,x) 
\geq { \frac{\eps}{\sqrt{\pi}} }\, (1-\varepsilon)\, \sum_{0<|j|\leq \tN}   c_j R_j^{\omega} - C\tN^{-(s-1)} \tR\, \varepsilon^{1-\delta}.
\ee
\end{cor}
\begin{proof}
By \Cref{lem:T_lips} $(b)$, for any 
$\omega\in \Nc (\alpha )\cap \cB_0(\eps,\delta, \tR, s)$ the phase synchronization property  \eqref{5.8} occurs.
Recalling the definitions \eqref{eq:theta}, \eqref{def:T2} and the identity \eqref{eq:part_to_full}, we have 
\be\label{piccaggio}
|\vartheta_j(t;\omega)| = \abs{\phi_j^\omega - \cT_j (\omega , \vec{\phi}^\omega)} \stackrel{\eqref{5.8}, \eqref{alpha.N}}{\leq} \eps 
\quad
\Rightarrow
\quad 
\cos\left(\vartheta_j(t;\omega)\right) \geq 1-\eps \ , \quad 
\forall 0<|j| \leq \tN \, . 
\ee
Then we bound
\begin{align}
\notag
\sup_{x\in\T} \eta_{\mathrm{app2}}(t,x)  \geq \eta_{\rm app2}(t,0) & \stackrel{\eqref{eq:eta2notg}}{=} 
   { \frac{\eps}{\sqrt{\pi}} }\sum_{0< |j|\leq \tN}   c_j R_j^{\omega}\,\cos\left(\vartheta_j(t;\omega)  \right) + \Pi_{>\tN} \eta_{\mathrm{app2}}(t,0) \\
    & \stackrel{\eqref{piccaggio}}{\geq}
    { \frac{\eps}{\sqrt{\pi}} } (1-\eps)\, \sum_{0< |j|\leq \tN}   c_j R_j^{\omega}+ \Pi_{>\tN} \eta_{\mathrm{app2}}(t,0) \, .
    \label{passiok}
\end{align}
To bound the last term, we note that for any $\omega\in \cB_0 $
\begin{align}
\big\| \Pi_{>\tN} \eta_{\mathrm{app2}}^{\omega}(t,\cdot) \big\|_{L^\infty} & 
\stackrel{\eqref{etaapp2}}{\lesssim} \norm{ \Pi_{>\tN} \zeta_0^{\omega}}_{\dot H^1}  \stackrel{\eqref{zeta0}}{\lesssim}\norm{ (\Pi_{>\tN} \eta_0^{\omega}, 
\Pi_{>\tN} \psi_0^{\omega})
}_{H_0^\frac34 \times \dot H^\frac54} \notag \\
& \lesssim \tN^{-(s- 1)} \, \norm{(\eta_0^{\omega},\psi_0^{\omega})}_{H_0^{s} \times \dot{H}^{s+\frac12}} 
\stackrel{\eqref{cB0p}}{\lesssim} \tN^{-(s- 1)}\, \tR \, \eps^{1-\delta} \,  .  \label{passiok2}
\end{align} 
Estimate \eqref{cor:eta.syn} follows by 
\eqref{passiok2} and \eqref{passiok} .
\end{proof}

 The rest of this subsection is devoted to the proof of \Cref{lem:T_lips}.

The fixed point will be found by  
the  random Brouwer fixed point theorem \ref{thm:Brouwer}. In order to verify its hypotheses, we first prove a deterministic Lipschitz property of the nonlinear phases, Lemma \ref{lem:bilipschitz}. Then, writing the 
nonlinear phase operator $\cT$ in \eqref{def:T}-\eqref{def:T2} as the composition of a continuous and a measurable map as in \eqref{eq:decomp_T},  we will deduce  properties ($i$) and ($ii$) in \Cref{lem:T_lips}.
\\[1mm]
{\bf Deterministic nonlinear phase operator 
$\widetilde \cT $.}
For any $s\in\R_{+}$, let 
\begin{equation}\label{eq:B0tilde}
\begin{aligned}
\cY (\eps,\delta, \tR, s, \tN) := \Big\lbrace 
v & 
= (\vec{r}, \wt\eta_{0,\tN},\wt\psi_{0,\tN} )  
\in  \R_{+}^{2\tN}\times  {H}^{s}_0 \times \dot{H}^{s+\frac12} \  | \ \Pi_{\le\tN} \, \wt\eta_{0,\tN}
=  \Pi_{\le\tN}\wt\psi_{0,\tN}  =0 \, , \\
 & 
\norm{v}_{\cY}:=
\Big( \sum_{0< |j|\leq \tN} |j|^{2s} c_j^2\,r_j^2 + \big\| (\wt\eta_{0,\tN},\wt\psi_{0,\tN}) \big\|_{{H}_0^{s} \times \dot{H}^{s+\frac12}}^2\Big)^{1/2}
  \leq \tR\, \varepsilon^{-\delta}\Big\rbrace  
  \end{aligned}
\end{equation}
where the constants $c_j$ are introduced in \eqref{ckdk}.
Then we define for any $0 < |j| \leq  \tN$ the deterministic map
\begin{equation}\label{wtT}
\begin{split}
{\wt\cT}_j & :\cY(\eps,\delta, \tR, s, \tN) \times \R^{2\tN}  \longrightarrow \R \, , \\
(v , \vec{\phi} ) & \longmapsto 
{\wt\cT}_j ( v, \vec{\phi}) = \int_0^t \cL_j\big(I(z(\tau; v, \vec{\phi})\big) \,\di\tau \, , 
\qquad 
z(\tau; v, \vec{\phi}):= \Upsilon^t ((\wt\eta_0,\wt\psi_0) (\cdot;v,\vec{\phi})) \, , 
\end{split}
\end{equation}
where $\Upsilon^t$ is introduced  in  \eqref{Upsilon} and, 
for any $(v, \vec{\phi})\in 
\cY (\eps,\delta, \tR, s, \tN)\times \R^{2\tN}$,
\begin{equation}\label{eq:trunc_init_data}
\begin{split}
\wt\eta_0 (x;v,\vec{\phi}) & := \frac{\eps}{\sqrt{\pi}} \sum_{0<|j|\leq\tN}c_j\, r_j \,\cos\left(\phi_j +  j x\right)+ \eps\,\wt\eta_{0,\tN} (x)\\
 \wt\psi_0 (x;v,\vec{\phi}) & := \frac{\eps}{\sqrt{\pi}}  \sum_{0<|j|\leq\tN}d_j\, r_j \,\sin\left(\phi_j+  j x\right)+ \eps\,\wt\psi_{0,\tN} (x)
\end{split}
\end{equation}
where $ d_j = c_j |j|^{-\frac12}$.

\begin{lem}[{\bf Deterministic nonlinear phase operator 
$\widetilde \cT $}]\label{lem:bilipschitz}
Let $\tN \in \N$. 
Fix  
$\delta \in (0,1)$ and  $s \geq  N + \frac54$. 
For any $\tR > 0$, any $0<\kappa \ll 1-\delta$, there exists $\eps_0 := \eps_0(\tR, \kappa)>0$ and 
$C_2 >0 $ 
such that for   any $\eps \in (0, \eps_0)$,  
for any $|t| \leq T_\eps = \eps^{-3(1-\delta) + \kappa}$, 
for any $0 < |j| \leq  \tN $, 
the  map $\wt\cT_j $ in \eqref{wtT}  is well defined,
\be\label{eq:badboundtilde0}
| \widetilde 
\cT_j (v , \vec{\phi} )|\leq |t| \, \sup_{\tau\in [0,t]} |\cL_j (I(z(\tau; v, \vec{\phi})))|  
\leq C_2 |t| \tN  \, , 
\ee
and satisfies the Lipschitz property,    
for any $ (v, \vec\phi) $, $ (\und{v}, \und{\vec \phi}) 
$ in 
$ \cY (\eps,\delta, \tR, s, \tN) \times \R^{2\tN}$, 
\begin{equation}\label{eq:bilipschitz2}
\abs{\widetilde{\cT}_j (v, \vec\phi) - \widetilde{\cT}_j(\und{v}, \und{\vec\phi})} \leq C_2 \tN\, \tR\,  \eps^{1-\delta} \,  e^{C_{\rm lip} |t|} \left(\tR \eps^{1-\delta}\,\| \vec \phi - \und{\vec \phi} \|_{\ell^\infty}
+
\eps\,\norm{v-\und{v}}_{\cY}  
 \right)  
\end{equation}
where $ C_{\rm lip} $ is the positive constant in 
\eqref{lip.z}.
\end{lem}

\begin{proof}
For any 
$(v, \vec \phi) 
\in \cY(\eps, \delta, \tR, s, \tN) \times \R^{2\tN}$ 
 (recall  \eqref{eq:B0tilde}), the initial data
$ (\wt\eta_0,\wt\psi_0) (\cdot;v,\vec{\phi}) $ in \eqref{eq:trunc_init_data} 
belong to the ball $\cB_0(\eps, \delta, \tR, s)$ in 
\eqref{cB0}. Thus by Theorem \ref{cor:LIP} 
the function $ z(\tau; v, \vec{\phi}) $  
in \eqref{wtT} is well defined for any $|t| \leq T_\varepsilon $, as well as the operator $\wt\cT_j $ in   \eqref{wtT},
and
the estimate  \eqref{eq:badboundtilde0} follows by 
$$
| \widetilde 
\cT_j (v , \vec{\phi} )|\stackrel{\eqref{wtT}}
\leq |t| \, \sup_{\tau\in [0,t]} |\cL_j (I(z(\tau; v, \vec{\phi})))| 
\stackrel{\eqref{HCS1},\eqref{z.control}} \lesssim \,|t | 
\, \tN (1 +  \tR^2 \varepsilon^{2-2\delta}) \leq C_1
|t| \tN  \, .
$$
For any  $(v, \vec \phi), (\und{v}, \vec{\und{\phi}}) $ in $ \cY(\eps, \delta, \tR, s, \tN) \times \R^{2\tN}$  denote $z(\tau):= \Upsilon^\tau ((\wt\eta_0,\wt\psi_0) (\cdot;v,\vec{\phi}))$ and 
$\und{z}(\tau):= \Upsilon^\tau ((\wt\eta_0,\wt\psi_0) (\cdot;\und{v},\vec{\und{\phi}}))$.
The  data in 
\eqref{eq:trunc_init_data} satisfy 
the Lipschitz property 
\be\label{iota.lip}
\big\| (\wt\eta_0,\wt\psi_0) (\cdot;v,\vec{\phi})  - (\wt\eta_0,\wt\psi_0) (\cdot;\und{v}, 
\vec{ \und{\phi}}) \big\|_{{H}_0^{s} \times \dot{H}^{s+\frac12}} \lesssim \eps\,\norm{v-\und{v}}_{\cY} + \tR\,\eps^{1-\delta}\,\| \vec \phi - \und{\vec \phi}\|_{\ell^\infty} \, .
\ee
The  Lipschitz property of $I\mapsto \cL_j(I)$ in \eqref{L_lips} and the one of $(\eta_0, \psi_0) \mapsto \Upsilon^\tau(\eta_0, \psi_0)$ in \Cref{cor:LIP} imply  
that, for any $|j| \leq \tN $, 
\begin{align*}
|\cL_j \big(I( z(\tau))\big) - & \cL_j \big(I(\und{z}(\tau))\big) |  \stackrel{\eqref{L_lips}}{\lesssim} \tN 
\norm{I(z(\tau))
-
I(\und{z}(\tau))}_{\Fc L^{2,1}}\\
& \stackrel{\eqref{actions}} \lesssim \tN 
\big( \norm{z(\tau)}_{H^2}+\norm{\und{z}(\tau)}_{H^2} \big)\, \norm{z(\tau)-\und{z}(\tau)}_{L^2}  \\
& \stackrel{\eqref{z.control},\eqref{lip.z}}{\lesssim} \tN \,\tR\, \varepsilon^{1-\delta}\, e^{C_{\rm lip} |\tau| }\,\norm{(\wt\eta_0,\wt\psi_0) (\cdot;v,\vec{\phi})- (\wt\eta_0,\wt\psi_0) (\cdot;\und{v},\vec{\und{\phi}})}_{H_0^{s_0+2}\times \dot{H}^{s_0+2}}\\
& \stackrel{\eqref{iota.lip}}{\lesssim}
 \tN \,\tR\, \varepsilon^{1-\delta}\, e^{C_{\rm lip} |\tau| }\, \left(  \tR\,\eps^{1-\delta}\,
\big\| \vec \phi - \und{\vec \phi} \big\|_{\ell^\infty} + 
\eps\,\norm{v-\und{v}}_{\cY} 
\right) \, .
\end{align*}
Integrating in $\tau$ in \eqref{wtT} one gets the Lipschitz estimate \eqref{eq:bilipschitz2}.
\end{proof}

We introduce the map
\begin{equation}
\begin{split}
\iota : \cB_0(\eps,\delta, \tR, s) \subset \mathbb{\Omega} & \longrightarrow \cY (\eps,\delta,\tR,s,\tN) \subset \R_{+}^{2\tN}\times  {H}^{s}_0 \times \dot{H}^{s+\frac12} \, , 
\quad
\omega  \longmapsto (\vec{R}^{\omega}, 
\eps^{-1} \Pi_{>\tN} \wt\eta_{0}^{\omega}, \eps^{-1} \Pi_{>\tN}  \wt\psi_{0}^{\omega}) \, , 
\end{split}
\end{equation}
where 
$(\wt\eta_{0}^{\omega},\wt\psi_{0}^{\omega})(x;\vec{\phi})$
and $\vec{R}^{\omega} :=(R_j^{\omega})_{0<|j|\leq \tN}$ 
are defined in  \eqref{eq:part_rand_init}.
Note that, by its very definition,  $\iota$ is measurable with respect to $\wt{\Fc} \cap \cB_0(\eps,\delta, \tR, s)$, since $\wt{\Fc} \cap \cB_0(\eps,\delta, \tR, s) = \sigma(\iota)$, cf. \eqref{259:1639}.
The map $\iota$ turns the deterministic datum \eqref{eq:trunc_init_data} into the random datum \eqref{eq:part_rand_init}, i.e.
\begin{equation}
(\wt\eta_0,\wt\psi_0) (x;\iota(\omega),\vec{\phi}) = (\wt\eta_0^{\omega},\wt\psi_0^{\omega})
(x;\vec{\phi}) \, . 
\end{equation}
As a result 
\begin{equation}\label{eq:decomp_T}
\cT_j (\omega, \vec{\phi}) = 
{\wt\cT}_j ( \iota(\omega), \vec{\phi})   \, , \quad \forall 0 < |j| \leq \tN \, .
\end{equation}

\medskip

We now prove the first two 
properties ($i$)-($ii$) stated in \Cref{lem:T_lips}.

\smallskip

($i$) It follows by \eqref{eq:decomp_T} and  \eqref{eq:bilipschitz2} 
with $v=\und{v}=\iota (\omega)$.

\smallskip

($ii$) In view of 
\eqref{eq:decomp_T} it suffices to prove the measurability  of  $\wt{\cT}_j (\cdot, \vec{\phi})$ and $\iota$ separately. The operator $\wt{\cT}_j (\cdot, \vec{\phi})$ is Lipschitz by \Cref{lem:bilipschitz}, thus continuous. Moreover, the map $\iota$ is weakly measurable. By the Pettis measurability theorem\footnote{
 A mapping from a measure space to a separable Banach space is weakly measurable if and only if it is strongly measurable.
} \cite{Pettis}, and by the separability of the space $H^s$, this mapping is also strongly measurable.
As a result, for each fixed $\vec{\phi}$, the map $\cT_j (\cdot, \vec{\phi})\colon 
 ( \cB_0(\eps,\delta, \tR, s), \wt\cF )
\to (\R^{2\tN}, \cB(\R^{2\tN}) )$ is also measurable.

\medskip

($iii$) We now prove the 
existence of a random 
fixed point theorem of $ \cT $. In view of 
\eqref{eq:decomp_T} and  \eqref{eq:badboundtilde0},  
provided $\eps$ is sufficiently small, we have
\be\label{eq:badbound}
\max_{0 < |j| \leq \tN}|\cT_j (\omega , \vec{\phi} )|\leq 
\tN \, \varepsilon^{-3} \, , \quad \forall 
|t| \leq T_\varepsilon < \varepsilon^{-3}\, . 
\ee 
Define the set $K:=[-  \tN \, \varepsilon^{-3}, \tN \, \varepsilon^{-3}]^{2\tN}$.  
For each $\omega$, the bound 
\eqref{eq:badbound} guarantees that $\cT(\omega,\cdot):K\rightarrow K$   and  \Cref{lem:T_lips}-($i$) that 
each $\cT_j (\omega,\cdot):K\rightarrow K$ is continuous. By \Cref{lem:T_lips}-($ii$), each $\cT_j (\cdot,\phi):(\cB_0(\eps,\delta, \tR, s),\wt{\Fc} \cap \cB_0(\eps,\delta, \tR, s))\rightarrow \R^{2\tN}$ is  a random operator.
Then   the random Brouwer fixed point \Cref{thm:Brouwer} guarantees the existence of a random variable $\vec{\phi}^{\ast,\omega}$ in  $(\cB_0(\eps,\delta, \tR, s),\wt{\Fc} \cap \cB_0(\eps,\delta, \tR, s),\wt{\P})$ satisfying \eqref{eq:T_fixedp}.

Next, we extend the random variable $\vec{\phi}^{\ast,\omega}$ to the full probability space $(\mathbb{\Omega}, \widetilde{\mathcal{F}},\widetilde{\P})$ by setting $\vec{\phi}^{\ast,\omega} :=0$ for any $\omega\in\mathbb{\Omega}\setminus\cB_0(\eps,\delta, \tR, s)$. Note that, with this definition, $\vec{\phi}^{\ast,\omega}$ is $\wt {\mathcal{F}}$-measurable.

\medskip

Finally, we prove the properties $(a)$-$(c)$
of the fixed point  in \Cref{lem:T_lips}.

\smallskip

($a$) The map $F: \omega\mapsto (\vec{\phi}^{\omega},\vec{\phi}^{\ast,\omega})$ is $\cF$-measurable and since 
$G: \R^{\tN}\times\R^{\tN} \to \R_+ $,   
$ (a,b) \mapsto  \norm{a-b}_{\ell^{\infty}}$, 
is continuous, the composed map $G\circ F$ is measurable, as well as  $(G\circ F)^{-1} ([0,r))$ for any $r>0$. Consequently, the set $\Nc ( \alpha)$ in \eqref{eq:nhood_fixedp} is measurable. Its measure is $(2\alpha)^{2\tN}$ as follows by item $(c)$ with $\cA = \mathbb{\Omega}$.

\smallskip

($b$) For any $\omega \in \Nc (\alpha)\cap \cB_0(\eps,\delta, \tR, s)$, for any 
$ |t| \leq  \eps^{-3(1-\delta) + \kappa} $, we have, for any $ 0 < |j| \leq \tN $,  
\[
\begin{split}
\big| \phi_j^{\omega} - \cT_j (\omega , \vec{\phi}^{\omega})\big | & \stackrel{\eqref{eq:T_fixedp}}{\leq} 
\big | \phi_j -  \phi_j^{\ast,\omega} \big | + \big| \cT_j (\omega, \vec{\phi}^{\ast,\omega}) - \cT_j (\omega , \vec{\phi}^{\omega})\big| \\
&  \stackrel{\eqref{eq:nhood_fixedp},\eqref{eq:bilipschitz}}{\leq} \alpha +  \, C_2 \, \tN
\tR^2 \, \varepsilon^{2(1-\delta)}\, e^{C_{\mathrm{lip}} |t|} \,
\big\| \vec{\phi}^{\ast,\omega}-\vec{\phi}^{\omega}\big\|_{\ell^{\infty}} 
\stackrel{\eqref{eq:nhood_fixedp}}{\leq}  2\alpha\, \, \tN\,e^{C_{\mathrm{lip}} \varepsilon^{-3(1-\delta)}}
\end{split}
\]
for   $ \eps $ small.

\smallskip

($c$) Using the standard  properties of conditional expectation in \Cref{rk:cond_exp},
for any  $\cA \in \widetilde \cF$ we have 
\be\label{259:1648}
\P  \left(
\cA \cap \Nc (\alpha)\right) = \E \left[\1_\cA \, \1_{\Nc (\alpha)} \right] = 
\E \left[ \E\left[ \1_\cA \, \1_{\Nc (\alpha)} \, | \widetilde \cF \right]\right] 
=  \E \left[ \1_\cA \, \E\left[  \1_{\Nc (\alpha)} \, | \widetilde \cF \right]\right] \, . 
\ee
Consider now the random variable
$\E \big[ \1_{\Nc (\alpha)} \, | \widetilde \cF \big]$. 
We wish to compute it explicitly using 
Lemma \ref{thm:ricarsnick} with  $\cG = \widetilde \cF$, $X =\vec{\phi}^{\omega}$, $Y =\vec{\phi}^{\ast,\omega}$ and $f\colon \R^{2\tN}\times \R^{2\tN} \to \R $ is the bounded function defined by
$$
f(x,y) := \1_{\widetilde \cN(y;\alpha)}(x) \, , \quad 
\widetilde{\Nc} (y;\alpha) := \left\lbrace \vec\phi  \in\R^{2\tN} \ \big | \ \|\vec\phi - y\|_{\ell^\infty}< \alpha \right\rbrace \, , \quad 
\forall y \in \tR^{2 \tN} \, .
$$
The assumptions are met since 
$\widetilde \cF \subset \cF$,  $Y= \vec{\phi}^{\ast,\omega}$ is $\widetilde \cF$-measurable according to \Cref{lem:T_lips}-($a$), and  
 the random variable  $X= \vec{\phi}^{\omega}$ is  independent from $\widetilde \cF$, see
 \Cref{rk:independence}. 
Then \Cref{thm:ricarsnick} yields
 $$
\E\left[  \1_{\Nc (\alpha)} \, | \widetilde \cF \right]
= g(\vec{\phi}^{\ast,\omega})
\quad
\mbox{ with }
\quad 
 g(y):= \E\left[ f( \vec{\phi}^{\omega}, y)] \right] = \left(\frac{ \alpha}{\pi}\right)^{2\tN} \ . $$
Plugging into \eqref{259:1648} proves \eqref{splitting}.
The proof of \Cref{lem:T_lips} is complete. 

\subsection{The lower bound}

We  now prove the lower bound for \eqref{eq:LDP} in \Cref{thm:main} $(ii)$.

\begin{prop}[{\bf LDP: Lower-bound on optimal timescale}]\label{thm:lo_LDP_improved}
There exists $\und{s}>0$ 
such that
for any  $\lambda_0 >0$, $\delta \in (\frac35,1)$, and $s \geq \und{s}$ there exists $\tR_0  >0$ and  for any $\tR\geq \tR_0$,   any $0<\kappa \ll 1-\delta$, the following holds true. There exists $\eps_0 > 0 $ such that for   any $\eps \in (0, \eps_0)$, 
for any random initial data $(\eta_0^\omega, \psi_0^\omega) \in  \cB_0(\eps, \delta, \tR , s)$ in \eqref{cB0}, 
the solution $(\eta(t,x), \psi(t,x))$ of the water waves equations \eqref{ww} exists up to times $|t| \leq T_\eps:= \eps^{-3(1-\delta)+\kappa}$, and satisfies 
the lower bound 
\begin{equation}\label{eq:lo_LDP_improved}
 \liminf_{\eps\to 0+} \eps^{2\delta} \log \P \left(
\left\{ \sup_{x\in\T} \eta (t,x) \geq \lambda_0 \varepsilon^{1-\delta}\right\} \cap \cB_0(\eps,\delta, \tR, s) \right) \geq  -\frac{\lambda_0^2}{2 \bsi^2} \, ,
\end{equation}
with $\bsi$ in \eqref{upsigma}.
\end{prop}
\begin{proof}
 Let $s \geq \und{s} := N+\frac54$ (with $ N $ given by \Cref{BNFtheorem}) and $\tR \geq \tR_0$ with  $\tR_0$ defined  in \eqref{R0}.
For any $ \varepsilon > 0 $ 
we choose 
\begin{equation}\label{N.choice}
\tN := \tN_\eps :=  \lfloor\eps^{-\gamma} \rfloor  \in \N \ ,  \quad 0<\gamma<5\delta-3 \, , 
\end{equation}
which is possible since $\delta>3/5$. 
By \Cref{thm:approx2}, 
provided $\eps$ is small enough, 
 any initial datum $(\eta_0^\omega, \psi_0^\omega) \in \cB_0(\eps, \delta, \tR, s)\equiv \cB_0$  gives rise to a  water waves solution defined 
 for any  $|t| \leq T_\eps $, whose wave profile 
is uniformly approximated by 
$ \eta _{\rm app2}(t,x)$ defined in  \eqref{etaapp2}, as stated in \eqref{upxi_error2}. 
By  the triangle inequality 
\be\label{sandwich2}
\sup_{x \in \T} \eta(t,x) \geq 
\sup_{x\in \T} \eta_{\rm app2}(t,x) - \norm{\eta(t)-\eta_{\mathrm{app2}}(t)}_{L^\infty(\T)}  
\ee
and $ \eta(t,x) $ becomes large
if the approximate profile 
$\eta_{\rm app2}(t,x)$ does. 
We define  the event
$$
\Cc_\eps  := \left\{ \norm{\eta(t)-\eta_{\mathrm{app2}}(t)}_{L^\infty(\T)} \geq \und{\tC}\,\lambda_0  \tR^4 \, \,  \eps^{1-\delta + \kappa} \right\} \, , 
\quad
\und{\tC}:= 2 {\tC_2} \lambda_0^{-1} \quad \text{with} 
\quad \tC_2 \mbox{ in } \eqref{upxi_error2} \, .
$$
By \Cref{thm:approx2},
for any $|t|\leq T_\eps  $ we have the inclusion
$ \cB_0  \subseteq  \Cc_\eps^c $, hence
$\cB_0 = \cB_0 \cap \cC_\eps^c$, and 
by \eqref{sandwich2}
\begin{align}
\notag
\P &\Big(
\Big\{ \sup_{x\in\T} \eta (t,x) \geq \lambda_0 \varepsilon^{1-\delta}\Big\} \cap \cB_0  \Big)\\
\notag
&  \geq 
 \P \Big(
\Big\{ \sup_{x\in\T} \eta_{\rm app2} (t,x) 
-\norm{\eta(t) - 
\eta_{\rm app2}(t)}_{L^\infty(\T)}
\geq \lambda_0 \varepsilon^{1-\delta}\Big\} \cap \cB_0  
\cap \cC_\eps^c 
\Big)
\\
&\geq \P \Big(
\Big\{ \sup_{x\in\T} \eta_{\mathrm{app2}} (t,x) 
\geq 
\lambda_0\varepsilon^{1-\delta}
 (1+\und{\tC}\tR^4\,\varepsilon^{\kappa}) 
\Big\} \cap \cB_0 
\Big) \label{1909:1645-1} \\
&
\geq  \P \Big(
\Big\{
\sup_{x\in\T} \eta_{\mathrm{app2}} (t,x) 
\geq 
\lambda_0\varepsilon^{1-\delta}
 (1+\und{\tC}\tR^4\,\varepsilon^{\kappa}) 
\Big\} 
\cap \cB_0   \cap \Nc (\alpha) \Big)  \label{1909:1645}
\end{align} 
where
$ \Nc (\alpha)$ is the measurable 
set in \eqref{eq:nhood_fixedp} with 
\be\label{choice.alpha}
\alpha:= (2\tN_\eps)^{-1} \,  \eps\,  e^{ -C_{\rm lip} \varepsilon^{-3(1-\delta)}}    \ . 
\ee
With this choice of $\alpha$, condition \eqref{alpha.N} is fulfilled and then we apply 
\Cref{cor:phase_sync}  (with $\tN_\eps $ in \eqref{N.choice})
to deduce that
\begin{align}\notag
\eqref{1909:1645} 
& \stackrel{\eqref{cor:eta.syn}}{ \geq }
 \P \Big(
\Big\{\sum_{0<|j| \leq  \tN_\eps}   c_j R_j^{\omega}\, \geq 
\sqrt{\pi}\,\lambda_0\frac{\varepsilon^{-\delta}}{1-\eps} (1+\und{\tC}\tR^4\,\varepsilon^{\kappa}+C \lambda_0^{-1} \eps^{\gamma(s- 1)}\tR )\Big\} \cap \cB_0  \cap \Nc (\alpha)\Big) \\
\label{2509:1636}
&  \geq
 \P \Big( 
 \underbrace{
\Big\{\sum_{0<|j| \leq  \tN_\eps}   c_j R_j^{\omega}\, \geq 
\,\sqrt{\pi}\, \lambda_0
\varepsilon^{-\delta} (1+ o_\varepsilon(1) ) \Big\} \cap \cB_0
}_{=: \cA}  \cap \, \Nc (\alpha)\Big)  \, . 
\end{align}
The set $\cA$ depends only on the random variables $\big(\vec{R}^{\omega},\Pi_{>\tN}  \wt\eta_0^{\omega},\Pi_{>\tN} \wt\psi_0^{\omega}\big)$ and  therefore is 
$\wt\cF$-measurable (recall \eqref{259:1639}). In addition we claim that 
\be\label{prob.cA} 
\and \P(\cA) \geq  \,\exp\Big( -\frac{\lambda_0^2\, \eps^{-2\delta}}{2\bsi^2} +o(\varepsilon^{-2\delta})\Big) \, . 
\ee
Assuming this lower bound for the moment,  let us conclude the proof of \eqref{eq:lo_LDP_improved}.
We  give a lower bound on the right-hand side of \eqref{2509:1636}. 
By \eqref{splitting}
we have
$ \P(\cA \cap \Nc (\alpha))
 = 
 \left(\frac{ \alpha}{\pi}\right)^{2\tN_\eps} \, \P(\cA)$. 
By the choice of $\tN_\eps $ in \eqref{N.choice}, of $\alpha$ in \eqref{choice.alpha} and that    $0<\gamma< 5 \delta -3$ (cf. \eqref{N.choice}) we have 
\be\label{1909:1828}
0\leq -2\tN_\eps \log(\alpha/\pi) = 
\cO (\eps^{-3(1-\delta)-\gamma})
=o(\eps^{-2\delta}) 
\quad
\Rightarrow
\quad 
\left(\frac{ \alpha}{\pi}\right)^{2\tN_\eps} = \exp( - o(\eps^{-2\delta})) \, .
\ee
Hence  by  \eqref{1909:1828} and \eqref{prob.cA}  we deduce that 
\begin{align}
\label{1909:1726}
\eqref{1909:1645} \geq  \P(\cA \cap \Nc (\alpha)) 
& { \geq}\exp \Big( -\frac{\lambda_0^2\, \eps^{-2\delta}}{2\bsi^2} - o(\varepsilon^{-2\delta})
\Big) \, . 
\end{align}
Finally we conclude by  \eqref{2509:1636}
 and \eqref{1909:1726}, that 
\begin{align*}
\P \Big(
\Big\{ \sup_{x\in\T} \eta (t,x) \geq \lambda_0 \varepsilon^{1-\delta}\Big\} \cap \cB_0  \Big)
&  \geq
\exp\Big( -\frac{\lambda_0^2\,\eps^{-2\delta}}{2\bsi^2} - o(\varepsilon^{-2\delta})
\Big) 
\end{align*}
and the desired bound \eqref{eq:lo_LDP_improved} follows.
\\[1mm]
{\sc Proof of  \eqref{prob.cA}:}
By the very definition of $\cA$ in \eqref{2509:1636}, we estimate 
\be\label{finaste}
\P(\cA)  \geq \P\Big( \Big\{ 
\sum_{0<|j|\leq \tN_\eps }   c_j R_j^{\omega}\, \geq \sqrt{\pi}\lambda_0\, 
\varepsilon^{-\delta} 
(1+o_\varepsilon (1) ) \Big\}  
\Big) - \P(\cB_0^c )  
\ee
and we bound the two probabilities separately. 

For any $\tM \in \N$, for any  $\eps $ small enough we have 
$  \tN_\eps \geq \tM $, and then 
\begin{align*} 
 \P  \Big( \Big\{ \sum_{0<|j|\leq \tN_\eps}   c_j R_j^{\omega} \geq \sqrt{\pi}\lambda_0\,
\varepsilon^{-\delta} 
(1+ o_\varepsilon(1) )\,  \Big\}  
\Big) 
\geq
 \P  \Big( \Big\{ \sum_{0<|j|\leq \tM}   c_j R_j^{\omega} \geq \sqrt{\pi}\lambda_0\,
\varepsilon^{-\delta} 
(1+ o_\varepsilon(1) )\,  \Big\}  
\Big)  \, . 
\end{align*}
We apply  Lemma \ref{thm:LDP_Rayleigh} to the right hand side with a sequence $(c_j)_{j \in \Z\setminus \{0\}}$ with $c_j =0$ for any $|j| > \tM$, getting 
\begin{align*} \label{269:1735}
\liminf_{\eps \to 0} \eps^{2\delta}\log \P  \Big( \Big\{ \sum_{0<|j|\leq \tN_\eps}   c_j R_j^{\omega} \geq \sqrt{\pi}\lambda_0\,
\varepsilon^{-\delta} 
(1+ o_\varepsilon(1) )
\,  \Big\}  
\Big) 
\geq  -\frac{\pi\,\lambda_0^2}{ \sum_{0 < |j| \leq \tM} c_j^2} \, ,
\quad \forall \tM \in \N \, . 
\end{align*}
Then  taking $\tM \to \infty$ 
and recalling \eqref{upsigma} yields the lower bound
\be\label{2909:1916}
\liminf_{\eps \to 0} \eps^{2\delta}\log \P  \Big( \Big\{ \sum_{0<|j|\leq \tN_\eps}   c_j R_j^{\omega} \geq \sqrt{\pi}\lambda_0 (1+ o_\varepsilon(1) ) \,
\varepsilon^{-\delta} 
\,  \Big\}  
\Big) 
\geq  -\frac{2\,\lambda_0^2}{ \bsi^2}    \ . 
\ee
By \eqref{terrone}, the probability $\P(\cB_0^c ) 
\leq - 
\exp\left( -\frac{\tR^2\,\varepsilon^{-2\delta}}{4\norm{\vec c }_{h^s}^2}+1 \right)  $. 
Taking $\tR \geq \tR_0 >  2 \lambda_0 \| \vec c \|_{h^s}  \bsi^{-1}$ (cf. \eqref{R0}),
shows that $\P(\cB_0^c) \leq \frac12\P\Big( \Big\{ 
\sum_{0<|j|\leq \tN_\eps}   c_j R_j^{\omega}\, \geq \sqrt{\pi}\lambda_0\, 
\varepsilon^{-\delta} 
(1+o_\varepsilon (1) ) \Big\}  
\Big) $ for $\eps$ small enough. 
In view of  \eqref{finaste}, \eqref{2909:1916}, the claimed estimate    \eqref{prob.cA} follows.
\end{proof}

\subsection{Dispersive focusing }\label{sec:allr}

In this section we prove \Cref{thm:df}. 
The upper bound in \eqref{0210:1548} follows by \Cref{thm:main_intro} since $\P (
\mathfrak{R}^t(\lambda_0 \varepsilon^{1-\delta}) \cap 
\mathfrak{P}^t(\tM,\eps^{\frac{\kappa}{5}}))\leq \P (
\mathfrak{R}^t(\lambda_0 \varepsilon^{1-\delta}))$.
To establish the lower bound in \eqref{0210:1548}, we proceed as follows. 
Approximating again $\eta(t,x)$ with $\eta_{\rm app2}(t,x)$ (defined in \eqref{etaapp2}) as we did to obtain the inequality in \eqref{1909:1645-1}, we get 
\be\label{0210:1451}
\P \Big(
\mathfrak{R}^t(\lambda_0 \varepsilon^{1-\delta}) \cap 
\mathfrak{P}^t(\tM,\eps^{\frac{\kappa}{5}})
\Big) \geq \P \Big(
\mathcal{A}_2^t\left(\lambda_0 \varepsilon^{1-\delta}(1+o_{\eps} (1))\right) \cap 
\mathfrak{P}^t(\tM,\eps^{\frac{\kappa}{5}})\cap \cB_0\Big) 
\ee
where  $\cB_0 := \cB_0(\eps, \delta, \tR, s)$ as in \eqref{cB0}-\eqref{cB0p}, and
\begin{equation}\label{defA2app}
\mathcal{A}_2^t(\lambda):=\Big\{ \omega\in\mathbb{\Omega} \ \big |\ \sup_{x\in\T} \eta_{\mathrm{app2}}(t,x) \geq \lambda \Big\}.
\end{equation}
Next we define the event 
\be\label{evR}
\mathcal{R}(\tM,\varepsilon) := \Big\{ \omega\in\mathbb{\Omega} \mid  R_k^{\omega}\geq \eps^{-\delta+\frac{\kappa}{2}}\, , 
\quad \forall 0<|k|\leq\tM \Big\} \, .
\ee
We  claim that

\smallskip

($i$) $\mathcal{R}(\tM,\varepsilon)^c$ is negligible for rogue waves, i.e.
\begin{equation}\label{0210:1452}
\limsup_{\eps\to 0^{+}} \varepsilon^{2\delta} \log \P \Big( \mathcal{A}_2^t\big(
\lambda_0 \varepsilon^{1-\delta}(1+o_{\eps} (1)) \big) \cap \mathcal{R}(\tM,\varepsilon)^c \Big) < - \frac{\lambda_0^2}{2\bsi^2} \, .
\end{equation}

($ii$) Within the set $ \cN(\alpha) $
defined in \eqref{eq:nhood_fixedp} and with $
\alpha $ in \eqref{choice.alpha}, 
a lower bound on the moduli $ R_j^\omega $
implies the phases are synchronized, i.e. 
\begin{equation}\label{0210:1455}
 \mathfrak{P}^t(\tM,\eps^{\frac{\kappa}{5}})\cap \cB_0 \supset 
\mathcal{N} (\alpha) \cap 
\mathcal{R}(\tM,\varepsilon)\cap \cB_0 \, . 
\end{equation}
Assuming ($i$)-($ii$), we can prove the sharp lower bound \eqref{0210:1548}. Indeed
\begin{align*}
    \eqref{0210:1451} & \stackrel{\eqref{0210:1455}}{\geq }
    \P \Big(\mathcal{A}_2^t\left(\lambda_0 \varepsilon^{1-\delta}(1+o_\eps (1))\right) \cap \mathcal{N} (\alpha) \cap 
\mathcal{R}(\tM,\varepsilon) \cap \cB_0\Big)  \\
& \geq  \P \Big(\mathcal{A}_2^t\left( \lambda_0 \varepsilon^{1-\delta}(1+o_\eps (1))\right) \cap \mathcal{N} (\alpha)\cap \cB_0  \Big) - 
\P\Big(\mathcal{A}_2^t\left(\lambda_0 \varepsilon^{1-\delta}(1+o_{\eps} (1))\right) \cap \mathcal{R}(\tM,\varepsilon)^c
\Big) \\
& \stackrel{\eqref{1909:1726}}{\geq}
\exp \Big( -\frac{\lambda_0^2\, \eps^{-2\delta}}{2\bsi^2} - o(\varepsilon^{-2\delta})
\Big) - \P\left(\mathcal{A}_2^t\left((\lambda_0 \varepsilon^{1-\delta}(1+o_{\eps} (1))\right) \cap \mathcal{R}(\tM,\varepsilon)^c
\right)\\
& \stackrel{\eqref{0210:1452}}{\geq} \exp \Big( -\frac{\lambda_0^2\, \eps^{-2\delta}}{2\bsi^2} - o(\varepsilon^{-2\delta})
\Big)
\end{align*}
proving \eqref{0210:1548}. We now prove the claims ($i$) and ($ii$).
\smallskip

($i$) {\sc  Proof of \eqref{0210:1452}.}
Recalling \eqref{evR} 
we have 
\begin{align}
\P \Big( \mathcal{A}_2^t\left(\lambda_0 \varepsilon^{1-\delta}(1+o_{\eps} (1))\right) \cap \mathcal{R}(\tM,\varepsilon)^c \Big) & \leq \sum_{0<|k|\leq \tM} \P \Big( \mathcal{A}_2^t\left(\lambda_0 \varepsilon^{1-\delta}(1+o_{\eps} (1))\right) \cap \left\{ R_k^{\omega} < 
\eps^{-\delta+\frac{\kappa}{2}} \right\} \Big) \notag \\
& \hspace{-1cm} \stackrel{\eqref{defA2app}, \eqref{eq:eta2notg}}
\leq \!\!\!\!\!\!
\sum_{0<|k|\leq \tM} 
\!\!\! \P \Big ( \Big \{ \sum_{j\in\Z\setminus\{0\}} c_j R_j^{\omega} \geq \sqrt{\pi}\lambda_0 \varepsilon^{-\delta} (1+o_{\eps} (1))\Big \} \cap \left\{ R_k^{\omega} < \eps^{-\delta+\frac{\kappa}{2}} \right\} \Big ) \notag \\
&  \hspace{-1cm} \leq \sum_{0<|k|\leq \tM} \P \Big( \Big\{ \sum_{j\neq k} c_j R_j^{\omega} \geq \sqrt{\pi}\lambda_0 \varepsilon^{-\delta} (1+o_{\eps} (1))\Big\} \Big) \notag \\
&  \hspace{-1cm} \stackrel{\Cref{thm:LDP_Rayleigh}}{\leq}  \sum_{0<|k|\leq \tM}  \exp \Big( - \frac{\pi \lambda_0^2}{\sum_{j\neq k} c_j^2} \, \eps^{-2\delta} + o (\eps^{-2\delta} )\Big) \notag \\
&   \hspace{-1cm} \leq \tM \, \exp \Big( - \frac{\pi \lambda_0^2}{\sum_{j\neq\tM} c_j^2} \, \eps^{-2\delta} + o (\eps^{-2\delta} )\Big) \label{upplims}
\end{align}
where in the  last step we used
$ \sum_{j\neq\tM} c_j^2\geq {\sum_{j\neq k} c_j^2}$ for any $0<|k|\leq \tM$, since the $ c_j $ in \eqref{ckdk} are decreasing.
Taking the logarithm in \eqref{upplims}, multiplying by  $\eps^{2\delta}$, and taking the limsup as $\eps \to 0$, yields \eqref{0210:1452},  recalling also  that $ \bsi^2 > \frac{1}{2\pi} 
\sum_{j \neq \tM} c_j^2 $ by definition, see \eqref{upsigma}. 

\smallskip

($ii$) {\sc Proof of \eqref{0210:1455}.} 
Take  $\omega \in \mathcal{N} (\alpha) \cap 
\mathcal{R}(\tM,\varepsilon)\cap \cB_0$, cf.
\eqref{eq:nhood_fixedp}, \eqref{evR}, \eqref{cB0p}. We  wish to prove that $\omega \in \mathfrak{P}^t(\tM, \eps^{\frac{\kappa}{5}})$ defined in 
\eqref{eventp}. 
Expand in Fourier the
water waves solution  with random initial data \eqref{eq:init_data},
$$
u(t,x) = \frac{1}{\sqrt{2\pi}} \sum_{j \in \Z\setminus \{0\}} u_j(t) e^{\im j x} = 
 \frac{1}{\sqrt{2\pi}} \sum_{j \in \Z\setminus \{0\}}  |u_j(t)| e^{\im \theta_j(t)} e^{\im j x} \, , 
$$ 
provided by Theorem \ref{cor:LIP}.  
Also recall that 
by \eqref{etaapp2}-\eqref{eq:theta} the $ j$-th 
Fourier coefficient of the approximate solution $ u_{\rm app2} (t,x) $ is equal to 
$ (u_{\rm app2})_j = \varepsilon c_j R_j^\omega |j|^{-\frac14} e^{\im \vartheta_j(t;\omega)} $. Then, 
by the triangle inequality, 
 the continuous embedding 
 $ \cF L^{\frac14,1} \hookrightarrow H^1 (\T)$
 and \Cref{thm:approx2}, we get that
for any  $\omega \in \cB_0$,
 for any  $ |t| \leq \eps^{-3(1-\delta) +\kappa} $, 
\be\label{flchec}
\sum_{j \in \Z\setminus \{0\}}  \Big|  
|j|^{\frac14}|u_j(t)| - \underbrace{\eps c_j R_j^\omega}_{=|j|^{\frac14}|(u_{\rm app2})_j|} \Big| 
\lesssim \norm{u(t,\cdot ) - u_{\rm app2}(t, \cdot )}_{\cF L^{\frac14,1}} \lesssim \tR^4 \eps^{1-\delta + \kappa}  \, . 
\ee
Then the triangle inequality yields 
\begin{align}
\sum_{j \in \Z\setminus\{0\}}\!\!\!\! \eps c_j R_j^\omega \abs{e^{\im \theta_j(t)}  - e^{\im \vartheta_j(t;\omega)}}
&  \lesssim  \sum_{j \in \Z\setminus \{0\}} \Big (
 \abs{\eps c_j R_j^\omega-|j|^{\frac14} |u_j(t)|}
 +
 \Big| |j|^{\frac14} \underbrace{|u_j(t)| e^{\im \theta_j(t)}}_{= u_j(t)} - |j|^{\frac14} (u_{\rm app2})_j \Big|\Big ) \notag  \\
&   \stackrel{\eqref{flchec}} \lesssim
\norm{u(t,\cdot ) - u_{\rm app2}(t, \cdot )}_{\cF L^{\frac14,1}} \!\!\!\!  \stackrel{\eqref{flchec}}
\lesssim  \tR^4 \eps^{1-\delta + \kappa} \label{eq:trian} \, . 
\end{align}
Exploiting that $\omega \in \mathcal{R}(\tM,\varepsilon)$ (cf. \eqref{evR}) the estimate \eqref{eq:trian} implies that
\be\label{cjleq}
c_j \,\big| e^{\im \theta_j(t)}  - e^{\im \vartheta_j(t;\omega)} \big|  \lesssim \eps^{\kappa/2}  \, , \quad \forall 0<|j|\leq \tM \, .
\ee
Finally since $\omega\in \cN (\alpha)$, 
with $
\alpha $ in \eqref{choice.alpha}, 
\Cref{cor:phase_sync} yields $|\vartheta_j(t;\omega)|\leq \eps$ for any $0<|j|\leq \tM$ (for $\eps$ small enough). As a result, inserting in \eqref{cjleq} we deduce 
$|e^{\im \theta_j(t)} - 1| \lesssim \eps^{\kappa/2}$  for any $ 0<|j|\leq \tM$, meaning that
$\omega \in \mathfrak{P}^t(\tM, \eps^{\frac{\kappa}{5}})$ for $ \varepsilon $  small enough  
(recall \eqref{eventp}). This proves  \eqref{0210:1455}.
\qed

\appendix

\section{Proof of the normal form theorem}\label{app:normalform}

In this section we prove Theorem \ref{BNFtheorem}. 
It is a variant of the 
paradifferential Birkhoff normal form result in  \cite{BFP}. We first  introduce paradifferential calculus following \cite{BD,BFP},  reporting only definitions and properties  that will be used along the proof of  \Cref{BNFtheorem}. 

\subsection{Paradifferential calculus}\label{app:para}

Given an interval $ I = (-T,T)\subset \R$ 
and $s\in \R$,
we define the space of functions 
$$
C_*^K(I,\dot{H}^s\left(\mathbb{T},\mathbb{C}^2)\right) := 
\bigcap_{k=0}^K C^k\big(I, \dot{H}^{s-  k}(\mathbb{T},\mathbb{C}^2)\big)
$$ 
endowed with the  norm 
$$
\sup_{t\in I} \| U(t, \cdot)\|_{K,s} \qquad 
{\rm where} \qquad 
\| U(t, \cdot)\|_{K,s}:= {\mathop \sum}_{k=0}^K \| \partial_t^kU(t, \cdot)\|_{\dot{H}^{s- k}} \, , 
$$
and its subspace 
$$
C_{*\R}^K(I,\dot{H}^s\left(\mathbb{T},\mathbb{C}^2)\right)
:= \Big\{ U\in C_*^K(I,\dot{H}^s\left(\mathbb{T},\mathbb{C}^2)\right) \ : \ 
U=\begin{pmatrix} u \\ \bar{u}\end{pmatrix}\Big\} \, .
$$  
The parameter $ s $  will denote the space Sobolev regularity of the water waves solution $ U(t, \cdot) $
and $ K $ its time regularity. Since
the gravity  water waves vector field  
 loses $ 1$-derivative, $ \pa_t^k U(t, \cdot ) $ will 
 belong to $ {\dot H}^{s-k}$. We denote $ \dot H^s_\R (\T,\C^2 )$
 the subspaces of functions $ U = \vect{u}{\bar u} $ of $ \dot H^s (\T,\C^2 )$.

Given $r>0$ we denote the ball  
\be\label{ball}
B^K_{s, \R}(I;r):= \left\{U \in C_{*\R}^K(I,\dot{H}^s\left(\mathbb{T},\mathbb{C}^2)\right) \colon \ \ \sup_{t\in I} \| U(t, \cdot)\|_{K,s} \leq r \right\}
\ee 
and similarly by $B^K_{s}(I;r)$
the ball of radius $r$ in $C_*^K(I,\dot{H}^s\left(\mathbb{T},\mathbb{C}^2)\right)$.

 Given a linear operator  $ R(U) [ \cdot ]$ acting on $\dot L^2(\T,\C)$ 
we associate the linear  operator  
$\ov{R}(U)[v] := \ov{R(U)[\ov{v}]} $\ for all    $v: \T \rightarrow \C$.
We say that a matrix of operators acting on $\dot L^2(\T,\C^2)$  is \emph{real-to-real}, if it has the form 
\begin{equation}\label{vinello}
R(U) =
\left(\begin{matrix} R_{1}(U) & R_{2}(U) \\
\ov{R_{2}}(U) & \ov{R_{1}}(U)
\end{matrix}
\right) \, , \quad \forall  \,  U = \vect{u}{\ov{u}} \, .
\end{equation}
Similarly we say that a vector field  
$$
X(U): =  \vect{X(U)^+}{X(U)^-}  \quad \text{is real-to-real if} \quad
\bar{X(U)^+}=X(U)^- \, , \quad \forall \, 
U = \vect{u}{\ov{u}}  \, .  
$$
If $\cU := (U_1, \ldots, U_p)$ is a $p$-tuple of functions, $\vec{n} := (n_1, \cdots, n_p) \in \N^p$, we set
$$
\Pi_{\vec n}\cU := (\Pi_{n_1} U_1 , \cdots , \Pi_{n_p} U_p ) 
\, , \quad (\Pi_n u)(x) := \tfrac{1}{\sqrt{2\pi}} u_n e^{\im n x} + \tfrac{1}{\sqrt{2\pi}} u_{-n} e^{-\im n x}  \, . 
$$
\paragraph{\bf Paradifferential operators.}

We  introduce  first 
 paradifferential  operators (Definition \ref{quantizationtotale}) and then 
smoothing operators (Definition \ref{Def:Maps}).

\smallskip
\noindent{\bf Classes of symbols.}
 In the definition below we introduce two classes of symbols  $\wt{\Gamma}_p^m$  and $\Gamma_{K,K',p}^m$. Roughly speaking the class $\wt{\Gamma}_p^m$ contains symbols of order $m$ and homogeneity $p$ in $U$, whereas the class $\Gamma_{K,K',p}^m$ contains non-homogeneous symbols of order $m$ that vanishes at degree at least $p$ in $U$ and that are $(K-K')$-times differentiable in $t$; we can think of  the parameter $K'$ like the number of time derivatives of $U$ that are contained in the symbols. 
In the following we denote $ \dot{H}^{\infty}(\mathbb{T},\mathbb{C}^2)
:= \bigcap_{s \in \R} \dot{H}^{s}(\mathbb{T},\mathbb{C}^2)$.
\begin{defn}
Let $m\in \R$, $p, \tN\in \N_0 $, 
$ K, K' \in \N_0 $ with $ K' \leq K  $, and $ r>0$.\\
($i$) $p$-{\bf homogeneous symbols.} We denote by $\wt{\Gamma}^m_p$ the space of symmetric $p$-linear maps 
from $ (\dot{H}^{\infty}(\mathbb{T},\mathbb{C}^2))^p$ to the space of $ C^\infty $ functions from $\mathbb{T}\times \R$ to $\mathbb{C}$, 
$ (x, \xi) \mapsto a(U_1, \ldots, U_p;x,\xi)$,  satisfying the following: there exist $\mu>0$ and, for any $\alpha, \beta\in \N_0$, 
there is a constant $C>0$ such that 
\begin{equation}\label{homosymbo}
|\partial_x^{\alpha}\partial_{\xi}^{\beta}  
a( \Pi_{\vec n} \cU ;x,\xi)|
\leq 
C |\vec{n}|^{\mu+\alpha} \langle \xi \rangle^{m-\beta} 
\prod_{j=1}^p \| \Pi_{n_j} U\|_{L^2}
\end{equation}
for any $ \cU = (U_1,\dots,U_p)\in ( \dot{H}^{\infty}(\mathbb{T},\mathbb{C}^2) )^p$ and $\vec{n}=(n_1,\dots,n_p)\in \mathbb{N}^p$. 
 In addition we require the translation invariance property
$ a\left( \tau_{\varsigma} \cU; x,\xi\right)= a\left( \cU; x+\varsigma, \xi\right) $, for any $ \varsigma\in \R $, 
where $\tau_\varsigma$ is the translation operator in \eqref{X.tra0}. 
For $ p = 0 $ we denote by $\wt{\Gamma}^m_0 $ the space of constant coefficients symbols $ \xi \mapsto a(\xi) $ which satisfy \eqref{homosymbo} with $ \alpha = 0 $ and the right hand side replaced by $ C \la \xi \ra^{m - \beta} $. 

 We denote by $\sum\limits_p^\tN \widetilde \Gamma^{m}_q$ the class of pluri-homogeneous symbols $\sum_{q=p}^{\tN}a_{q}$ with $a_q \in  \widetilde{\Gamma}_{q}^m$.

($ii$) {\bf Non-homogeneous symbols. }   We denote by $\Gamma_{K,K',p}^m[r]$ the space of functions  $ (U;t,x,\xi)\mapsto a(U;t,x,\xi) $, 
defined for $U\in B_{s_0}^{K'}(I;r)$ for some $s_0$ large enough, with complex values, such that for any $0\leq k\leq K-K'$, any $\ts\geq s_0$, there are $C>0$, $0<r(\ts)<r$ and for any $U\in B_{s_0}^K(I;r(\ts))\cap C_{*}^{k+K'}(I, \dot{H}^{\ts}(\mathbb{T},\mathbb{C}^2))$ and any $\alpha,\beta \in \N_0$, with $\alpha\leq \ts-s_0$ one has the estimate
$$
| \partial_t^k\partial_x^\alpha\partial_\xi^\beta a(U;t,x,\xi)| \leq C \langle \xi \rangle^{m-\beta} \| U\|_{k+K',s_0}^{p-1}\|U\|_{k+K',\ts} \, .
$$
If $ p = 0 $ the right hand side has to be replaced by $ C \langle \xi \rangle^{m-\beta} $. 

($iii$) {\bf Symbols.} We denote by $\Sigma \Gamma_{K,K',p}^m[r,\tN]$ the space of functions 
$ (U;t,x,\xi) \mapsto a(U;t,x,\xi), $
with complex values such that there are homogeneous symbols $a_q\in \wt{\Gamma}_q^m$, $q=p,\dots, \tN$ and a non-homogeneous symbol $a_{>\tN}\in \Gamma_{K,K',\tN+1}^m$ such that 
\be\label{espsymbol}
a(U;t,x,\xi)= {\mathop \sum}_{q=p}^{\tN} a_q(U,\dots,U;x,\xi) + a_{>\tN}(U;t,x,\xi) \, .
\ee
We denote by  $\Sigma \Gamma_{K,K',p}^m[r,\tN]\otimes \mathcal{M}_2(\mathbb{C})$ the space of $2\times 2$ matrices with entries in  $\Sigma \Gamma_{K,K',p}^m[r,\tN]$.

We say that a symbol  $a(U;t,x,\xi) $ is \emph{real} if it is real valued for any 
$ U \in B^{K'}_{s_0,\R}(I;r)$.
\end{defn}

We also define classes of functions in analogy with our classes of symbols.

\begin{defn}{\bf (Functions)} \label{apeape} Let $p, \tN \in \N_0 $,  
 $K,K'\in \N_0$ with $K'\leq K$, $r>0$.
We denote by $\widetilde{\mathcal{F}}^{\R}_{p}$, resp. $\mathcal{F}_{K,K',p}^{\R}[r]$, 
$\Sigma\mathcal{F}_{K,K',p}^{\R}[r,\tN]$,
the subspace of $\widetilde{\Gamma}^{0}_{p}$, resp. $\Gamma^0_{K,K',p}[r]$, 
resp. $\Sigma\Gamma^{0}_{K,K',p}[r,\tN]$, 
made of those symbols which are independent of $\xi $ and
 real valued for any $ U \in B^{K'}_{s_0,\R}(I;r)$.
\end{defn}

\paragraph{Paradifferential quantization.}
Given $p\in \N_0$ we consider   functions
  $\chi_{p}\in C^{\infty}(\R^{p}\times \R;\R)$ and $\chi\in C^{\infty}(\R\times\R;\R)$, 
  even with respect to each of their arguments, satisfying, for $0<\delta\ll 1$,
\begin{align*}
&{\rm{supp}}\, \chi_{p} \subset\{(\xi',\xi)\in\R^{p}\times\R; |\xi'|\leq\delta \langle\xi\rangle\} \, ,\qquad \chi_p (\xi',\xi)\equiv 1\,\,\, \rm{ for } \,\,\, |\xi'|\leq \delta \langle\xi\rangle / 2 \, ,
\\
&\rm{supp}\, \chi \subset\{(\xi',\xi)\in\R\times\R; |\xi'|\leq\delta \langle\xi\rangle\} \, ,\qquad \quad
 \chi(\xi',\xi) \equiv 1\,\,\, \rm{ for } \,\,\, |\xi'|\leq \delta   \langle\xi\rangle / 2 \, . 
\end{align*}
For $p=0$ we set $\chi_0\equiv1$. 
We assume moreover that 
$$ 
|\partial_{\xi}^{\alpha}\partial_{\xi'}^{\beta}\chi_p(\xi',\xi)|\leq C_{\alpha,\beta}\langle\xi\rangle^{-\alpha-|\beta|} \, , \  \forall \alpha\in \N_0, \,\beta\in\N_0^{p} \, ,  
\ \ 
|\partial_{\xi}^{\alpha}\partial_{\xi'}^{\beta}\chi(\xi',\xi)|\leq C_{\alpha,\beta}\langle\xi\rangle^{-\alpha-\beta}, \  \forall \alpha, \,\beta\in\N_0 \, .
$$ 
The Weyl quantization  of a smooth symbol $ a (x, \xi) $ 
is the operator
acting on a
$ 2 \pi $-periodic function
$u(x)$ 
 as
$$
{\rm Op}^{W}(a)u=\frac{1}{\sqrt{2\pi}}\sum_{k\in \Z}
\Big(\sum_{j\in\Z}\hat{a}\big(k-j, \frac{k+j}{2}\big)\hat{u}(j) \Big)\frac{e^{\im k x}}{\sqrt{2\pi}}
$$
where $\hat{a}(k,\xi)$ is the $k^{th}-$Fourier coefficient of the $2\pi-$periodic function $x\mapsto a(x,\xi)$.

\begin{defn}{\bf (Bony-Weyl quantization)}\label{quantizationtotale}
If $a$ is a symbol in $\widetilde{\Gamma}^{m}_{p}$, 
respectively in $\Gamma^{m}_{K,K',p}[r]$,
we set
$$
a_{\chi_{p}}(\mathcal{U};x,\xi) := \sum_{\vec{n}\in \N^{p}}\chi_{p}\left(\vec{n},\xi \right)a(\Pi_{\vec{n}}\mathcal{U};x,\xi) \, , 
\quad a_{\chi}(U;t,x,\xi) :=\frac{1}{2\pi}\int_{\mathbb{R}}  
\chi (\xi',\xi )\hat{a}(U;t,\xi',\xi)e^{\im \xi' x}\di \xi'  \, ,
$$
where in the last equality $  \hat a $ stands for the Fourier transform with respect to the $ x $ variable, and 
we define the \emph{Bony-Weyl} quantization of $ a $ as 
$$
\opbw(a(\mathcal{U};\cdot)) := {\rm Op}^{W} (a_{\chi_{p}}(\mathcal{U};\cdot)) \, ,\qquad
\opbw(a(U;t,\cdot)) := {\rm Op}^{W} (a_{\chi}(U;t,\cdot)) \, .
$$
If  $a$ is a symbol in  $\Sigma\Gamma^{m}_{K,K',p}[r,\tN]$, 
we define its \emph{Bony-Weyl} quantization 
$$
\opbw(a(U;t,\cdot))= {\mathop \sum}_{q=p}^{\tN}\opbw(a_{q}(U,\ldots,U;\cdot))+\opbw(a_{>\tN}(U;t,\cdot)) \, . 
$$
 \end{defn}

\noindent
$\bullet$ 
The action of
$\opbw(a)$ on  homogeneous spaces only depends
on the values of the symbol $ a = a(U;t,x,\xi)$  for $|\xi|\geq 1$. Whenever 
we encounter a symbol that is not smooth at $\xi=0 $,
such as, for example, $a = g(x)|\x|^{m}$ for $m\in \R\setminus\{0\}$, or $ \sign (\xi) $, 
we will consider its smoothed out version
$\chi(\xi)a$, where
$\chi\in  C^{\infty}(\R;\R)$ is an even and positive cut-off function satisfying  
\begin{equation}\label{cutoff11}
\chi(\x) =  0 \;\; {\rm if}\;\; |\x|\leq \tfrac{1}{8}\, , \quad 
\chi (\x) = 1 \;\; {\rm if}\;\; |\x|>\tfrac{1}{4}, 
\quad  \pa_{\x}\chi(\x)>0\quad\forall  \x\in \big(\tfrac{1}{8},\tfrac{1}{4} \big) \, .
\end{equation}

\noindent
$ \bullet $
A matrix of paradifferential operators $ \opbw(A(U;t, x,\x))$ is real-to-real, i.e. \eqref{vinello} holds, if and only if 
the  matrix of symbols $A(U;t, x,\x)$ has the form 
$$
A(U;x,\x) =
\left(\begin{matrix} {a}(U;t, x,\x) & {b}(U;t, x,\x)\\
{\ov{b(U;t, x,-\x)}} & {\ov{a(U;t, x,-\x)}}
\end{matrix} 
\right)  \, .
$$
We now define  smoothing operators. 
Roughly speaking, the class $\widetilde{\mathcal{R}}^{-\vr}_{p}$ contains smoothing operators
which gain $\vr$ derivatives and are homogeneous of degree $p$ in $U$, while 
$\mathcal{R}_{K,K',p}^{-\vr}$ contains non-homogeneous $\vr$-smoothing operators which
vanish at degree at least $p$ in $U$, satisfy tame estimates  and are $(K-K')$-times differentiable in $t$.
The constant $ \mu $ in \eqref{hom:rest} 
takes into account  losses of derivatives in the ``low" frequencies.

\begin{defn}{\bf (Classes of smoothing operators)}\label{Def:Maps}
Let  $ \vr \geq 0 $,  $p,\tN \in \N$,  $ K, K'\in \N_0 $,  
 with $K'\leq K$, and $ r > 0 $.

(i) {\bf $p$-homogeneous  smoothing operators.} 
We denote by $\widetilde{\mathcal{R}}^{-\vr}_{p}$
 the space of $(p+1)$-linear operators $M$ 
 from $(\dot{H}^{\infty}(\T,\C^{2}))^{p}\times \dot{H}^{\infty}(\T,\C)$ to 
 $\dot{H}^{\infty}(\T,\C)$ which are symmetric
 in $(U_{1},\ldots,U_{p})$, of the form
$ (U_{1},\ldots,U_{p+1})\to R(U_1,\ldots, U_p)U_{p+1} $ 
that satisfy the following. There are $\mu\geq0$ and  $C>0$ such that 
 \be \label{hom:rest} 
\|\Pi_{n_0}R(\Pi_{\vec{n}}\mathcal{U})\Pi_{n_{p+1}}U_{p+1}\|_{L^{2}}\leq
 C\frac{{\rm \max}_2( n_1,\ldots, n_{p+1})^{\vr+\mu}}{\max( n_1,\ldots, n_{p+1})^{\varrho}} \prod_{j=1}^{p+1}\|\Pi_{n_{j}}U_{j}\|_{L^{2}} 
\ee
  for any  $ \mathcal{U}=(U_1,\ldots,U_{p})\in (\dot{H}^{\infty}(\T,\C^{2}))^{p}$, any 
 $ U_{p+1}\in \dot{H}^{\infty}(\T,\C) $,
 $ \vec{n} = (n_1,\ldots,n_p) $ in $  \N^{p}$, any $ n_0,n_{p+1}\in \N$, and where $\max_2$ denotes the second largest element among $n_1,\ldots , n_{p+1}$.
 In addition we require the translation invariance property
$ R( \tau_\varsigma {\cal U}) [\tau_\varsigma U_{p+1}]  =  
\tau_\varsigma \big( R( {\cal U})U_{p+1} \big)$ for any 
$ \varsigma \in \R $.

(ii) {\bf Non-homogeneous smoothing-operators.} 
  We denote by  $\mathcal{R}^{-\vr}_{K,K',p}[r]$ 
  the space of operators $(U,t,V)\mapsto R(U;t) V $ defined on $B^{K'}_{s_0}(I;r)\times I \times C^0_{*}(I,\dot{H}^{s_0}(\T,\C))$ for some $ s_0 >0  $, 
  which are linear in the variable $ V $ and such that the following holds true. 
  For any $s\geq s_0$ there are $C>0$ and 
  $r(s)\in]0,r[$ such that for any 
  $U\in B^K_{s_0}(I;r(s))\cap C^K_{*}(I,\dot{H}^{s}(\T,\C^2))$, 
  any $ V \in C^{K-K'}_{*}(I,\dot{H}^{s}(\T,\C))$, any $0\leq k\leq K-K'$, $t\in I$, we have that
\begin{equation}
\label{piove}
\|{\partial_t^k\left(R(U;t)V\right)}\|_{\dot{H}^{s-  k + \vr}}
 \leq C \!\!\!\! \sum_{k'+k''=k} \!\!\!\! \|{V}\|_{k'',s}\|{U}\|_{k'+K',s_0}^{p} 
 +\|{V}\|_{k'',s_0}\|U\|_{k'+K',s_0}^{p-1}\|{U}\|_{k'+K',s} \, .
\end{equation}
In case $ p = 0$ we require  the estimate
$ \|{\partial_t^k\left(R(U;t)V\right)}\|_{\dot{H}^{s-  k + \vr}}
 \leq C  \|{V}\|_{k,s}$.
 
(iii) {\bf Smoothing operators.}
We denote by $\Sigma\mathcal{R}^{m}_{K,K',p}[r,\tN]$, 
the space of operators $(U,t,V)\to R(U;t)V$ such that there are homogeneous  smoothing operators  $R_{q} $ in $ \widetilde{\mathcal{R}}^{m}_{q}$, $q=p,\ldots, \tN$  and a non--homogeneous smoothing  operator $R_{>\tN}$  in  
$\mathcal{R}^{m}_{K,K',\tN+1}[r]$ such that 
$$
R(U;t)V= {\mathop \sum}_{q=p}^{\tN}R_{q}(U,\ldots,U)V+R_{>\tN}(U;t)V \, .
$$
 We denote by  $\Sigma\mathcal{R}^{-\vr}_{K,K',p}[r,\tN]\otimes\mathcal{M}_2(\C)$
the space of $2\times 2$ matrices whose entries are  in $\Sigma\mathcal{R}^{-\vr}_{K,K',p}[r,\! \tN]$.
\end{defn}

If $R( U_1,\dots , U_p)$ is a $p$--homogeneous 
 $\varrho $-smoothing operator in $  \widetilde \cR_p ^{-\vr}$ then 
  for any $K\in \N_0$  there exists $s_0 >0$ such that for any $s \geq s_0$,
  for any   $U \in C^K_{*}(I,\dot{H}^{s}(\T,\C^2))$,  any $ v \in C^K_{*}(I,\dot{H}^{s}(\T,\C)) $, one has (cf. e.g. \cite{BMM2}[Lemma 2.8])
 \be
 \label{mappabonetta1}
  \| R(U_1,\dots, U_p) v\|_{\dot H^{s+\varrho}}
  \lesssim_{K} 
  \norm{v}_{\dot H^s } \prod_{a =1}^p \norm{U_a}_{\dot H^{ s_0}}  +   \norm{v}_{\dot H^{s_0} }\sum_{\bar a =1}^p \norm{U_{\bar a}}_{\dot H^s}
  \prod_{a =1\atop a\neq \bar a}^p \norm{U_a}_{\dot H^{s_0}} 
  \, . 
 \ee
To simplify notation we will omit the dependence on the time $t$ from the symbols and smoothing remainders. Furthermore  
we fix  $\tN = 3 $ and 
we write $ \Sigma\mathcal{R}^{-\vr}_{K,K',p}$ instead of  
$ \Sigma\mathcal{R}^{-\vr}_{K,K',p}[r,3]$,  and similarly for 
functions and smoothing operators.

\subsection{Paralinearization of the equations and strategy}

We now start proving \Cref{BNFtheorem}. 
The starting point is to write the water waves system 
\eqref{ww}
in paradifferential form. Precisely \cite[Proposition 3.3]{BFP} guarantees that for any $K \in \N$ and $\vr \gg 1$, there exists $s_0>0$ such that for any $s \geq s_0$, for any $0< r \leq r_0(s)$ small enough, if 
 the variable 
 $U = \begin{pmatrix}
     u \\ \bar u
 \end{pmatrix}$ with $u$ defined in \eqref{u0} belongs to $B^K_{s, \R}(I;r)$ (with $ s \rightsquigarrow N $ and recall notation \eqref{ball}) and  solves 
\be\label{ww.paraco}
\pa_t U = \Opbw{\im A_1 (U;x) \xi  + \im A_{\frac12}(U;x) |\xi|^{\frac12}  + A_0(U;x) + A_{-1}(U;x,\xi)}U + R(U)U 
\ee
where
\begin{align}\label{A1}
& A_1(U;x):= \begin{pmatrix} -V(U;x) & 0 \\  0 &  - V(U;x) \end{pmatrix} \ , \\
\label{A12}
& A_{\frac12}(U;x):=  \begin{pmatrix} -(1+a(U;x)) & -a(U;x) \\  a(U;x) &  1+a(U;x) \end{pmatrix} \ , \quad a := \frac12 \left( \pa_t B  + V B_x \right) \in  \Sigma \cF^\R_{K,1,1} \ , \\
 \label{A0}
& A_0(U;x):= -\frac14 \begin{pmatrix} 0& 1 \\ 1 & 0 \end{pmatrix}  V_x(U;x) \ , 
\end{align}
the functions $V(U;x), B(U;x)  $ belong to 
$ \Sigma \cF^\R_{K,0,1}$ 
(they are  the functions \eqref{def:V}-\eqref{form-of-B} expressed in terms of the complex variable $ U $), 
$A_{-1}(U; \cdot)$ is a real-to-real matrix of symbols in $\Sigma \Gamma^{-1}_{K,1,1} \otimes \cM_2(\C)$ and $R(U)$
 is a real-to-real matrix of smoothing operators
belonging to $\Sigma \cR^{-\vr}_{K,1,1} \otimes \cM_2(\C)$.

\smallskip
We now explain the main changes we address 
with respect to   \cite{BFP} 
to achieve 
the Lipschitz properties \eqref{phi_est} and \eqref{lipB}  of the normal form transformation. 
In order to obtain energy estimates, following  \cite{BFP} we  
block-diagonalize the matrix of symbols in \eqref{ww.paraco} at positive orders at {\em any} degree of homogeneity. 
 Since the matrix $A_{1}(U;x) \xi$ in \eqref{A1} is already diagonal, one needs to block-diagonalize only the 
 matrix  $A_{\frac12}(U;x) |\xi|^{\frac12}$ in \eqref{A12}.
 This is achieved, as in \cite{BFP}, conjugating \eqref{ww.paraco} with the flow in \eqref{generatore20} generated by a paradifferential operator with a inhomogeneous symbol 
 $m_{-1}(U;x)$ in \eqref{def:lambda}. The first difference is that
\\[1mm]
$ \bullet\ (i)$ 
 in \Cref{lem:step12} we additionally 
 prove that the map 
 $U\mapsto m_{-1}(U;x)$ is  Lipschitz. 
 We exploit that $ m_{-1}(U;x)$ is a function of the symbol $a(U;x)$ defined in \eqref{A12}, which we write in terms of $(\eta, \psi)$ by computing 
 \be\label{eq:patB}
 \pa_t B = \di_{\eta,  \psi} 
 B(\eta, \psi)[X(\eta, \psi)] = 
 \di_{\eta} 
 B(\eta, \psi)[X^{(\eta)}(\eta, \psi)] + 
 \di_{\psi} 
 B(\eta, \psi)[X^{(\psi)}(\eta, \psi)]  
 \ee
 where $ X(\eta, \psi)$ is the 
 water waves vector field in \eqref{ww}, which is analytic by 
 \eqref{wwana}, as well as $ B $ is \eqref{an:DNX}. In this way we check the Lipschitzianity of 
 $ a(U;\cdot )$ in $ U $ and thus of $ m_{-1}(U;\cdot) $. 
 \\[1mm]
$ \bullet\ (ii)$ as in \cite{BFP} we also need to remove the off-diagonal symbols up to very large negative orders (in order to prove energy estimates). 
In \cite{BFP} these symbols are removed at {\em any} order of homogeneity. In this way one loses control of the Lipschitz property since the symbols in $A_{-1}(U;x)$ are not explicit at any degree (and it is not proved they are Lipschitz).  
We overcome this issue by removing those symbols only up to cubic  degree of homogeneity. As a result, 
the normal form transformations we perform are 
the flow in  \eqref{flow-jth} generated 
by the matrix of symbols $ M^{(j)} (U)$ in \eqref{generatore-jth} which are pluri-homogeneous in $ U $, 
thus satisfy a Lipschitz property in $U $, as we show in the 
next section.

\smallskip
\noindent
{\bf Notation:} In the sequence  $s_0>0$ denotes  the index of a low Sobolev norm, that may change from line to line.

\subsection{Lipschitz properties of linear flows}

We consider linear flows of the form
 \be\label{flussoG}
\begin{cases}
\pa_\tau \Phi^\tau(U)= G(\tau,U)\Phi^\tau(U)\\
\Phi^0(U)=\uno
\end{cases} 
\ee
where $G(\tau,U)$ is a matrix of  real-to-real paradifferential or smoothing operators. 
Similarly to  \cite{MM}, we prove that 
the flow map   $U\mapsto \Phi^\tau(U)$ 
is Lipschitz provided 
$U \mapsto G(\tau, U)$ is  (in suitable topologies).
We first 
consider the case  $G(\tau, U)$ is a  paradifferential operator.
 
\begin{lem}\label{lem:flow.ad}
Let $p \in \N$, $\tN $ with $ p \leq \tN $, $ \tau \in [0,1]$. 
Let  $G(\tau, U)$ be either
\begin{itemize}
\item[(i)] $ \Opbw{\sm{ \frac{\beta(U;x)}{1+\tau \beta_x(U;x)}\im \xi}{0}{0}{\frac{\beta(U;x)}{1+\tau \beta_x(U;x)}\im \xi}}$ with $\beta(U;x)$ a $p$-homogeneous real value function in $\wt \cF_p^\R$;
\item[(ii)] $ \Opbw{\sm{ \im f(U;x, \xi)}{0}{0}{-\im \bar{f(U;x, -\xi)}}}$ with $f(U; x, \xi) $ a $p$-homogeneous  symbol in  $\wt\Gamma_p^m, m < 1$, which additionally is  real valued in case $m \in (0,1)$;
\item[(iii)] $\Opbw{\sm{0}{ g(U;x,\xi)}{\bar{g(U;x, -\xi)}}{0}}$ with $g(U;x, \xi)$ a  pluri-homogeneous symbol in $\sum_p^\tN \wt \Gamma^m_q $, $m \leq  0$, or 
a function $g(U;x)$ satisfying $g(0;\cdot) =0$ and the Lipschitz estimate
\be\label{0509:1056}
 \norm{   g(U_1;x) - g(U_2;x)  }_{L^\infty} \lesssim   
 \norm{U_1-U_2}_{\dot H^{s_0}} \, . 
\ee
\end{itemize}
Then, there exists $ s_0, r > 0 $ such that, for any $U \in \dot H_\R^{s_0}(\T, \C^2)$ satisfying 
 $ \norm{U}_{\dot H^{s_0}} \leq r $, the flow  $\Phi^\tau(U)$ in \eqref{flussoG} is well defined, $\Phi^\tau(0) = \uno$,  and satisfies
 for  any  $s \in \R$ 
 and $V \in \dot H^s (\T, \C^2)$  the bound 
\begin{equation}\label{flusso.est}
\sup_{\tau \in [0,1]}
\norm{\Phi^\tau(U) V}_{\dot H^s} \leq  (1 + C_s \norm{U}_{\dot H^{s_0}}) \norm{V}_{\dot H^s} \, , 
\end{equation}
 as well as its  inverse $\Phi^\tau(U)^{-1}$. In addition  \\
{\sc Lipschitz property:}
There exist $s_0, r > 0$  such that
for any  $U_j\in \dot H_\R^{s_0}(\T, \C^2)$ with $ \norm{U_j}_{\dot H^{s_0}} \leq r$, $j=1,2$, 
 any $s \in \R$ and 
any  $ Z \in \dot H^s (\T, \C^2)$,  
\be\label{est.diff.Phi}
\sup_{\tau \in [0,1]}\norm{\Phi^\tau(U_1)Z - \Phi^\tau(U_2)Z }_{\dot H^{s-m}} \leq C_{s,r}  \norm{U_1 -U_2}_{\dot H^{s_0}} \norm{Z}_{\dot H^s} 
\ee
for some $C_{s,r}>0$, and with 
   $m=1$ in case of item $(i)$.
\end{lem} 

 \begin{proof}
 Under the assumption $(i)$ or $(ii)$ or $(iii)$, the flow $\Phi^\tau(U)$ is well defined for any $\tau \in [0,1]$ and
 satisfies \eqref{flusso.est} by \cite[Lemma 3.16 $(i)$]{BMM2}.  
Let us prove the Lipschitz estimate \eqref{est.diff.Phi}.
For any     $U_j\in \dot H_\R^{s_0}(\T, \C^2)$, $j=1,2$, with $ \norm{U_j}_{\dot H^{s_0}} \leq r$, the difference 
  $ \Phi^\vartheta(U_1)- \Phi^\vartheta(U_2)$ 
fulfills the variational equation
$$
\begin{cases}
\partial_\vartheta \Big(\Phi^\vartheta(U_1)- \Phi^\vartheta(U_2) \Big) =G(\vartheta,U_2) \Big(\Phi^\vartheta(U_1)- \Phi^\vartheta(U_2) \Big)
+ \Big(G(\vartheta, U_1) - G(\vartheta, U_2) \Big) \Phi^\vartheta(U_1) \\
\Phi^0(U_1)- \Phi^0(U_2)=0 \, , 
\end{cases}  
$$
whose solution is given by the  Duhamel formula 
\begin{align}\label{duhamel.diff}
\Phi^\vartheta(U_1)- \Phi^\vartheta(U_2) =  \Phi^\vartheta(U_2) \int_0^\vartheta  \Phi^\tau (U_2)^{-1} \ \Big(  G(\tau,U_1) - G(\tau, U_2) \Big)\   \Phi^\tau (U_1)\, {\rm d}\tau  \ . 
\end{align}
We claim that, for  $G(\tau, U)$ as in item $(i), (ii)$ or $(iii)$, for any $s \in \R$,  any  $Z \in \dot H^s(\T, \C^2)$, 
\begin{equation} 
\sup_{\tau \in [0,1]}\norm{ \Big(G(\tau,U_1) - G(\tau, U_2)\Big)Z}_{\dot H^{s-m}} \leq C_{s,r} 
\norm{U_1 - U_2}_{\dot H^{s_0}}\, \norm{Z}_{\dot H^s}, \label{claimata} 
\end{equation}
where $m = 1$ in  case $(i)$. 
Then the  Lipschitz estimate \eqref{est.diff.Phi} follows 
by  \eqref{duhamel.diff}, \eqref{claimata} and  \eqref{flusso.est}.

It remains to prove \eqref{claimata} which 
is a consequence of the estimate \eqref{cont00} for the action of 
paradifferential operators.
In case $(i)$, we have  
\be\label{3008:1215}
\Big(G(\tau,U_1) - G(\tau, U_2)\Big)Z = \Opbw{\sm{ a(\tau; x, \xi)}{0}{0}{a(\tau; x, \xi)}}Z \ 
\ee
where
\begin{align*}
    a(\tau; x, \xi) & :=  \left(\frac{\beta(U_1;x)}{1+\tau \beta_x(U_1;x)}-\frac{\beta(U_2;x)}{1+\tau \beta_x(U_2;x)}  \right)\im \xi 
\\
& =
\frac{\beta(U_1;x)-\beta(U_2;x)}{1+\tau \beta_x(U_1;x)}\im \xi
-\tau \frac{\beta(U_2;x) \left(\beta_x(U_1; x) - \beta_x(U_2;x) \right)}{(1+\tau \beta_x(U_1;x)) \,(1+\tau \beta_x(U_2;x))  }\im \xi \, . 
\end{align*}
Since  $\beta(U;x)\in \wt \cF^\R_{p} $ is a homogeneous polynomial in $U$, one has 
$$
\beta(U_1;x) - \beta(U_2;x) = 
\sum_{1 \leq i_1 \leq \cdots \leq i_{p-1} \leq 2} 
\beta(U_1-U_2, U_{i_1}, \cdots, U_{i_{p-1}};x)
$$
implying, by the very definition \eqref{homosymbo} (for a 
function so $m \rightsquigarrow 0$, $  \beta \rightsquigarrow 0 $), the estimate
$$
\norm{\beta(U_1;\cdot) - \beta(U_2;\cdot)}_{L^\infty} 
+\norm{\beta_x (U_1;\cdot) - \beta_x (U_2;\cdot)}_{L^\infty}  \leq C_{r}  \norm{U_1 - U_2}_{\dot H^{s_0}}  
$$
for some $C_r >0$.
Then the seminorm \eqref{seminorm} of the symbol $a(\tau; x, \xi)$ is bounded by 
$$
\sup_{\tau \in [0,1]}
|a(\tau; \cdot)|_{1,L^\infty, 4} =  
\max_{|\beta| \leq 4} \sup_{\xi \in \R} \big\| \la \xi \ra^{-1+|\beta|} \, \pa_\xi^\beta a(\tau; \cdot, \xi) \big\|_{L^\infty}  
\leq C_r \norm{U_1 - U_2}_{\dot H^{s_0}} 
$$
and  \eqref{cont00} implies that the operator \eqref{3008:1215} satisfies \eqref{claimata}.

Similar arguments prove   estimate \eqref{claimata} also in cases $(ii)$, $(iii)$.
 \end{proof}
 
Now we consider the case $G(\tau,U) $ is a matrix of smoothing operators.

\begin{lem}\label{flow.s.ad}
Let $\varrho \geq 0$, $p \in \N$ and 
$G(\tau, U):= 
 R(U) $ be a matrix of smothing operators in $  \wt \cR^{-\varrho}_p \otimes \cM_2(\C)$.
Then there  is $s_0\geq 0$ such that for any $s \geq s_0$, there is $r = r(s), C_{s,r}>0$ so that for 
any  $ \norm{U}_{\dot H^{s}} \leq r$, the flow $\Phi^\tau(U)$ in \eqref{flussoG} 
is a well defined bounded operator in 
$\dot H^s(\T, \C^2)$, satisfies 
$\Phi^\tau(0) = \uno$, and 
\begin{equation}\label{flusso.estR}
 \norm{[\Phi^\tau(U)]^{\pm 1} V}_{\dot H^s} \leq  C_s \norm{V}_{\dot H^s} ( 1+ C_s\norm{U}_{\dot H^{s_0}})  + C_s \norm{V}_{\dot H^{s_0}} \norm{U}_{\dot H^{s}}
\end{equation}
uniformly in $\tau \in [0,1]$. In addition one has\\
{\sc Lipschitz property:} for 
any  $ \norm{U_i}_{\dot H^{s}} \leq r$,  $i=1,2$,  any  $ Z \in \dot H^s (\T, \C^2)$,  one has
\be\label{est.diff.Phi2}
\norm{\Phi^\tau(U_1)Z-\Phi^\tau(U_2)Z }_{\dot{H}^{s}} \leq C_{s,r}  \norm{U_1-U_2}_{\dot{H}^{s}} \norm{Z}_{\dot{H}^{s}} \ .
\ee
\end{lem} 
 \begin{proof}
The well-definition of the flow $\Phi^\tau(U)$ and \eqref{flusso.estR} follow from 
\cite[Lemma A.3]{BFP} with $K = 0$. 
 The proof of \eqref{est.diff.Phi2} follows 
similarly  to the previous lemma. Formula \eqref{duhamel.diff} holds with $ G(\tau,U_1)-G(\tau, U_2) $ replaced by $R(U_1)-R(U_2)$.
Then replace  
 estimate \eqref{claimata} by 
 \begin{equation*}
 \norm{R(U_1)Z-R(U_2)  Z}_{\dot{H}^{s}}\leq
 \!\!\!\!\!\!
 \sum_{1 \leq i_1 \leq \cdots \leq i_{p-1} \leq 2} 
 \norm{
R(U_1-U_2, U_{i_1}, \cdots, U_{i_{p-1}})Z}_{\dot{H}^{s}}  \leq 
 C_{s,r} \norm{U_1 - U_2}_{\dot{H}^{s}} \norm{Z}_{\dot{H}^{s}}
 \end{equation*}
 obtained applying  \eqref{mappabonetta1}.
 \end{proof}

\subsection{Normal form with a Lipschitz map}
We now start the normal form procedure.

\smallskip
\noindent{\bf Block-diagonalization up to quartic degree bounded terms.}
 The first step is to block-diagonalize system \eqref{ww.paraco}  in the variables $(u, \bar u)$ up to a quartic degree term.
The following result is a variant of \cite[Proposition 3.10]{BFP} with a new  bounded remainder $B(U)W$  in \eqref{sistemaDiag} which satisfies the quartic bound \eqref{estB(U)}.  
  \begin{prop}\label{teodiagonal}
Let  $ \vr \gg 1   $  and  $ K \geq  K':=2\vr+2  $.  There exists $s_0>0$ such that, 
for any $s\geq s_0$, for all $0<r \leq r_0(s)$ small enough,  
and  any solution $U\in B^K_{s, \R}(I;r)$ of  \eqref{ww.paraco},  the following holds:

\begin{enumerate}
\item[(i)] 
there is a map 
$ \Psi_{diag}^{\theta}(U) $, $ \theta\in [0,1] $,  satisfying, for some $C=C(s,r,K)>0$, for  any 
$V\in \dot{H}^s(\T,\C^{2})$
$$
 \| \Psi_{diag}^{\theta}(U)V\|_{\dot{H}^{s}}
+\|(\Psi_{diag}^{\theta}(U))^{-1}V\|_{\dot{H}^{s}}
\leq \big(1+C \|{U}\|_{K,s_0}\big)\|V\|_{\dot H^s} \,;
$$
\item[(ii)] the function $W:=(\Psi_{diag}^{\theta}(U)U)_{|_{\theta=1}}$ solves the real-to-real system 
\begin{equation}\label{sistemaDiag}
 \pa_{t}W =
 \opbw\left(
 \sm{d(U;x,\xi) + r_{-1/2}(U;x,\x)}{0}{0}{\overline{d(U;x,-\xi)}
 +\bar{r_{-1/2}(U;x,-\x)}}\right)
W +R(U)W + B(U)W
\end{equation}
where $d(U;x,\x)$  is a  symbol of the form
\begin{equation}\label{sistemainiziale3}
d(U;x,\xi):=- \im V(U;x) \xi  -  \im (1+ a^{(0)}(U;x) ) |\xi|^{1/2} 
\end{equation}
where
 $  a^{(0)}(U;x)  $ is a function in $  \Sigma {\cF}^{\R}_{K,1,1}  $, 
$ r_{-1/2}(U;x,\x) $ is a symbol in $  \Sigma {\Gamma}^{-1/2}_{K,2\vr+2 ,1} $, 
$R(U) $ is a real-to-real matrix of smoothing operators in $ \Sigma\mathcal{R}^{-\vr}_{K,2\vr+2 ,1}\otimes\mathcal{M}_2(\mathbb{C})$, 
and $B(U)$ satisfies the estimate for any $s \in \R$
\be\label{estB(U)}
\norm{B(U)V}_{\dot H^s} \leq C_s 
\norm{V}_{\dot H^s} \norm{U}_{2\vr+1 , s_0}^3  \, . 
\ee
\end{enumerate}
{\sc Lipschitz property:} The operator $\Psi_{diag}^\theta(U) $ satisfies a Lipschitz property as 
\eqref{est.diff.Phi} and  $\Psi_{diag}^{\theta}(0) = \uno$.
  \end{prop}

\noindent
Proposition \ref{teodiagonal} is proved by 
applying a sequence a  transformations 
which iteratively block-diagonalize \eqref{ww.paraco} in decreasing orders up to the quartic-degree bounded term  $ B(U)W  $.

{\sc Step 1: Block diagonalization at order $\frac12$.} 
We diagonalize the matrix
of symbols $A_{\frac12}(U; x)|\xi|^{\frac12}$
in \eqref{ww.paraco}, up to a matrix of symbols of order 0. 
The following lemma is \cite[Lemma 3.11]{BFP}, and we prove  in addition  that 
the normal form map satisfies a Lipschitz property  as 
\eqref{est.diff.Phi}.
\begin{lem}\label{lem:step12}
 Define the function 
\begin{align}\label{def:lambda}
m_{-1}(U;x) := - \log \Big(\frac{1+\lambda}{\sqrt{(1+ a+ \lambda)^2 - a^2}}\Big) \in \Sigma\mathcal{F}_{K,1,1} \, , \qquad \lambda(U;x):= \sqrt{1+2 a(U;x)}
\, , 
\end{align}
where
 $a(U;x)$ is the function in \eqref{A12}, and 
the flow 
\begin{equation}\label{generatore20}
\pa_{\theta}\Psi_{-1}^{\theta}(U) =\opbw(M_{-1})\Psi_{-1}^{\theta}(U), \quad 
\Psi_{-1}^{0}(U) = {\rm Id} \, , \quad  
M_{-1} :=
\begin{pmatrix} 0 & m_{-1}(U;x) \\ \ov{m_{-1}(U;x)} & 0 \end{pmatrix}
\, .
\end{equation}
If $U$ solves \eqref{ww.paraco},  then the function 
\begin{equation}\label{TRA5}
W_0:=(\Psi_{-1}^{\theta}(U))_{|_{\theta=1}}U 
\end{equation}
 solves the real-to-real system
\begin{equation}\label{NuovoParaprod}
\begin{aligned}
\pa_t W_0=\opbw\Big( 
\sm{ d(U;x,\x)}{ 0}{0}{\ov{d(U;x,-\x)} }
+A^{(0)} 
\Big)W_0+R^{(0)}(U)W_0
\end{aligned}
\end{equation}
where $ d(U;x,\xi)  $ is the symbol in  \eqref{sistemainiziale3} with  
$a^{(0)}(U;x):= \lambda(U;x) -1 \in \Sigma \cF^\R_{K,1, 1} $, 
a matrix of symbols 
\begin{equation}\label{lambdazero}
\begin{aligned}
A^{(0)} :=\left(\begin{matrix}  c_0(U;x, \xi)  & b_0(U;x, \xi) \\
 \ov{b_0(U;x, - \xi)} & \ov{c_0(U;x, - \xi)}   \end{matrix}\right)  \, , \ 
c_0 \in \Sigma\Gamma^{-\frac{1}{2}}_{K,2,1} \, ,\;\;  b_0 \in \Sigma\Gamma^{0}_{K,2,1}  \, , 
\end{aligned}
\end{equation}
and a real-to-real matrix of smoothing operators $R^{(0)}(U)$ in  
$ \Sigma\mathcal{R}^{-\vr}_{K,2,1}\otimes\mathcal{M}_2(\mathbb{C})$.

 {\sc Lipschitz property:} The operator $\Psi_{-1}^\theta(U) $ satisfies a Lipschitz property as 
\eqref{est.diff.Phi} and $\Psi_{-1}^{\theta}(0) = \uno$.
\end{lem}

\begin{proof}
We prove only the Lipschitz property of the flow map $\Psi_{-1}^\theta(U) $, 
the other statements  are proved in  \cite[Lemma 3.11]{BFP}. It follows by Lemma \ref{lem:flow.ad} $(iii)$ once we prove that 
the function $m_{-1}(U;x) $ satisfies the 
Lipschitz estimate  \eqref{0509:1056}.
We now prove that, provided 
$\ts > 5 $ and $ r> 0 $ is small enough,  the map
\be\label{defm-1U}
B^{\ts}_\R(r) := \{ U \in H_\R^\ts  (\T,\C^2) , 
\| U \|_{{\dot H}^\ts} < r \} 
\to  H^{\ts-\frac94}(\T,\R) \, , 
\quad U\mapsto m_{-1}(U;\cdot) \ \   \text{is analytic} \, ,
\ee
and consequently  $m_{-1}(U_1;\cdot)- m_{-1}(U_2;\cdot)$ satisfies the Lipschitz estimate \eqref{0509:1056}. 
The explicit form of the function $m_{-1}(U;x) $ in 
\eqref{def:lambda} is obtained by a direct 
calculus. 
Since $m_{-1}(U;x)$ depends on $U$ only through the function $a(U;x)$, and the map  $ a \mapsto m_{-1} $ is  analytic provided $\norm{a}_{H^{s_0}}$ is small enough, by composition,  
it is sufficient to 
prove, for any $\ts > 5 $, 
the analyticity of the map 
\be\label{aUdot}
B^{\ts}_\R(r)\to  H^{\ts-\frac94}(\T, \R) \, , \quad 
U \mapsto a(U;\cdot) \stackrel{\eqref{A12}} 
= \tfrac12 \left( \pa_t B  + V B_x \right)
\stackrel{\eqref{eq:patB}} = \tfrac12  \big( \di_{\eta,  \psi} 
 B(\eta, \psi)[X(\eta, \psi)] + V B_x
 \big) 
\ee
where $ X(\eta, \psi) $  is the water waves vector field, 
which is analytic as stated in \eqref{wwana}, and $ V, B $ are analytic as well in $ (\eta, \psi) $  as stated in \eqref{an:DNX}. 
Thus the map 
 $(\eta, \psi) \mapsto a(\eta,\psi) $ 
 is analytic from 
$B^{\ts}(r)\times \dot H^{\ts}(\T) \to H^{\ts-2}(\T)$ for any $ \ts > \frac92 $.
To conclude the proof it is sufficient to 
express 
$(\eta, \psi)$ as  functions of $U$,
and prove that this mapping is analytic. First note 
that the good unknown map $(\eta, \psi) \mapsto \cG(\eta, \psi)$  in \eqref{an:omega} 
is, for any $ \ts > \tfrac72 $  locally invertible around zero with an analytic inverse  
$(\eta, \upomega) \mapsto (\eta, \psi) = \cG^{-1}(\eta, \upomega)$ because  $\di \cG(0,0) = \uno$.   
Hence, cf. \eqref{u0}, \eqref{omega0},
$$
 H^\s_\R(\T,\C^2) \to H^{\ts-\frac14}(\T,\R)\times H^{\ts - \frac14}(\T,\R) \ , \quad 
U= \vect{u}{\bar u} \mapsto (\eta, \psi) = \cG^{-1}\Big( \frac{|D|^{\frac14} [u + \bar u]}{\sqrt{2}} \, , \ \frac{ |D|^{-\frac14}[u - \bar u]}{\im \sqrt{2}} \Big)
$$ 
is analytic 
 being  composition of analytic maps.
As a result, by  composition with the analytic 
map in \eqref{aUdot} 
we deduce 
\eqref{defm-1U}. 
\end{proof}

{\sc Step 2: Block diagonalization at negative orders.} 
We  now  block-diagonalize the matrix of symbols $A^{(0)}$ in \eqref{lambdazero} using  the flow of paradifferential operators with   pluri-homogeneous symbols. 
Given a symbol $a(U;  x, \xi)$ in $ \Sigma \Gamma_{K, K', p}^m[r,N]$   of the form \eqref{espsymbol} 
we denote, for $p\leq q \leq \tN$, the projections on the pluri-homogeneous symbols
$$
\cP_{\leq q}[ a(U;  x, \xi)] :=  \sum_{l=p}^{q} a_l (U; x, \xi) \, ,
 \quad \text{and} \quad 
 \cP_{\geq q+1}[ a(U;  x, \xi)] := a(U;  x, \xi) - \cP_{\leq q}[ a(U;  x, \xi)] \, . 
 $$

\begin{lem}\label{lem:indud}
For $ j = 0, \ldots, 2 \vr $, there are $s_0 >0$ and \\
$ \bullet $ 
paradifferential operators of the form  
 \begin{align}
\mathcal{Y}^{(j)}(U) :=
\opbw \Big( \sm{d(U;x,\xi)}{ 0}{ 0}{  \bar{d(U;x,-\xi)}}\Big)
+\opbw (A^{(j)} ) \label{sist2 j-th}
\end{align}
where  $d(U;x,\x)$ is the symbol  defined in \eqref{sistemainiziale3},
$A^{(j)}$ is a  matrix of symbols of the form
\begin{equation}\label{bjbjbj}
A^{(j)} =\left(
\begin{matrix}
 c_j(U;x,\xi)& b_j(U;x,\xi)\\
 \bar{b_j(U;x,-\xi)} & \bar{c_j(U;x,-\xi)}
\end{matrix}
\right), \;\; c_j\in \Sigma\Gamma^{-\frac{1}{2}}_{K,j+2,1},\;\;  b_j\in \Sigma\Gamma^{-\frac{j}{2}}_{K,j+2,1} \, , 
\end{equation}
$ \bullet  $  
a real-to-real  matrix of smoothing operators $R^{(j)} (U) $  in  
$\Sigma\mathcal{R}^{-\vr}_{K,j+2,1}\otimes\mathcal{M}_2(\mathbb{C})$,\\
$ \bullet  $  
a real-to-real  matrix of operators $B^{(j)} (U) $ satisfying
for any $s \in \R$, for any $ V \in \dot H^s (\T,\C^2) $, 
\be\label{estB(U)j}
\|B^{(j)}(U)V\|_{\dot H^s} \leq C_s \norm{U}_{j+1, s_0}^3 
\norm{V}_{\dot H^s}
\ee
such that, if $ W_j $, $ j = 0, \ldots, 2 \vr - 1 $, solves 
\be\label{sist1 j-th}
\pa_{t}W_j=\big(\mathcal{Y}^{(j)}(U)+R^{(j)}(U) + B^{(j)}(U)\big)W_j, \quad  W_j:=\begin{pmatrix}w_j \\ \bar w_j \end{pmatrix} \, , 
\ee
then 
\begin{equation}\label{nuovavarjth}
W_{j+1}:=(\Psi_{j}^{\theta}(U)W_{j})_{|_{\theta=1}} \, , 
\end{equation}
where $\Psi_{j}^{\theta}(U)$ is the flow at time $\theta\in [0,1]$ of 
\be\label{flow-jth}
\partial_{\theta} \Psi^{\theta}_{j} (U) = \im  \opbw{( M^{(j)}(U; x,\xi) )} \Psi_{j}^{\theta}(U) \, ,
 \quad \Psi_{j}^{0}(U) = {\rm Id} \, ,  
\ee
with
\begin{equation}\label{generatore-jth}
\!\! \small
M^{(j)} (U;x,\xi):=
\begin{pmatrix} 0 & - \im  m_{j}(U;x,\x) \\ - \im \,  \bar{m_{j}(U;x,-\x)} & 0 \end{pmatrix}\,,
\,\;\;
m_j := \cP_{\leq 2} \Big( \frac{-\chi(\xi)b_{j}(U;x,\x)}{2\im(1+a^{(0)}(U;x))|\xi|^{\frac{1}{2}}} \Big) \in 
\sum_{1}^2\wt \Gamma^{-\frac{j+1}{2}}_{p} \, ,
\end{equation}
and  $\chi$ is defined in \eqref{cutoff11},
satisfies a system of the form \eqref{sist1 j-th} with $ j + 1 $ instead of $ j $.\\
{\sc Lipschitz property:} 
The operator $\Psi_{j}^\theta(U) $ satisfies a Lipschitz property as 
\eqref{est.diff.Phi} and $\Psi_{j}^{\theta}(0) = \uno$.
\end{lem}

\begin{proof}
The  Lipschitz property of  $\Psi_{j}^\theta(U) $  follows by Lemma \ref{lem:flow.ad} $(iii)$ since the symbol $ M^{(j)} (U) $ in \eqref{generatore-jth} is pluri-homogeneous.
The rest of the proof proceeds by induction. 
\\[1mm]
{\bf Inizialization.} System \eqref{NuovoParaprod} is \eqref{sist1 j-th} for $ j = 0 $ where  the 
paradifferential operator $ \mathcal{Y}^{(0)}(U)  $ has the form   \eqref{sist2 j-th} with  the matrix of symbols 
$ A^{(0)}$ defined in  \eqref{lambdazero} and $B^{(0)} = 0$. 
\\[1mm]
{\bf Iteration.} 
We  argue by induction. Suppose that $ W_j $ solves system \eqref{sist1 j-th} with operators
$ \mathcal{Y}^{(j)}(U)  $ of  the form   \eqref{sist2 j-th}-\eqref{bjbjbj},  smoothing operators 
$R^{(j)} (U) $  in   $\Sigma\mathcal{R}^{-\vr}_{K,j+2,1}\otimes\mathcal{M}_2(\mathbb{C})$ and
a matrix of real-to-real operators $B^{(j)}(U)$ satisfying 
\eqref{estB(U)j}.
By 
formula \cite[(A.2)]{BFP} the conjugated system has the form 
\be\label{sis:j+1}
\pa_{t}W_{j+1} =   \big( (\pa_t  \Psi_{j}^{1}(U)) \Psi_{j}^{-1}(U) +
 \Psi_{j}^{1}(U)  \mathcal{Y}^{(j)} (U)   \Psi_{j}^{-1}(U) 
 + \Psi_{j}^{1}(U)  B^{(j)} (U)   \Psi_{j}^{-1}(U) 
  \big) W_{j+1}
\ee
up to a smoothing operator in  $\Sigma\mathcal{R}^{-\vr}_{K,j+2,1}\otimes\mathcal{M}_2(\mathbb{C})$. 

By the proof of \cite[Lemma 3.12]{BFP}, 
$(\pa_t  \Psi_{j}^{1}(U)) \Psi_{j}^{-1}(U)$
is a 
 paradifferential operator with symbol in $ \Sigma\Gamma^{-\frac{j+1}{2}}_{K,j+3,1}\otimes\mathcal{M}_2(\C)$ plus a 
 smoothing operator in $\Sigma\mathcal{R}^{-\vr}_{K,j+3,1}\otimes\mathcal{M}_2(\mathbb{C})$, and  
 \begin{align}  \label{eq:428}
  \Psi_{j}^{1}(U)  & \mathcal{Y}^{(j)}(U)  \Psi_{j}^{-1}(U) 
  = \opbw\Big( \sm{d(U;x,\x)}{ 0}{0}{\bar{d(U;x,-\x)}}
  +\sm{c_j(U;x,\x)}{ q_j(U; x , \xi)}{ \bar{q_j(U; x, -\xi)}}{\bar{c_j(U;x,-\x)}}
  \Big) 
  \ , \\
& q_{j}(U;x,\x):= \chi(\x) b_{j}(U;x,\x) +2\im m_{j}(U;x,\xi)(1+a^{(0)}(U;x))|\xi|^{\frac{1}{2}} \, ,  
\label{equa-jthquatuor}
\end{align} 
plus a paradifferential operator with symbol in $ \Sigma\Gamma^{-\frac{j+1}{2}}_{K,j+2,1} \otimes\mathcal{M}_2(\C) $ 
and a 
 smoothing operator belonging to 
 $\Sigma\mathcal{R}^{-\vr}_{K,j+2,1}\otimes\mathcal{M}_2(\mathbb{C})$. 
 
Consider now the off-diagonal symbol 
$q_{j}(U;x,\x)$ in \eqref{eq:428}.
By \eqref{equa-jthquatuor} and 
  the choice of $m_{j}(U;x,\x)$ in \eqref{generatore-jth} we have that
  \begin{align} \label{tildebj}
  q_j(U;x, \xi)
& = (1+a^{(0)}(U;x)) \, 
\cP_{\geq 3} \left( \frac{\chi(\xi)b_{j}(U;x,\x)}{1+a^{(0)}(U;x)} \right) 
\in \Gamma^0_{K, j+2, 3}[r]
  \, .
  \end{align}
We conclude that \eqref{nuovavarjth} solves a system of the form 
\eqref{sist2 j-th}-\eqref{sist1 j-th}
 with $j\rightsquigarrow j+1$ and
 \begin{align*}
 B^{(j+1)}(U) & := \opbw \Big({ \sm{ 0}{ q_{j}(U;x,\x) }{ \bar{q_{j}(U;x,-\x)}}{ 0}} \Big)+ \Psi_{j}^{1}(U)  B^{(j)} (U)   \Psi_{j}^{-1}(U)    
 \end{align*}
 which satisfies, by  \eqref{tildebj},  \cite[Proposition 3.8]{BD}, $ \Psi_j^\theta (U) $ satisfies 
 \eqref{flusso.est},  \eqref{estB(U)j}, the bound 
$ \|B^{(j+1)}(U)V\|_{\dot H^s} \lesssim_s 
 \norm{U}_{j+2, s_0}^3 \norm{V}_{\dot H^s} $,  
namely \eqref{estB(U)j} with $j \leadsto j+1$.
\end{proof}

\begin{proof}[Proof of Proposition \ref{teodiagonal}] The proof follows combining Lemmata \ref{lem:step12} and  \ref{lem:indud}, defining
$$
\Psi_{diag}^\theta(U):= \Psi_{2\vr-1}^\theta(U) \circ \cdots \Psi_{0}^\theta(U) \circ \Psi_{-1}^\theta(U) 
$$ 
where the maps $\Psi_{j}^\theta(U)$, $j=0, \ldots, 2\vr -1$, resp. $\Psi_{-1}^\theta(U) $, are defined in \eqref{nuovavarjth}, resp.  \eqref{TRA5}. 
Then $\Psi_{diag}^\theta(U)$ satisfies a Lipschitz property as \eqref{est.diff.Phi} since all the maps $\Psi_j^\theta(U)$, $j = 0, \cdots, 2\vr -1$ and $\Psi_{-1}^\theta(U)$ satisfy estimates
\eqref{flusso.est} and \eqref{est.diff.Phi}. Finally $\Psi_{diag}^\theta(0) = \uno$ as well as all the maps 
$\Psi_{j}^\theta(0) $ and $ \Psi_{-1}^\theta(0) $.
\end{proof}

\smallskip
\noindent{\bf Reductions to constant coefficients and Poincar\'e-Birkhoff normal form.}
  At this point the proof is analogous to  \cite[Propositions 4.4, 5.2, and 6.2]{BFP}.
  \begin{prop}\label{cor:BNF}
 There exists $ \vr_0 > 0 $ such that, for all $ \vr \geq \vr_0 $, $ K \geq K' = 2 \vr + 2 $, 
there exists $s_0>0$ such that, 
for any $s\geq s_0$, for all $0<r \leq r_0(s)$ small enough,  
and  any solution $U\in B^K_s(I;r)$ of the water waves system \eqref{ww.paraco}, there is a
real-to-real, bounded and invertible operator 
$\mathfrak{C}^{\theta}(U)$, 
 $ \theta \in [0,1] $, 
such that 
$Z :=  \vect{z}{\bar{z}} = \mathfrak{C}^{1}(U)[U]$ 
solves
\begin{equation}\label{sistemaBNF1000}
 \pa_{t}Z= -\im{\bf \Omega} Z + X_{H^{(4)}_{ZD}} + \mathcal{X}_{\geq 4 }(U,Z) 
\end{equation}
where:

\begin{itemize}
\item ${\bf \Omega}:= \begin{pmatrix}
    |D|^{\frac12} & 0 \\ 0 & - |D|^{\frac12}
\end{pmatrix} $ and $X_{H^{(4)}_{ZD}}$ is the Hamiltonian vector field of the Hamiltonian $H_{ZD}^{(4)}$ in \eqref{theoBirH};

\item $ \mathcal{X}_{\geq 4}(U,Z) $ has the form 
 \begin{equation}\label{Stimaenergy100}
 \mathcal{X}_{\geq 4}(U,Z)= 
 \Opbw{ \sm{{{H}}_{\geq3}(U; x, \xi )}{0}{0}{\bar{{{H}}_{\geq3}(U; x, -\xi )}} }Z +\mathfrak{R}_{\geq 3}(U) Z + B_{\geq 3}(U) Z
 \end{equation}
where 
\begin{equation}
    H_{\geq 3}(U; x, \xi):= \im \alpha_{\geq 3}(U; x) \xi + \im \beta_{\geq 3}(U;x) |\xi|^{\frac12} + \gamma_{\geq 3}(U; x, \xi) \in  \Gamma^{1}_{K,K',3}\otimes\mathcal{M}_2(\C)
\end{equation}
with real valued functions $\alpha_{\geq 3}(U;x), \beta_{\geq 3}(U;x)$ in $\cF^\R_{K, K', 3}$ and a symbol 
$\gamma_{\geq 3}(U; x, \xi)$ in $\Gamma^0_{K, K', 3}$, 
$\mathfrak{R}_{\geq 3}(U)$ is a matrix of real-to-real smoothing operators in 
 $\mathcal{R}^{-(\vr - \vr_0)}_{K,K',3}\otimes\mathcal{M}_2(\C)$
 and $B_{\geq 3}(U)$ a matrix of operators satisfying
\be\label{estB3}
\norm{B_{\geq 3}(U)V}_{\dot H^s} \leq 
C_s \norm{V}_{\dot H^s} \norm{U}_{K', s_0}^3  
+ C_s \norm{V}_{\dot H^{s_0}} \norm{U}_{K',s}  \norm{U}_{K',s_0}^3  \, . 
\ee
\end{itemize}
Furthermore the map $\mathfrak{C}^{\theta}(U)$ satisfies the following properties:
\\[1mm]
(i) there is a constant $C$  depending on $s$, $r$ and $K$, such that for  any $s \geq s_0$ and any $V\in \dot{H}^{s}
(\T,\C^2)$
\begin{equation}\label{stimafinaleFINFRAK}
 \| \mathfrak{C}^{\theta}(U)[V]\|_{\dot{H}^{s}}+
\|  (\mathfrak{C}^{\theta}(U))^{-1}[V]\|_{\dot{H}^{s}} 
\leq \|V\|_{\dot H^s}(1+C\|U\|_{K,s_0}) 
+ C\|V\|_{\dot H^{s_0}}\|U\|_{K,s} \,,
\end{equation}
uniformly in $\theta\in [0,1]$;
\\[1mm]
{\sc Lipschitz property:} 
for any $s \geq s_0$, there is $ r =  r(s), C_{s,r}>0$ so that for   any   $ \norm{U_i}_{\dot H^{s}} \leq r$, $i=1,2$, 
any  $ V \in \dot H^{s+1} (\T, \C^2)$,  one has
 \begin{equation}\label{diffC}
 \norm{ \mathfrak{C}^{1}(U_1)V-\mathfrak{C}^{1}(U_2)V }_{\dot H^{s}} \leq C_s
\norm{U_1 - U_2}_{\dot H^{s}} \norm{V}_{\dot H^{s+1}} \, .
 \end{equation}
 Finally  $\mathfrak{C}^{\theta}(0) = \uno$.
\end{prop}

\begin{proof}
We only prove that $\mathfrak{C}^\theta(U)$ satisfies the Lipschitz estimate   
\eqref{diffC} and 
that the new term 
$ B_{\geq 3} $ satisfies  \eqref{estB3}.  
This requires to analyze the structure of the map $\mathfrak{C}^\theta(U)$ which is the same as \cite{BFP}. 
By {\em Proof of Proposition 5.2} \cite[pag. 1472]{BFP},
\be\label{3008:1842}
\mathfrak{C}^\theta(U) = \cB^\theta(U)\circ \mathfrak{F}^\theta(U) \ , 
\ee where by
\cite[formula (4.69)]{BFP}  and 
 \cite[Proof of Proposition 5.2, pag. 1472]{BFP}
\be\label{3008:1843}
\mathfrak{F}^\theta(U):=\Upsilon_{fin}^\theta(U)\circ \vec{\Phi}_5^\theta(U)\circ \ldots \circ \vec{\Phi}_1^\theta(U) \circ \Psi^\theta_{ diag}(U) \ , 
\qquad 
\cB^\theta(U) := \cB^\theta_2(U)\circ \cB^\theta_1(U) \, , 
\ee
where $\Psi^\theta_{diag}(U)$ is the map of Proposition \ref{teodiagonal} and \\
$\bullet$
 $\vec{ \Phi}_1^\theta(U)$ is given in \cite[(4.10), (4.13)]{BFP}, 
$\vec{ \Phi}_2^\theta(U)$ is given in \cite[(4.26), (4.27)]{BFP},
 $\vec{ \Phi}_3^\theta(U)$ is  given in \cite[(4.33), (4.34)]{BFP}, 
 $\vec{\Phi}_4^\theta(U)$ is given in \cite[(4.41), (4.42)]{BFP},
 and 
$\vec{ \Phi}_5^\theta(U)$ is given in \cite[(4.48), (4.49)]{BFP}.
In particular 
 $\vec{ \Phi}_1^\theta(U)$, 
$\vec{ \Phi}_2^\theta(U)$ 
and $\vec{\Phi}_4^\theta(U)$ are   time flows
 as \eqref{flussoG} with 
 $$
 G(\tau, U) =   \Opbw{\sm{ \frac{\beta(U;x)}{1+\tau \beta_x(U;x)}\im \xi}{0}{0}{\frac{\beta(U;x)}{1+\tau \beta_x(U;x)}\im \xi}} \,  , 
 \qquad \beta\in \wt \cF_i^\R, \ \  i  = 1, 2  \, , 
 $$
whereas  $\vec{ \Phi}_3^\theta(U)$ 
and $\vec{ \Phi}_5^\theta(U)$  are  
 time  flows as \eqref{flussoG} with 
 $$G(\tau, U) =  \Opbw{\sm{ \im f(U;x, \xi)}{0}{0}{-\im {f(U;x, -\xi)}}} \ , \ \ \ f\in \wt \Gamma_2^\vr \ , \mbox{ real valued } , \ \ \vr <1 \, , 
 $$ 
 (precisely $f=\beta(U;x) |\xi|^{\frac12}$  or 
 $f =\beta(U;x) \sign(\xi)$ with  $\beta \in \wt \cF^\R_2$). \\
$\bullet$ 
  $\Upsilon_{fin}^\theta(U)$  is  defined in \cite[Proof of Proposition 4.4, pag 1462]{BFP} as the composition 
$$
\Upsilon_{fin}^\theta(U):= \Upsilon_{2\vr}^\theta(U)\circ \cdots \circ \Upsilon_1^\theta(U)
$$
where each $\Upsilon_{j+1}^\theta(U)$, $j=0,\ldots, 2\vr -1$ is  defined in  \cite[(4.68)]{BFP} as
$$
\Upsilon_{j+1}^\theta(U):= \cA_{j+1,3}^\theta(U)\circ  \cA_{j+1,2}^\theta(U)\circ  
\cA_{j+1,1}^\theta(U) \, .
$$
The maps  $ \cA_{j+1,1}^\theta(U)$ in 
\cite[(4.54), (4.55)]{BFP}, $ \cA_{j+1,2}^\theta(U)$  
in \cite[(4.58), (4.59)]{BFP}, and
 $ \cA_{j+1,2}^\theta(U)$
in \cite[(4.63), (4.65)]{BFP}
are  time flows 
 as \eqref{flussoG} with 
 $$
 G(U) = \Opbw{\sm{ \im f(U;x, \xi)}{0}{0}{-\im \bar{f(U;x, -\xi)}}}  \ , \ \ \ f\in \wt \Gamma_p^\rho \, ,  \ \ \rho <0 \ , \,  p=1,2  \, . 
 $$
$\bullet$  $\cB^\theta_i(U)$, $i=1,2$, are   defined 
 resp. in \cite[(5.23), (5.35)]{BFP}, and are 
 time  flows as \eqref{flussoG} with $G(\tau, U)$  matrices of smoothing operators in $\wt\cR^{-\varrho}_p \otimes \cM_2(\C)$ for some $\varrho \geq 0$, $p\in \{1,2\}$.

Applying  repeatedly Lemma \ref{lem:flow.ad}, there exist $s_0,r>0$ such that
for any $ \norm{U_i}_{\dot H^{s_0}} \leq r$, $i=1,2$, any $s \in \R$ and 
any  $ V \in \dot H^s (\T, \C^2)$,  one has
 \begin{align}\label{1701:1153}
& \sup_{\theta \in [0,1]} \norm{ \Phi_j^\theta(U_1) V-\Phi_j^\theta(U_2) V}_{\dot{H}^{s-1}} \leq C_{s,r}\norm{U_1-U_2}_{\dot{H}^{s_0}} \norm{V}_{\dot{H}^{s}}  \, , \qquad j =1,\ldots, 5 \, , \\ 
&
\sup_{\theta \in [0,1]} \norm{ \cA_{j+1,k}^\theta(U_1)V- \cA_{j+1,k}^\theta(U_2) V}_{\dot{H}^{s}} \leq C_{s,r} \norm{U_1-U_2}_{\dot{H}^{s_0}} \norm{V}_{\dot{H}^{s}}  \, , \quad j =0,
\ldots, 2\vr-1 \, , \ \ k =1,2,3 \, , \notag 
\end{align}
whereas by \Cref{flow.s.ad}, there is $s_0>0$ and, for any $s \geq s_0$, there is $r = r(s), C_{s,r}>0$ so that for 
any  $ \norm{U}_{\dot H^{s}} \leq r$, any $ V \in \dot H^s (\T, \C^2)$, 
\be\label{1701:1154}
 \sup_{\theta \in [0,1]}\norm{\cB^{\theta}_i(U_1)V -  \cB^{\theta}_i(U_2)V}_{\dot H^s} \leq C_{s,r}
\norm{U_1 - U_2}_{\dot H^{s}} \norm{V}_{\dot H^{s}} \, , \qquad i=1,2 \, .
 \ee
 Then $\mathfrak{C}^\theta(U)$ satsifies \eqref{diffC}  in view of \cref{flusso.est,1701:1153,1701:1154,flusso.estR}, and $\mathfrak{C}^\theta(0)  = \uno$ as well as each map in its definition.

The vector field $\cX_{\geq 4}(U,Z)$ is the one in \cite[formula (5.11)]{BFP} with the addition of the operator  $ B_{\geq 3}(U)$ in \eqref{Stimaenergy100}, coming  from the conjugation of the  quartic bounded map
$B(U)$ in \eqref{sistemaDiag}. Explicitly  
 \begin{align*}
B_{\geq 3}(U):= \wt{\mathfrak{C}}(U)\, B(U) \, \wt{\mathfrak{C}}(U)^{-1} 
\qquad \text{where} \qquad 
 \wt{\mathfrak{C}}(U):=\cB^1_2(U)\circ \cB_1^1(U)\circ\Upsilon_{fin}^1(U)\circ \vec{\Phi}_5^1(U)\circ \ldots \circ \vec{\Phi}_1^1(U) 
 \end{align*}
and it satisfies \eqref{estB3} by applying repeatedly \eqref{flusso.est}, \eqref{flusso.estR} and estimate \eqref{estB(U)}.
\end{proof}

\smallskip
\noindent{\bf Proof of Theorem \ref{BNFtheorem}.}
We define 
$\mathfrak{B}(u)u:= [\mathfrak{C}^1(U)U]^+$ the $ u$-component of the map $\mathfrak{C}^1(U)U$ where $\mathfrak{C}^1(U)$ is defined in \Cref{cor:BNF}.
Then $z =\mathfrak{B}(u)u $ solves  \eqref{theoBireq} in view of \eqref{sistemaBNF1000}.
Estimate \eqref{Germe} follows  from \eqref{stimafinaleFINFRAK} and \Cref{lemma6.3}.
It implies the equivalence of norms 
$C_s^{-1}\norm{z}_{\dot H^s} \leq 
\norm{u}_{\dot H^s}
\leq
C_s\norm{z}_{\dot H^s}
$ for some $C_s >0$.
The first estimate in 
 \eqref{phi_est} follows from the Lipschitz estimate \eqref{est.diff.Phi2} for $\mathfrak{C}^1(U)$ putting  $U_2 = 0$, and the second estimate follows from the first one writing 
 $\mathfrak{B}(U)^{-1} - \uno = - \mathfrak{B}(U)^{-1} \circ(\mathfrak{B}(U) - \uno)$. 
The vector field $\mathcal{X}_{\geq 4}(U,Y)$ is defined in  \eqref{Stimaenergy100}. 
Its first two terms satisfy 
 the energy estimate   \eqref{theoBirR}  by \cite[Lemma 6.4]{BFP}.
 Also the third term $B_{\geq 3}(U)Z$ satisfies the same estimate by 
 \eqref{estB3},  \eqref{BISMA2} and the equivalence of norms $\norm{z}_{\dot H^s} \sim \norm{u}_{\dot H^s}$.
Now we prove estimate 
\eqref{X_4_1}.
By \cite[Proposition 3.8]{BD}, estimate \eqref{piove} for $\mathfrak{R}_{\geq 3}(U)$ and \eqref{estB3}  we get 
$$
\norm{{\mathcal X_{\geq 4}^+(U,Z)}}_{\dot{H}^{1}} \lesssim \norm{U}_{K, s_0}^3  \norm{z}_{\dot H^{s_0}} 
\stackrel{\eqref{Germe}, \eqref{BISMA2}}{\lesssim } 
\norm{u}_{\dot H^{s_0}}^4 \, . 
$$
Estimate \eqref{lipB}  follows because $\mathfrak{C}^\theta(U)$ satisfies  \cref{diffC,stimafinaleFINFRAK}.

\footnotesize

\vspace{1em}
\noindent{\bf Acknowledgments:}
The authors 
thank E. Vanden-Eijnden and Yu Deng for interesting comments. 
A. Maspero is  supported by   the European Union  ERC CONSOLIDATOR GRANT 2023 GUnDHam, Project Number: 101124921 and
PRIN 2022 (2022HSSYPN).
 Views and opinions expressed are however those of the authors only and do not necessarily reflect those of the European Union or the European Research Council. Neither the European Union nor the granting authority can be held responsible for them.
 G. Staffilani is supported by the Simons Foundation through the Collaboration Grant on Wave Turbulence, and  the NSF through  grants DMS-2052651 and DMS-2306378. R.G. was supported by GNAMPA-INdAM.

\bibliographystyle{amsplain}
\providecommand{\href}[2]{#2}

\end{document}